\documentclass[11pt]{article}
\usepackage{natbib}
\usepackage{amssymb, amsmath, amsthm}
\usepackage{hyperref}
\hypersetup{colorlinks,citecolor=blue}
\usepackage{cleveref}
\usepackage{verbatim}
\usepackage{enumitem}
\usepackage{float}

\usepackage{amsfonts,mathtools,mathrsfs,mathtools,color}
\usepackage{graphicx}
\usepackage{mathpazo}

\usepackage{array}

\newtheorem{definition}{Definition}
\newtheorem{theorem}{Theorem}
\newtheorem{remark}{Remark}
\newtheorem*{theorem*}{Theorem}
\newtheorem{lemma}{Lemma}
\newtheorem{fact}{Fact}
\newtheorem{corollary}{Corollary}

\newtheorem{proposition}{Proposition}

\newtheorem{claim}{Claim}

\usepackage{bbm}

\DeclareMathOperator*{\argmin}{argmin}
\DeclareMathOperator*{\argmax}{argmax}

\DeclareMathOperator{\Span}{Span}

\newcommand{\AutoAdjust}[3]{{ \mathchoice{ \left #1 #2  \right #3}{#1 #2 #3}{#1 #2 #3}{#1 #2 #3} }}
\newcommand{\Xcomment}[1]{{}}

\newcommand{\inteval}[1]{\Big[#1\Big]}
\newcommand{\InParentheses}[1]{\AutoAdjust{(}{#1}{)}}

\newcommand{\InAngles}[1]{\AutoAdjust{\langle}{#1}{\rangle}}
\newcommand{\InNorms}[1]{\AutoAdjust{\|}{#1}{\|}}

\newcommand{\R}{\mathbb{R}} 
\newcommand{\N}{\mathcal{N}}
\renewcommand{\part}[2]{\frac{\partial #1}{\partial #2}}
\newcommand{\X}{\mathcal{X}}
\newcommand{\Y}{\mathcal{Y}}
\newcommand{\Z}{\mathcal{Z}}
\newcommand{\C}{\mathcal{C}}
\newcommand{\I}{\mathcal{I}}
\newcommand{\half}{\frac{1}{2}}

\newcommand{\dg}{\mathsf{dg}}

\newcommand{\hz}{\hat{z}}
\newcommand{\hZ}{\widehat{\Z}}
\newcommand{\hF}{\widehat{F}}

\newcommand{\ba}{\overline{a}}

\newcommand{\bz}{\overline{z}}
\newcommand{\bZ}{\overline{Z}}
\newcommand{\bF}{\overline{F}}

\newcommand{\boldx}{\boldsymbol{x}}

\newcommand{\ind}{\mathbbm{1}}
\newcommand{\ha}{\hat{a}}
\newcommand{\tar}{\textsc{Target}}

\newcommand{\gap}{\textsc{Gap}}
\newcommand{\sos}{\textsc{SOS}}
\newcommand{\LHSI}{\text{LHS of Inequality}}
\newcommand{\by}{\boldsymbol{y}}
\newcommand{\bw}{\boldsymbol{w}}

\newcommand{\LHSE}{\text{LHS of Equation}}

\usepackage{lscape} 

\makeatletter
\newcommand{\setword}[2]{%
  \phantomsection
  #1\def\@currentlabel{\unexpanded{#1}}\label{#2}%
}
\makeatother

\definecolor{MyGray}{rgb}{0.8,0.8,0.8}

\newenvironment{prevproof}[2]{\noindent {\bf {Proof of {#1}~\ref{#2}:}}}{$\hfill \blacksquare$}

\newcommand{\notshow}[1]{{}}
\renewcommand{\vspace}{\notshow}

\newenvironment{nalign}{
    \begin{equation}
    \begin{aligned}
}{
    \end{aligned}
    \end{equation}
    \ignorespacesafterend
}

\def \projhamnospace {tangent residual}
\def \projham {tangent residual }
\def \capprojham {Tangent Residual}
\def \Ham {r^{tan}}
\def \normal {N}
\def \unitnormal {\widehat{N}}
\title{Tight Last-Iterate Convergence of the Extragradient and the Optimistic Gradient Descent-Ascent Algorithm for Constrained Monotone Variational Inequalities}
\author{Yang Cai\thanks{Supported by a Sloan Foundation Research Fellowship and the NSF Award CCF-1942583 (CAREER).} \thanks{Part of this work was done while the author was visiting the Simons Institute for the Theory of Computing.} 
\\Yale University\\yang.cai@yale.edu
\and Argyris Oikonomou\footnotemark[1] \footnotemark[2]
\\Yale University\\argyris.oikonomou@yale.edu
\and Weiqiang Zheng\footnotemark[2] \\Yale University\\weiqiang.zheng@yale.edu}

\begin{document}

\maketitle

\begin{abstract}%
The \emph{monotone variational inequality} is a central problem in mathematical programming that unifies and generalizes many important settings such as smooth convex optimization, two-player zero-sum games, convex-concave saddle point problems, etc.
The \emph{extragradient algorithm} by \citet{korpelevich_extragradient_1976} and the \emph{optimistic gradient descent-ascent algorithm} by \citet{popov_modification_1980} are arguably the two most classical and popular methods for solving monotone variational inequalities. Despite their long histories, the following major problem remains open. \emph{What is the last-iterate convergence rate of the extragradient algorithm or the optimistic gradient descent-ascent} algorithm for monotone and Lipschitz variational inequalities with constraints? We resolve this open problem by showing that both the extragradient algorithm and the optimistic gradient descent-ascent algorithm have a tight $O\left(\frac{1}{\sqrt{T}}\right)$ last-iterate convergence rate for \emph{arbitrary convex feasible sets}, which matches the lower bound by~\cite{golowich_tight_2020,golowich_last_2020}. Our rate is measured in terms of the standard \emph{gap function}. At the core of our results lies a \emph{non-standard performance measure -- the tangent residual}, which can be viewed as an adaptation of the norm of the operator that takes the \emph{local} constraints into account. We use the tangent residual (or a slight variation of the tangent residual) as the the potential function in our analysis of the extragradient algorithm (or the optimistic gradient descent-ascent algorithm) and prove that it is non-increasing between two consecutive iterates.
\end{abstract}

\thispagestyle{empty}
\addtocounter{page}{-1}
\newpage
\tableofcontents 
\thispagestyle{empty}
\addtocounter{page}{-1}
\newpage
\section{Introduction}\label{sec:intro}
The \emph{monotone variational inequality (VI)} problem plays a crucial role in mathematical programming,  providing a unifying setting for the study of optimization and equilibrium problems. It also serves as a computational framework for numerous important applications in fields such as Economics, Engineering, and Finance~\citep{facchinei_finite-dimensional_2007}.  
Monotone VIs have been studied since the 1960s~\citep{hartman_non-linear_1966, browder_nonlinear_1965,lions_variational_1967,brezis_methodes_1968,sibony_methodes_1970}. Formally, a monotone VI is specified by a closed convex set $\Z \subseteq \R^n$ and a \textbf{monotone} operator $F: \Z \rightarrow \R^n$,\footnote{$F$ is monotone if $\InAngles{F(z)-F(z'),z-z'} \ge 0$ for all  $z, z' \in \Z$.} with the goal of finding a $z^* \in \Z$ such that 
\begin{align}
    \InAngles{F(z^*), z^*-z} \le 0 \quad \forall z \in \Z.
\end{align}
We further assume the operator $F$ to be \textbf{Lipschitz}, which is a natural assumption that is satisfied in most applications and is also made in the majority of algorithmic works concerning monotone VIs.
An important special case of the monotone and Lipschitz VI is the convex-concave saddle point problem:
\begin{equation}\label{eq:minmax}
    \min_{x\in \X} \max_{y\in \Y} f(x,y),
\end{equation}
where $\X$ and $\Y$ are closed convex sets in $\R^n$, and $f(\cdot,\cdot)$ is smooth, convex in $x$, and concave in $y$.\footnote{If we set $F(x,y) = \begin{pmatrix}
  \nabla_x f(x,y)\\
  -\nabla_y f(x,y)
\end{pmatrix}$ and $\mathcal{Z}=\mathcal{X}\times \Y$,
then (i) $F(x,y)$ is a monotone and Lipschitz operator, and (ii) the set of saddle points coincide with the solutions of the monotone VI for operator $F$ and domain $\Z$. } Besides its central importance in Game Theory, Convex Optimization, and Online Learning, the convex-concave saddle point problem has recently received a lot of attention from the machine learning community due to several novel applications such as the generative adversarial networks (GANS) (e.g.,~\citep{goodfellow_generative_2014,arjovsky_wasserstein_2017}), adversarial examples (e.g.,~\citep{madry_towards_2017}), robust optimization (e.g.,~\citep{ben-tal_robust_2009}), and reinforcement learning (e.g.,~\citep{du_stochastic_2017,dai_sbeed_2018}).

The extragradient (EG) algorithm by~\citet{korpelevich_extragradient_1976} and the optimistic gradient descent-ascent (OGDA) algorithm by~\citet{popov_modification_1980} are arguably the two most classical and popular methods for solving  Lipschitz and monotone VIs. Interestingly, a fundamental property of these two simple and natural algorithms remained elusive despite their long histories. Namely, the last-iterates of both algorithms are only known to asymptotically converge to a solution of the monotone and Lipschitz VI,\footnote{The last-iterate asymptotic convergence of EG can be found in \cite{korpelevich_extragradient_1976} and~\cite{facchinei_finite-dimensional_2007}, and the last-iterate asymptotic convergence of OGDA can be found in \cite{popov_modification_1980} and \cite{hsieh_convergence_2019}.} but no upper bounds on the rate of convergence had been provided for general settings. 
Motivated by this gap of our understanding, the following question has been posed as an open question in several recent works~\citep{wei_linear_2021,golowich_last_2020,golowich_tight_2020, hsieh_convergence_2019}. \\ 

\medskip\noindent
\hspace{.7cm}\begin{minipage}{0.95\textwidth}
\emph{\centering What is the last-iterate convergence rate of the extragradient algorithm
and the optimistic  gradient  descent-ascent algorithm for monotone and Lipschitz variational inequalities with constraints? }
\end{minipage}\\
\medskip\noindent

We resolve this open problem by providing the tight last-iterate convergence rate of EG and OGDA under arbitrary convex constraints.
Indeed, the same problem has not been answered even for two-player zero-sum games, arguably one of the most basic monotone and Lipschitz VIs.\footnote{A two-player zero-sum game can be specified by its payoff matrix $A\in \mathbb{R}^{\ell\times m}$. It is a special case of the convex-concave saddle point problem, where $\X=\Delta^\ell$ , $\Y=\Delta^m$ ($\Delta^k$ denotes the $k$-dimensional simplex), and the function $f(x,y)=x^\top A y$.} Prior to our work, the only setting where an upper bound on the rate of convergence exists for either EG or OGDA for solving general Lipschitz and monotone VIs is when the problem is unconstrained, i.e., $\Z=\mathbb{R}^n$,~\citep{golowich_last_2020,gorbunov_extragradient_2021, golowich_tight_2020}.\\


\medskip\noindent
\hspace{1.5cm}\begin{minipage}{0.9\textwidth}
\begin{enumerate}
\item[{\bf Main Result:}] For any monotone and Lipschitz variational inequality problem with an \textbf{arbitrary convex constraint set $\Z$}, both EG and OGDA with \textbf{constant step size} achieve a \textbf{tight} last-iterate convergence of $O\left(\frac{1}{\sqrt{T}}\right)$ in terms of the standard convergence measures -- the gap function (Definition~\ref{def:gap function}) and the \projham (Definition~\ref{def:projected hamiltonian}).\footnote{In the unconstrained setting, the \projham is simply the $\ell_2$-norm of $F$. In the constrained setting, the \projham is the $\ell_2$-norm of $F$'s projection to the tangent cone.} See Theorem~\ref{thm:EG last-iterate constrained} for the formal statement. We further show that the tangent residual is an upper bound of the natural residual (Lemma~\ref{lem:projected hamiltonian dominates natural residue}), so our result also implies a tight last-iterate convergence rate of $O(\frac{1}{\sqrt{T}})$ for the natural residual.
\end{enumerate}
\end{minipage}\\

\medskip
\noindent Our upper bounds in terms of the gap function and the natural residual match the lower bounds of~\cite{golowich_last_2020,golowich_tight_2020}, that is, they match in all of the following terms: $T$, the Lipschitz constant of $F$, and the distance between the starting point $z_0$ and the solution $z^*$.


To the best of our knowledge, our result is the first to provide a last-iterate convergence rate for solving monotone and Lipschitz VIs using any algorithm that belongs to the general class known as \emph{p-stationary canonical linear iterative algorithms (p-SCLI)}~\citep{arjevani_iteration_2016}, which contains the EG, OGDA, and  other well-known algorithms. Although often viewed as an approximation to EG, OGDA has an additional feature compared to EG, i.e., it is a \emph{no-regret} learning algorithm (see e.g., ~\citep{rakhlin_online_2013}). A nice implication of our result for OGDA is that for smooth and monotone games (see Appendix~\ref{sec:additional prelim} for the definition), players can each play a no-regret learning algorithm, i.e., OGDA, with constant learning rate, and the overall player behavior exhibits $O(\frac{1}{\sqrt{T}})$ last-iterate convergence rate to a Nash equilibrium in terms of the gap function. 

\paragraph{Why Last-Iterate Convergence?} Both the EG and OGDA algorithms are known to have \emph{average-iterate} convergence. In particular, the average of the iterates of the algorithm converges at a rate of $O(1/T)$~\citep{nemirovski_prox-method_2004, auslender_interior_2005, tseng_accelerated_2008,monteiro_complexity_2010, mokhtari_unified_2019,hsieh_convergence_2019}. Nonetheless, there are  several important reasons to study last-iterate convergence. First, not only is last-iterate convergence theoretically stronger and more appealing, it is also the only type of convergence that describes the trajectory of an algorithm. As demonstrated by~\citet{mertikopoulos_cycles_2018}, the trajectory of an algorithm may be cycling around in the space perpetually while still converges in the average-iterate sense. In game theory, we often view these algorithms as models of agents' behavior in a system/game. Thus, only last-iterate convergence provides a description of the evolution of the system.  
Additionally, EG and OGDA have been successfully applied to improve the training dynamics in GANs, as the training of GANs can be formulated as a saddle point problem~\citep{daskalakis_training_2017,yadav_stabilizing_2017,liang_interaction_2019,gidel_variational_2018,gidel_negative_2019,chavdarova_reducing_2019}. On the one hand, in this formulation of GANs, the objective function $f$ is usually non-convex and non-concave, making existing theoretical guarantees for the average iterate inapplicable. On the other hand, the last iterate typically has good performance in practice. Thus, it is crucial to develop machinery that allows us to analyze the behavior of the last iterate of these algorithms.

\subsection{Our Performance Measure: the \capprojham}\label{sec:intro performance measure}
A major challenge we face for establishing the last-iterate convergence for EG or OGDA in the constrained setting is the choice of the convergence measure. For simplicity, we focus on our choice of the performance measure for EG, as our performance measure for OGDA is  similar and inspired by our performance measure for EG. In the unconstrained case, the central performance measure for EG is the norm of the operator. The key component in both~\citep{golowich_last_2020} and~\citep{gorbunov_extragradient_2021} is to establish that the norm of the operator at the last iterate (also the $T$-th iterate) is upper bounded by $O(\frac{1}{\sqrt{T}})$, which implies a $O(\frac{1}{\sqrt{T}})$ last-iterate convergence  rate for the gap function.

In the constrained setting, the norm of the operator is a poor choice to measure convergence, as it can be far away from $0$ even in the limit, and is hence insufficient to guarantee convergence in terms of the gap function. A standard generalization of the norm of the operator in the constrained setting is the natural residual (Definition~\ref{def:natural residual}), which takes the constraints into account and is guaranteed to converge to $0$ in the limit.  Unfortunately, we observe that the natural residual is \emph{not monotonically decreasing} even in basic bilinear games (see Appendix~\ref{sec:counterexamples}), making it difficult to directly analyze. Similar non-monotonicity has been observed for several other natural performance measures such as the norm of the operator mapping introduced in \citep{diakonikolas_halpern_2020} and the gap function, leaving all these performance measures unsuitable. See more discussion about these performance measures in Section~\ref{sec:performance measure} and Appendix~\ref{sec:counterexamples}.

We choose a non-standard performance measure: the \textbf{\projhamnospace}, which can be viewed as the norm of the operator projected to the tangent cone of the current iterate (Definition~\ref{def:projected hamiltonian}). To the best of our knowledge, this performance measure has not been used in the study of EG or OGDA. The \projham plays a crucial role in our analyses for both algorithms. Unlike the aforementioned performance measures, we show that the \emph{\projham is monotonically decreasing}  and has \emph{a last-iterate convergence rate of ${O(\frac{1}{\sqrt{T}})}$} for EG. For OGDA, we prove that a small modification of the tangent residual is monotonically decreasing, which implies that the tangent residual has a last-iterate convergence rate of ${O(\frac{1}{\sqrt{T}})}$. Using the convergence rate of the \projhamnospace, we can easily derive the last-iterate convergence rate of other classical performance measures such as the natural residual or the gap function. However, we suspect these rates can be challenging to obtain directly. 

\subsection{Sum-of-Squares based Analysis}\label{sec:contribution}


We provide a quick overview on how we establish the monotonicity of the \projham~of the EG algorithm. We first introduce the concept of sum-of-squares programming.
\paragraph{Sum-of-Squares (SOS) Programming.} Suppose we want to prove that a polynomial $p(x)\in \mathbb{R}[x_1,\ldots, x_n]$ is non-negative over a semialgebraic set $\mathcal{S}=\{x\in \mathbb{R}^n: g_i(x)\leq 0, \forall  i\in[m] \}$, where each $g_i(x)$ is also a polynomial. One way is to construct a \emph{certificate of non-negativity}, for example, by providing a set of nonnegative coefficients $\{a_i\}_{i\in[m]}\in \mathbb{R}^m_{\geq 0}$ such that $p(x)+\sum_{i\in[m]} a_i\cdot g_i(x)$ is a sum-of-squares polynomial, that is, a polynomial that can be expressed as the sum of squares of further polynomials. Surprisingly, if $p(x)$ is indeed non-negative over $\mathcal{S}$, a certificate of non-negativity always exists  as guaranteed by a foundational result in real algebraic geometry -- the Krivine-Stengle Positivestellensatz~\citep{krivine_anneaux_1964,stengle_nullstellensatz_1974}, a generalization of Artin's resolution of Hilbert's 17th problem~\citep{artin_uber_1927}. 
Note that, it is sometimes necessary to allow more sophisticated forms of certificates than in the example above, e.g., replacing each coefficient $a_i$ with a SOS polynomial $s_i(x)$, etc. The complexity of a certificate is parametrized by the highest degree of the polynomial involved. 
The SOS programming consists of a hierarchy of algorithms, where the $d$-th hierarchy is an algorithm that searches for a \emph{certificate of non-negativity} up to degree $2d$ based on semidefinite programming. 

We mainly discuss the analysis of EG here, as the analysis of OGDA is similar and also based on SOS programming. At the core of our analysis of the EG algorithm lies the monotonicity of the squared \projhamnospace, which can be formulated as the non-negativity of a \textbf{degree-4 polynomial} in the iterates.\footnote{The \projham is not a polynomial, but the squared \projham is a degree-4 polynomial} 
Our original proof directly applies SOS programming to certify the non-negativity of this degree-4 polynomial. The certificate is rather complex and involves a polynomial identity of a degree-8 polynomial in $27$ variables, which we discover by solving a degree-8 SOS program. Interested readers can find the proof in Appendix~\ref{sec:high degree last-iterate EG}. In this version, we include a simplified proof. By introducing \emph{auxiliary vectors} that are not part of the update rule of EG, we provide an equivalent formulation of the squared tangent residual (Lemma~\ref{lemma:property of tangent residual}) that is a degree-2 polynomial, which allows us to prove the monotonicity of the squared tangent residual using a degree-2 SOS program. The proof can be found in Section~\ref{sec:last-iterate EG}.

For OGDA, we are not able to show that the squared tangent residual is monotone. Inspired by the adaptive potential proof in~\citep{golowich_tight_2020}, we suspect that some extra correction term is needed to construct the potential function. Instead of trying to devise such a correction term manually, we manage to directly find one by searching over a family of performance measures using SOS programming. The search we perform is heuristic but might be helpful to discover potential functions in other problems. See Section~\ref{sec:OGDA} for a more detailed discussion.



\vspace{-.1in}
\subsection{Related Work}
\paragraph{Last-Iterate Convergence Rate for EG-like Algorithms in the Unconstrained Setting.} \cite{golowich_last_2020,golowich_tight_2020} show a lower bound of $\Omega(\frac{1}{\sqrt{T}})$ for solving bilinear games using any p-SCLI algorithms, which include EG and OGDA. For EG, \citet{golowich_last_2020} show an matching upper bound under an additional second-order smoothness condition. \citet{gorbunov_extragradient_2021} improve the result and show that the same upper bound holds without the second-order smoothness condition. For OGDA,  \cite{golowich_tight_2020} provides a matching upper bound under the same second-order smoothness condition. 
These upper bounds hold for all smooth and Lipschitz VIs. With the additional assumption that the operator $F$ is cocoercive,~\cite{lin_finite-time_2021} show a $O(\frac{1}{\sqrt{T}})$ convergence rate for online gradient descent. If we further assume that either $F$ is strongly monotone in VI or the payoff matrix $A$ in a bilinear game has all singular values bounded away from $0$, linear convergence rate is known for EG, OGDA, and several of their variants~\citep{daskalakis_training_2017,gidel_variational_2018,liang_interaction_2019,mokhtari_unified_2019,peng_training_2020,zhang_convergence_2019}.

\paragraph{Last-Iterate Convergence Rate for EG-like Algorithms in the Constrained Setting.} The results for the constrained setting are sparser. If the operator $F$ is strongly monotone, we know that EG and some of its variants have linear convergence rate~\citep{tseng_linear_1995,malitsky_projected_2015}. Several papers establish the asymptotic convergence, i.e., converge in the limit, of the optimistic multiplicative weight updates in constrained convex-concave saddle point problems~\citep{daskalakis_last-iterate_2018,lei_last_2020}. Finally, a recent paper by~\cite{wei_linear_2021} provides a linear rate convergence of OGDA for bilinear games when the domain is a polytope. They show that there is a \emph{problem dependent} constant $0<c<1$ that depends on the payoff matrix of the game as well as the constraint set, so that the error shrinks by a $1-c$ factor. However, $c$ may be arbitrarily close to $0$,  even if we assume the corresponding operator to be $L$-Lipschitz. As a result, their convergence rate is slower than ours when $T$ is not comparable to $\frac{1}{c}$, which may be exponentially large in the dimension $n$, though their rate will eventually catch up. Overall, their ``instance-specific'' bound is incomparable and complements the worst-case view taken in this paper, where we want to derive the worst-case convergence rate for all VIs with monotone and $L$-Lipschitz operator $F$. Our result is the first last-iterate convergence rate in this worst-case view and matches the~lower~bound~by~\cite{golowich_last_2020,golowich_tight_2020}.

\paragraph{Other Algorithms and Performance Measures.}
Other than the gap function, one can also measure the convergence using the norm of the operator if the setting is unconstrained, or the natural residual (Definition~\ref{def:natural residual}) or similar notions if the setting is constrained. 
In the unconstrained setting,~\cite{kim_accelerated_2021},~\cite{yoon_accelerated_2021}, and~\cite{Lee_fast_21} provide algorithms that obtain $O(\frac{1}{T})$ convergence rate in terms of the norm of the operator, which is shown to be optimal by~\citet{yoon_accelerated_2021} for Lipschitz and monotone VIs. In the constrained setting,~\cite{diakonikolas_halpern_2020} shows the same $O(\frac{1}{T})$ convergence rate under the extra assumption that the operator is cocoercive and loses an additional logarithmic factor when the operator is only monotone. Our result implies a $O(\frac{1}{\sqrt{T}})$ last-iterate convergence rate in terms of the natural residual for both EG and OGDA. From an optimization point of view, i.e., the goal is to solve a Lipschitz and monotone VI, we should choose one of the above faster algorithms over EG or OGDA. However, one of our main motivation is game theoretic, that is, we would like to view simple algorithms such as OGDA and EG as models of agents' behavior and understand the speed for the overall behavior to converge to a Nash equilibrium. From this game-theoretic view point, we believe understanding the last-iterate convergence rate of simple algorithms such as EG and OGDA is an important question.

\paragraph{Computer-Aided Proofs.} A powerful computer-aided proof framework -- the \emph{performance estimation problem (PEP)} technique~(e.g., \citep{drori_performance_2014,taylor_performance_2017}) is widely applied to analyze first-order iterative methods. Indeed, the last-iterate convergence rate of EG in the unconstrained setting by~\cite{gorbunov_extragradient_2021} is obtained via the PEP technique. Although the PEP framework can handle projections \citep{taylor_exact_2017, pep-proj-1,pep-proj-2,pep-proj-3}, the main challenge for applying it to the constrained setting is that, 
the PEP framework requires the performance measures to be polynomials of degree $2$ or less~(see e.g., \citep{taylor_exact_2017}).\footnote{More specifically, the PEP framework requires the performance measure as well as the constraints to be linear in (i) the function values at the iterates and (ii) the Gram matrix of a set of vectors consisting of the iterates and their gradients.} In fact, solving the PEP is equivalent to solving a degree-2 SOS program, which can be viewed as the dual of the PEP~\citep{tan2021analysis}. In the unconstrained setting, the performance measure is a degree-2 polynomial -- the squared norm of the operator, and that is why one can either use the PEP (as in~\citep{gorbunov_extragradient_2021}) or a degree-2 SOS to certify its monotonicity (Theorem~\ref{thm:monotone hamiltonian at unconstrained}).
In the constrained setting, we use the squared \projhamnospace~ to measure the algorithm's progress, which in our original formulation is 
a degree 4 polynomial, 
making the PEP framework not directly applicable.\footnote{The \projham is the square root of a rational function and can only be even harder to handle.} As the SOS approach can accommodate polynomial objectives and constraints of any degree, we could  directly apply it to certify the monotonicity of the \projham in the constrained setting, although the resulting proof is complex. With the new formulation of the squared \projham (Lemma~\ref{lemma:property of tangent residual}), we manage to simplify our proof and derive it using a degree-2 SOS program. It is also not hard to see that one can apply the PEP framework on the new formulation of the squared tangent residual, and the resulting program is the dual program of our degree-2 SOS program. We believe an interesting future direction is to understand whether there are natural settings in optimization where degree-2 SOS programs are provably insufficient and higher degree SOS programs are necessary.

\citet{lessard_analysis_2016} analyze first-order iterative algorithms for convex optimization using a technique inspired by  the stability analysis from control theory.
They model first-order iterative algorithms using discrete-time dynamical systems and search over quadratic potential functions that satisfy a set of Integral Quadratic Constraints (IQC). 
\citet{zhang_unified_21} extend the IQC framework to study smooth and \emph{strongly} monotone VIs in the unconstrained setting. 

\notshow{
\argyrisnote{
\cite{lessard_analysis_2016,zhang_unified_21} analyze first-order iterative algorithms for (constrained) minimax games / convex optimization by using a technique inspired by control theory and stability analysis.
They model first-order iterative algorithms using discrete-time system and search over quadratic potential functions that satisfy a set of Integral Quadratic Constraints (IQC).
We believe that the IQC framework cannot handle the analysis of our results,
since the potential function we use - the squared tangent residual - is unclear how to be reformulated into a quadratic form over constraints that can be encoded into the IQC framework.
}

\weiqiangnote{\cite{zhang_unified_21} analyse several first-order methods for smooth and \emph{strongly} monotone VIs in the unconstrained setting using Integral Quadratic Inequalities (IQC). In the objective of determining the optimal exponential decay rate of parameterized quadratic potential functions,  they introduce sufficient quadratic constraints which can be turned into solving SDPs.}
}

\paragraph{SOS Programming and Analysis of Iterative Methods.} SOS programming has been employed in the design and analysis of algorithms in convex optimization. To the best of our knowledge, these results only concern minimization of smooth and strongly-convex functions in the unconstrained setting.
\citet{fazlyab2018design}  propose a framework to search the optimal parameters of the algorithm, e.g., step size.
They use SOS programming to search over quadratic potential functions and parameters of the algorithm  with the goal of optimizing the exponential decay rate of the potential function.
\citet{tan2021analysis} proposes to use SOS programming to study the convergence rates of first-order methods in unconstrained convex optimization. 

\paragraph{Simultaneous Result on Last-Iterate Convergence of OGDA in the Unconstrained Setting.} Shortly after we obtained the last-iterate convergence rate for OGDA in the constrained setting, we learned in early March, 2022 from private communication that Eduard Gorbunov, Gauthier Gidel, and Adrien Taylor had been working on the same problem. At the time of the communication, they could obtain the same last-iterate convergence rate for OGDA in the unconstrained case using a different method based on PEP.


\notshow{
Finally, there are some works that use SOS programming to analyse/design algorithms for convex optimization~\citep{tan2021analysis,fazlyab2018design}.
, but only in the unconstrained setting or strong convex setting. Besides, they all analyse or search over quadratic potential functions, which can also be handled by the PEP framework. 
To the best of our knowledge, we are the first to analyse a non quadratic potential function, demonstrating the flexibility and power of our SOS programming based approach.}

\vspace{-.1in}
\section{Preliminaries}\label{sec:VI}
We consider the Euclidean Space $(\R^n, \InNorms{\cdot})$, where $\InNorms{\cdot}$ is the $\ell_2$ norm and $\InAngles{\cdot,\cdot}$ denotes inner product on $\R^n$. We use $z[i]$ to denote the $i$-th coordinate of $z \in \R^n$ and $e_i$ to denote the unit vector such that $e_i[j] := \ind[i=j]$, the dimension of $e_i$ is going to be clear from context. 
\paragraph{Variational Inequality. }Given a closed convex set $\Z \subseteq \R^n$ and an operator $F: \Z \rightarrow \R^n$, a variational inequality problem is defined as follows: find $z^* \in \Z$ such that 
\vspace{-.1in}
\begin{align}
    \InAngles{F(z^*), z^*-z} \le 0 \quad \forall z \in \Z.
\end{align}

\vspace{-.1in}
We say $F$ is \textbf{monotone} if $\InAngles{F(z)-F(z'),z-z'} \ge 0$, for all $z, z' \in \Z$, and is $L$-Lipschitz if, $\|F(z) - F(z')\| \le L \|z-z'\|$ for all $z, z' \in \Z$.

\vspace{-.1in}
\begin{remark}
One sufficient condition for  such a $z^*$ to exist is when the set $\Z$ is bounded, but there are also other sufficient conditions that apply to unbounded $\Z$. See~\citep{facchinei_finite-dimensional_2007} for more details. Throughout this paper, we only consider monotone VIs that have a solution.
\end{remark}

\vspace{-.15in}
\begin{definition}[Gap Function]\label{def:gap function}
A standard way to measure the performance of $z \in \Z$  is by its gap function defined as $\gap_{\mathcal{Z},F,D}(z) = \max_{z'\in \Z\cap \mathcal{B}(z,D)} \InAngles{F(z),z-z'}$, where $D>0$ is a fixed parameter and $\mathcal{B}(z,D)$ is a ball with radius $D$ centered at $z$.\footnote{Sometimes the gap function is defined to allow $z'$ to take value in $\Z\cap \mathcal{B}(z^*,\InNorms{z^*-z_0})$, where $z_0$ is the starting point of the EG algorithm, and  $z^*$ is the solution that the last iterate of the algorithm converges to. Due to Lemma~\ref{lem:EG Best-iterate}, $\InNorms{z_k-z^*}\leq \InNorms{z_0-z^*}$ for every $k$, so $\mathcal{B}(z_k,2\InNorms{z^*-z_0})$ contains $\mathcal{B}(z^*,\InNorms{z^*-z_0})$.} When $\Z, F$ and $D$ are clear from context, we omit the subscripts and write the gap function at $z$ as $\gap(z)$.
\end{definition}
\vspace{-.1in}

\paragraph{The Extragradient Algorithm.} Let 
$z_k$ be the $k$-th iterate of the Extragradient (EG) Algorithm. The update rule of EG is as follows:
\vspace{-.1in}
\begin{align}
     z_{k+\frac{1}{2}} &= \Pi_\Z\left[z_k-\eta F(z_k)\right]=\arg \min_{z\in \Z} \| z-\left(z_k-\eta F(z_k)\right)\|,\label{def:k+1/2-th step}\\
     z_{k+1} & = \Pi_\Z\left[z_k-\eta F(z_{k+\frac{1}{2}})\right]= \arg \min_{z\in \Z} \left \| z-\left(z_k-\eta F(z_{k+\frac{1}{2}})\right)\right \|. \label{def:k+1-th step}
 \end{align}
 
\paragraph{The Optimistic Gradient Descent-Ascent Algorithm.} Let $z_k$ and $w_k$ be the $k$-th iterate of the Optimistic Gradient Descent Ascent Method (OGDA) method. Let $z_0, w_0$ be arbitrary starting points in $\Z$. The update rule is as follows:
 \begin{align}\label{alg:OGDA}
     w_{k+1} &= \Pi_\Z\left[z_k-\eta F(w_k)\right]=\arg \min_{z\in \Z} \| z-\left(z_k-\eta F(w_k)\right)\|\\
     z_{k+1} & = \Pi_\Z\left[z_k-\eta F(w_{k+1})\right]= \arg \min_{z\in \Z} \left \| z-\left(z_k-\eta F(w_{k+1})\right)\right \|
\end{align}

Note that the OGDA method only requires $T$ queries to the operator at $\{w_k\}_{0\le k \le T-1}$, while EG requires $2T$ queries to the operator. Additionally, OGDA is a more natural algorithm in multi-agent online learning settings~\citep{cesa-bianchi_prediction_2006,shalev2012online}, as players play according to the strategy profile $w_k$ and receive gradient feedback $F(w_k)$ to compute $z_k$ and $w_{k+1}$, while EG requires players to play every half step $z_{k}$ and $z_{k+\frac{1}{2}}$ to get gradient feedback. Finally, as we mentioned before, OGDA is a no-regret algorithm while EG is not. 

In Section~\ref{sec:best iterate} and \ref{sec:EG last iterate everything}, we present the analysis of the EG algorithm and provide a detailed description about how to use SOS programming to derive the proof.
The analysis of the OGDA algorithm is a simple extension of our analysis to the EG algorithm.
We formally state the results of OGDA in Section~\ref{sec:OGDA} and postpone the detailed analysis of OGDA in Section~\ref{appx:OGDA}.

\vspace{-.3in}

\paragraph{Sum-of-Squares (SOS) Polynomials.} Let $\boldsymbol{x}$ be a set of variables. We denote the set of real polynomials in $\boldsymbol{x}$ as $\R[\boldsymbol{x}]$. We say that polynomial $p(\boldsymbol{x}) \in \R[\boldsymbol{x}]$ is an SOS polynomial
if there exist polynomials $\{q_i(\boldsymbol{x})\in \R[\boldx] \}_{i\in[M]}$ such that $p(\boldsymbol{x})=\sum_{i\in[M]}q_i(\boldsymbol{x})^2$. We denote the set of SOS polynomials in $\boldx$ as $\sos[\boldx]$. Note that any SOS polynomial is non-negative.
 
\paragraph{SOS Programs.} In Figure~\ref{fig:sos template} we present a generic formulation of a degree-$d$ SOS program.
The SOS program takes three kinds of input,
a polynomial $g(\boldsymbol{x})$,
sets of polynomials $\{g_i(\boldsymbol{x})\}_{i\in[M]}$ and  $\{h_i(\boldsymbol{x})\}_{i\in[N]}$.
Each polynomial in $\{g(\boldsymbol{x})\}\cup\{g_i(\boldsymbol{x})\}_{i\in[M]}\cup\{h_i(\boldsymbol{x})\}_{i\in[N]}$ has degree of at most $d$.
The SOS program searches for an SOS polynomial in the set of polynomials $\Sigma = \{g(\boldsymbol{x}) + \sum_{i\in[M]}p_i(\boldsymbol{x})\cdot g_i(\boldsymbol{x}) + \sum_{i\in[N]}q_i(\boldsymbol{x})\cdot h_i(\boldsymbol{x})\}$,
where $\{p_i(\boldsymbol{x})\}_{i\in[M]}$ and $\{q_i(\boldsymbol{x})\}_{i\in[N]}$ are polynomials in $\boldsymbol{x}$. 
More precisely for each $i\in[M]$, $p_i(\boldsymbol{x})$ is an SOS polynomial with degree at most $d-\deg(g_i(\boldsymbol{x}))$. For each $i\in[N]$,
$q_i(\boldsymbol{x})$ is a (not necessarily SOS) polynomial with degree at most $d-\deg(g_i(\boldsymbol{x}))$.
Note that any polynomial in set $\Sigma$ is at most degree $d$. In our applications, we choose $\{g_i(\boldsymbol{x})\}_{i\in[M]}$ to be non-positive polynomials and $\{h_i(\boldsymbol{x})\}_{i\in[N]}$ to be polynomials that are equal to $0$. Any feasible solution to the program certifies the non-negativity of $g(\boldsymbol{x})$.

\begin{figure}[h!]
\colorbox{MyGray}{
\begin{minipage}{\textwidth} {
\small
\noindent\textbf{Input Fixed Polynomials.} 
\begin{itemize}
\itemsep0em 
\item Polynomial $g(\boldsymbol{x})$
\item Polynomial $g_i(\boldsymbol{x})\in \R[\boldsymbol{x}]$ for all $i\in[M]$.
\item Polynomial $h_i(\boldsymbol{x})\in \R[\boldsymbol{x}]$ for all $i\in[N]$.
\end{itemize}
\noindent\textbf{Decision Variables of the SOS Program:}
\begin{itemize}
\itemsep0em 
\item $p_{i}(\boldsymbol{x})\in \sos[\boldsymbol{x}]$ is an SOS polynomial with degree at most $d-\deg\InParentheses{g_i}$, for all $i\in[M]$.
\item $q_{i}(\boldsymbol{x})\in \R[\boldsymbol{x}]$ is a polynomial with degree at most $d-\deg\InParentheses{h_i}$ , for all $i\in[N]$.
\end{itemize}
\textbf{Constraints of the SOS Program:}
\begin{align*}
g(\boldsymbol{x}) + \sum_{i\in [M]}p_i(\boldsymbol{x})\cdot g_i(\boldsymbol{x})
+ \sum_{i\in [N]}q_i(\boldsymbol{x})\cdot h_i(\boldsymbol{x})\in \sos[\boldsymbol{x}]
\end{align*}
}
\vspace{-.15in}
\end{minipage}} 
\caption{Generic degree $d$ SOS program.}\label{fig:sos template}
\vspace{-.1in}
\end{figure}

\paragraph{Roadmap of the Paper.} In \Cref{sec:performance measure}, we introduce our new performance measure -- the tangent residual and prove some of its properties.
In \Cref{sec:best iterate}, we show that the EG algorithm enjoys best-iterate convergence.
In \Cref{sec:EG last iterate everything}, we strengthen the convergence guarantee for the EG algorithm and obtain the tight last-iterate convergence rate by showing that the tangent residual (\Cref{def:projected hamiltonian}) is non-increasing across the iterations of the EG algorithm.
The last-iterate convergence rate for the tangent residual also implies a last-iterate convergence rate for the gap function (\Cref{def:gap function}) and the natural residual (\cref{def:natural residual}) as shown in \Cref{lem:hamiltonian to gap} and \Cref{lem:projected hamiltonian dominates natural residue}.
In \Cref{sec:OGDA}, we further prove the tight last-iterate convergence rate for the OGDA algorithm.
The analysis of the OGDA algorithm follows the same steps as in the analysis of the EG algorithm, and we postpone most of the details in Appendix~\ref{appx:OGDA}.

\section{The \capprojham~and Its Properties}\label{sec:performance measure}

We formally introduce our performance measure the \projham. As discussed in Section~\ref{sec:intro performance measure}, many standard and natural performance measures, i.e., the natural residual, $\InNorms{z_k-z_{k+1/2}}$,\footnote{$\InNorms{z_k-z_{k+1/2}}$ is proportional to the norm of the operator mapping introduced in~\citep{diakonikolas_halpern_2020}.} $\InNorms{z_{k}-z_{k+1}}$, $\max_{z\in \Z} \InAngles{F(z), z_k-z}$ and $\max_{z\in \Z} \InAngles{F(z_k), z_k-z}$, are unfortunately non-decreasing for EG. See Appendix~\ref{sec:counterexamples} for numerical examples.

\vspace{-.1in}
\begin{definition}[Unit Normal Cone]\label{def:normal cone}
Given a closed convex set $\Z\subseteq \R^n$ and a point $z\in \Z$,
we denote by $\normal_\Z(z)=\InAngles{v\in \R^n:\InAngles{v,z'-z}\leq 0, \forall z'\in \Z}$ the normal cone of $\Z$ at point $z$
and by $\unitnormal_\Z(z)=\{v\in \normal_\Z (z):\InNorms{v}\leq 1\}$ the intersection of the unit ball with the the normal cone of $\Z$ at $z$.
Note that $\unitnormal_{\mathcal{Z}}(z)$ is nonempty and compact for any $z\in \Z$, as $(0,\ldots,0)\in \unitnormal_{\Z}(z)$.
\end{definition}


\vspace{-.2in}
\begin{definition}[\capprojham]\label{def:projected hamiltonian}
Given an operator $F:\Z\rightarrow \R^n$ and a closed convex set $\Z$, let  $T_\Z(z) := \{ z'\in \R^n: \InAngles{z',a} \le 0, \forall a \in \C_{\Z}(z)\}$ be the tangent cone of $z$, and define $J_\Z(z) := \{z\} + T_\Z(z)$.
The \projham of $F$ at $z\in \Z$ is defined
as $r^{tan}_{(F,\Z)}(z):=\InNorms{\Pi_{J_\Z(z)}[z - F(z)] - z}.$ An equivalent definition is 
$r^{tan}_{(F,\Z)}(z):=\sqrt{
\|F(z)\|^2 - \max_{\substack{a\in \unitnormal_\Z(z),\\ \InAngles{F(z),a}\le 0}}\langle a, F(z) \rangle ^2}$.

\end{definition}
\vspace{-.1in}
\begin{remark}
 We show the equivalence of the two definitions of \projham in Lemma~\ref{lem:equivalent hamiltonian}. For the rest of the paper, we may use either of the two equivalent definitions depending on which one is more convenient.

\end{remark}

When the convex set $\Z$ and the operator $F$ are clear from context,
we are going to omit the subscript and denote 
the unit normal cone as $\unitnormal(z)=\unitnormal_\Z(z)$ and the \projham as $\Ham(z) = r^{tan}_{(F,\Z)}(z)$. Although the definition is slightly technical, one can think of the \projham as the norm of another operator $\widehat{F}$, which is $F$ projected to all directions that are not ``blocked'' by the boundary of $\Z$ if one takes an infinitesimally small step $\epsilon\cdot F(z)$, which is the same as projecting $F$ to $J_\Z(z)$. Intuitively, if the \projham is small, then the next iterate will not be far away from the current one. 



Next, we formally define the natural residual associated with the instance formally stated in Definition~\ref{def:natural residual}, and show how it is related to the \projhamnospace.

\begin{definition}\label{def:natural residual}
Consider an instance $\I$ of the variational inequality problem on convex set $\Z\subseteq \R^n$ and monotone operator $F:\Z\rightarrow \R^n$.
For $z\in \Z$, 
the natural map and natural residual associated with $\I$ is defined as follows
$$
F^{nat}_K(z) = z - \Pi_\Z(z- F(z)), \qquad r^{nat}_{(F,\Z)}(z) = \InNorms{F^{nat}_K(z)}.
$$
\end{definition}
Given an instance of the monotone VI constrained on convex set $\Z\subseteq\R^n$ and operator $F:\Z\rightarrow\R^n$,
point $z^*$ is a solution of the monotone VI iff $r_{(F,\Z)}^{nat}(z^*)=0$.
In Lemma~\ref{lem:projected hamiltonian dominates natural residue}, we show that the \projham upper bounds the the natural residual. See \Cref{fig:comparision between PH and NR} for illustration of how the \projham~relates to the natural residual.
\begin{lemma}\label{lem:projected hamiltonian dominates natural residue}
Consider an instance $\I$ of the variational inequality problem on convex set $\Z\subseteq \R^n$ and monotone operator $F:\Z\rightarrow \R^n$.
For any $z\in \Z$,  $r_{(F,\Z)}^{tan}\geq r_{(F,\Z)}^{nat}(z)$.
\end{lemma}

\notshow{
\begin{proof}
According to Lemma~\ref{lem:equivalent hamiltonian-2}, $H_{(F,\Z)}(z) = \InNorms{\Pi_{J_\Z(z)}\inteval{z - F(z)} - z}^2$, where $J_\Z(z) := z + T_\Z(z)$ and $T_\Z(z) = \{ z'\in \R^n: \InAngles{z',a} \le 0, \forall a \in \C_{\Z}(z)\}$ is the tangent cone of $z$. 
Observe that $z - \Pi_{\Z}\inteval{z - F(z)} = -\Pi_{\Z-z}\inteval{- F(z)}$ and similarly 
$z - \Pi_{J_\Z(z)}\inteval{z - F(z)} = -\Pi_{J_\Z(z)-z}\inteval{- F(z)}$.
Since $J_\Z(z)\supseteq \Z$, then $J_\Z(z)-z\supseteq \Z-z$ which further implies that
$$
\InNorms{-\Pi_{j_\Z(z)-z}\inteval{- F(z)}}
\geq 
\InNorms{-\Pi_{\Z-z}\inteval{- F(z)}}
\Leftrightarrow
\InNorms{z - \Pi_{J_\Z(z)}\inteval{z - F(z)}}^2
\geq 
\InNorms{z - \Pi_{\Z}\inteval{z - F(z)}}^2,
$$
which concludes the proof.
\end{proof}
}

\begin{proof}

\notshow{
Fix $z\in \Z$ and consider convex set $\bZ= \{\bz\in \R^n: \bz - z \in \Z\}$ and monotone operator $\bF(\bz)=F(\bz - z): \bZ\rightarrow \R^n$.
Observe that $H_{(F,\Z)}(z)=H_{(\bF,\bZ)}(0)$ and $r_{(F,\Z)}^{nat}(z)=r_{(\bF,\bZ)}^{nat}(0)$.

Let $w = \Pi_\Z (-F(z))$, 
$a_1=- \bF (0)-w\in \C_{\bZ}(w)$, $a_2=\argmax_{\substack{a\in \unitnormal_{\bZ}(0),\\ \InAngles{\bF (0),a}\le 0}}\langle a, \bF (0) \rangle ^2$ 
and consider convex set $\bZ_1 = \{\bz\in \R^n: \InAngles{a_1,\bz -w} \leq 0 \land \InAngles{a_2,\bz} \leq 0\}$, $ \bZ_2 = \{\bz\in \R^n: \InAngles{a_2,\bz } \leq 0\}$.
Observe that $\bF^{nat}_{\bZ}(0)=w = \Pi_{\Z_1} (-F(z))$ and
$\Z_2 \supseteq \Z_1 \supseteq \Z$,
which implies that $r_{(\bF,\bZ)}^{nat}(0)=\InNorms{\Pi_{\Z_1} (-F(z))}\leq \InNorms{ \Pi_{\Z_2} (-F(z))}$.
We conclude the proof by noticing that $\InNorms{ \Pi_{\Z_2} (-F(z))}^2= H_{(\bF,\bZ)}(0)$.
}

Let $w=\Pi_\Z(z-F(z))$ and $a_1=z-F(z)-w$. Observe that $$\InNorms{F(z)}^2=\InNorms{z-w}^2+\InNorms{a_1}^2-2\InAngles{z-w,a_1}.$$ Since $r^{nat}_{(F,\Z)}(z)^2=\InNorms{z-w}^2$ and $\InAngles{z-w,a_1}\leq 0$, we have $r^{nat}_{(F,\Z)}(z)^2\leq \InNorms{F(z)}^2-\InNorms{a_1}^2$. 

According to Lemma~\ref{lem:equivalent hamiltonian}, $\Ham(z)^2 = \InNorms{\Pi_{J_\Z(z)}\inteval{z - F(z)} - z}^2$, where $J_\Z(z) := z + T_\Z(z)$ and $T_\Z(z) = \{ z'\in \R^n: \InAngles{z',a} \le 0, \forall a \in \C_{\Z}(z)\}$ is the tangent cone of $z$. Since $J_Z(z)$ is a cone with origin $z$, we have $\displaystyle\InAngles{z-F(z)-\Pi_{J_\Z(z)}\inteval{z - F(z)}, z-\Pi_{J_\Z(z)}\inteval{z - F(z)}}=0$, and $\displaystyle\InNorms{\Pi_{J_\Z(z)}\inteval{z - F(z)} - z}^2=\InNorms{F(z)}^2-\InNorms{z-F(z)-\Pi_{J_\Z(z)}\inteval{z - F(z)}}^2$. As $\Z\subseteq J_\Z(z)$, $\displaystyle\InNorms{a_1}^2\geq \InNorms{z-F(z)-\Pi_{J_\Z(z)}\inteval{z - F(z)}}^2$, which implies that $$r^{nat}_{(F,\Z)}(z)^2\leq \InNorms{\Pi_{J_\Z(z)}\inteval{z - F(z)} - z}^2=\Ham(z)^2.$$


\notshow{
Let $w = \Pi_\bZ (-\bF(0))$, 
$a_1=-\bF(\bz)-w\in \C_\bZ(w)$, 
$a_2=\argmax_{\substack{a\in \unitnormal_\bZ(0),\\ \InAngles{\bF(0),a}\le 0}}\langle a, \bF(\bz) \rangle ^2$
and consider convex set $\bZ_1 = \{\bz\in \R^n: \InAngles{a_1,\bz -w} \leq 0 \land \InAngles{a_2,\bz } \leq 0\}$, $ \bZ_2 = \{\bz\in \R^n: \InAngles{a_1,\bz -w} \leq 0\}$

Let $w = \Pi_\Z (z-F(z))$, $a_1=z-F(z)-w\in \C_\Z(w)$, $a_2=\argmax_{\substack{a\in N_\Z(z),\\ \InAngles{F(z),a}\le 0}}\langle a, F(z) \rangle ^2$ and consider convex set $ \Z_1 = \{\bz\in \R^n: \InAngles{a_1,\bz -w} \leq 0 \land \InAngles{a_2,\bz -z} \leq 0\}$, $ \Z_2 = \{\bz\in \R^n: \InAngles{a_1,\bz -w} \leq 0\}$.
Observe that $w = \Pi_{\Z_1} (z-F(z))= \Pi_{\Z_2} (z-F(z))$ and
$\Z_2 \supseteq \Z_1 \supseteq \Z$.

We conclude the proof by noting that $r_{(F,\Z)}^{nat}(z)=r_{(F,\bZ)}^{nat}(z)$ and $H_{(F,\Z)}(z)=H_{(F,\bZ)}(z)$.
}
\end{proof}

\begin{figure}[ht]
\centering
\includegraphics[width=0.5\textwidth]{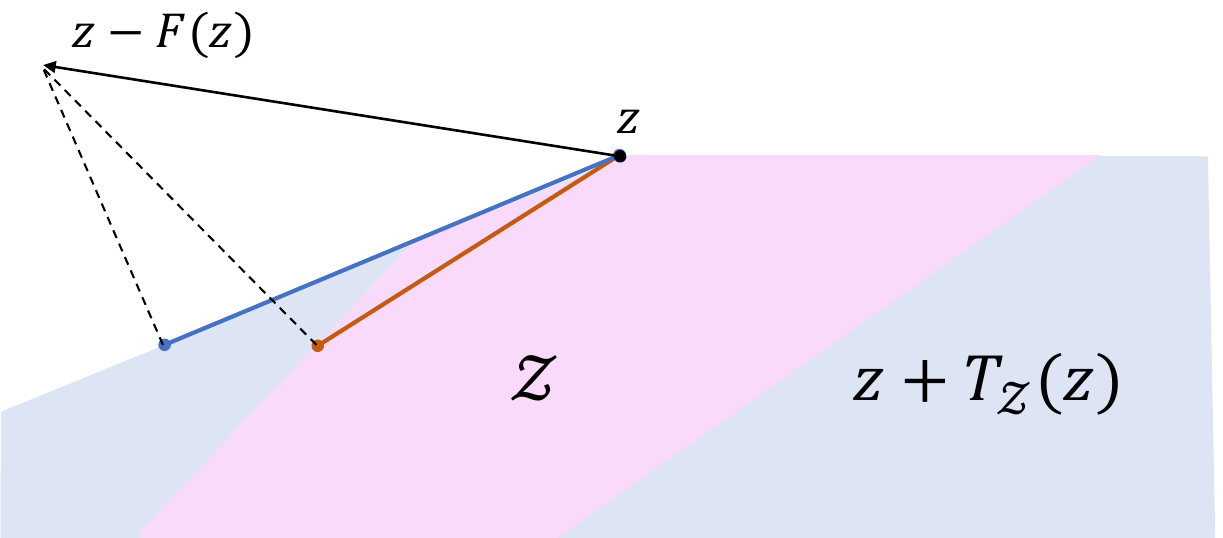}
\caption{Illustration of the \projham and the natural residual. The blue line represents the tangent residual and the red line represents the natural residual. It is clear that the tangent residual upper bounds the natural residual.}
\label{fig:comparision between PH and NR}
\end{figure}

Due to the above lemma, an upper bound of the \projham is also an upper bound of the natural residual. 
We show in \Cref{thm:monotone hamiltonian} the monotonicity of the \projham of the EG updates, which is the technical core of our analysis and implies the $O(\frac{1}{\sqrt{T}})$ convergence rate of the \projhamnospace. 
As a result, we also show that the natural residual has a $O(\frac{1}{\sqrt{T}})$ convergence rate. One may be tempted to directly use the natural residual as the convergence measure. However, from our  numerical experiments, the natural residual of the EG updates is not monotone, and we believe that it is very challenging to directly establish the convergence rate for the natural residue without using the \projham as a proxy.

In the next lemma, we argue why a small \projham implies a small gap function, hence an approximate solution of the variational inequality. The proof is postponed to Appendix~\ref{sec:additional prelim}.
\begin{lemma}\label{lem:hamiltonian to gap}[Adapted from the proof of Theorem~10 in \citep{golowich_last_2020}.]
Given a closed convex set $\Z \in \R^n$, an operator $F : \Z \rightarrow \R^n$ and $z \in \Z$, we have 
\[
\gap_{\mathcal{Z},F,D}(z) := \max_{z'\in \Z\cap \mathcal{B}(z,D)} \InAngles{F(z),z-z'}  \le D  \cdot r^{tan}_{(F,\Z)}(z).
\]
\vspace{-.02in}
\noindent If we have a convex-concave function $f(z):\Z \rightarrow \R$ such that $z = (x,y)$, $\Z = \X\times \Y$ where $\X$ and $\Y$ are closed convex sets, 
let $F(x,y) = \begin{pmatrix}
  \nabla_x f(x,y)\\
  -\nabla_y f(x,y)
\end{pmatrix}$,
then the duality gap at $z$ with respect to $\X'$ and $\Y'$ is
$\dg^{\X',\Y'}_f(z)
    := \max_{y'\in \Y'} f(x,y')- \min_{x'\in \X'} f(x',y) 
    \le D \sqrt{2} \cdot r^{tan}_{(F,\Z)}(z)$, if $z\in \X'\times \Y'$ and the diameters of $\X'$ and $\Y'$ are both upper bounded by $D$.\footnote{When $\X$ and $\Y$ are 
    bounded, we choose $\X'=\X$ and $\Y'=Y$, otherwise the convention is to choose $\X'$ and $\Y'$ to be $\X\cap \mathcal{B}(x^*,2\InNorms{x_0-x^*})$ and $\Y\cap \mathcal{B}(y^*,2\InNorms{y_0-y^*})$ respectively, where $(x^*,y^*)$ is a saddle point.}
\end{lemma}

 \vspace{-.1in}
 \section{Best-Iterate Convergence of EG with Constant Step Size}\label{sec:best iterate}
\vspace{-.05in}
En route to establish the last-iterate convergence of the EG algorithm, we first show a weaker guarantee known as the best-iterate convergence. Lemma~\ref{lem:EG Best-iterate} implies that after running EG for $T$ steps,
there exists an iteration $t^*\in[T]$ where $\InNorms{z_{t^*}-z_{t^*+\half}}^2 \leq O(\frac{1}{T})$. The proof can be found in~\citep{korpelevich_extragradient_1976} and ~\citep{facchinei_finite-dimensional_2007} (included in \Cref{appx:best iterate} for completeness).
\vspace{-.1in}
 \begin{lemma}[\citep{korpelevich_extragradient_1976,facchinei_finite-dimensional_2007}]\label{lem:EG Best-iterate}
 Let $\Z$ be a closed convex set in $\mathbb{R}^n$, $F(\cdot)$ be a monotone and $L$-Lipschitz operator mapping from $\Z$ to $\mathbb{R}^n$. For any solution $z^*$ of the monotone VI, that is, $\langle F(z^*), z^*-z\rangle \leq 0$ for all $z\in \Z$. For all $k$,
\vspace{-.1in}
 \begin{align}\label{eq:best-iterate-2}
     \left\|z_k-z^*\right\|^2\geq  \left\|z_{k+1}-z^*\right\|^2+(1-\eta^2 L^2) \|z_k-z_{k+\half}\|^2.
 \end{align}
 \end{lemma}
\notshow{
\begin{claim}\label{fact:close}
If $\eta \cdot L < 1$, then for all $k$
\begin{eqnarray*}
||z_{k+1/2} - z_{k+1}||\leq& \eta L ||z_k-z_{k+1/2}|| \\
||z_k-z_{k+1/2}||\leq& \frac{1}{1-\eta L}||z_k-z_{k+1}||\\
||z_{k+1/2} - z_{k+1}||\leq& \frac{\eta L}{1-\eta L}||z_k-z_{k+1}|| 
\end{eqnarray*}

\end{claim}
\begin{proof}
We remind that $z_{k+\frac{1}{2}} = \Pi_\Z\left[z_k-\eta F(z_k)\right]$ and $z_{k+1} = \Pi_\Z\left[z_k-\eta F(z_{k+\frac{1}{2}})\right]$.
By non-expansiveness of the projection operator and the $L$-Lipschitzness of function $F$,
we have that $||z_{k+1/2}-z_{k+1}||\leq ||\eta (F(z_{k+1/2})- F(z_k))||\leq \eta L ||z_k - z_{k+1/2}||$.
By triangle inequality we have:
$$
||z_k-z_{k+1/2}|| \leq ||z_k - z_{k+1}|| + ||z_{k+1/2}-z_{k+1}|| \leq ||z_k - z_{k+1}|| + \eta L||z_k-z_{k+1/2}||
$$
which implies that $||z_k-z_{k+1/2}||\leq \frac{1}{1-\eta L}||z_k-z_{k+1}||$.
The second part of the statement follows by $||z_{k+1/2}-z_{k+1}||\leq \eta L ||z_k - z_{k+1/2}||$.
\end{proof}
}

\vspace{-.15in}
In Lemma~\ref{lem:bound hamiltonian by distance}, we relate $\InNorms{z_k-z_{k+\half}}$ with the \projham at $z_{k+1}$, and derive the best-iterate convergence guarantee in terms of the \projham in Lemma~\ref{lem:best iterate hamiltonian}. 
The proofs of Lemma~\ref{lem:bound hamiltonian by distance} and~\ref{lem:best iterate hamiltonian} are postponed to~\Cref{appx:best iterate}.


\vspace{-.15in}
\begin{lemma}\label{lem:bound hamiltonian by distance}
For all $k$, $\Ham(z_{k+1}) \leq \left(1+\eta L + (\eta L )^2 \right)\frac{||z_k-z_{k+1/2}||}{\eta}$.
\end{lemma}
\vspace{-.05in}

In lemma~\ref{lem:best iterate hamiltonian}, we argue the \projham has a best-iterate convergence with rate $O(\frac{1}{\sqrt{T}})$.
\vspace{-.15in}
\begin{lemma}\label{lem:best iterate hamiltonian}
Let $\Z$ be a closed convex set in $\mathbb{R}^n$, $F(\cdot)$ be a monotone and $L$-Lipschitz operator mapping from $\Z$ to $\mathbb{R}^n$. Suppose the step size of the EG algorithm $\eta\in(0,\frac{1}{L})$, then for any solution $z^*$ of the monotone VI
and any integer $T>0$,
there exists $t^*\in [T]$ such that:
$$
\InNorms{z_{t^*}-z_{t^*+\half}}^2
\leq 
\frac{1}{T}\frac{\InNorms{z_0-z^*}^2}{1-(\eta L)^2},
\qquad \textsc{ and }\qquad
\Ham(z_{t^*+1})
\leq
\frac{1+\eta L + (\eta L )^2}{\eta}
\frac{1}{\sqrt{T}}\frac{\InNorms{z_0-z^*}}{\sqrt{1-(\eta L)^2}}.
$$
\end{lemma}

\section{Last-Iterate Convergence of EG with Constant Step Size}
\label{sec:EG last iterate everything}
In this section, we show that the last-iterate convergence rate is $O(\frac{1}{\sqrt{T}})$. In particular, we  prove that the \projham is non-increasing, which, in combination with Lemma~\ref{lem:best iterate hamiltonian}, implies the last-iterate convergence rate of EG. To establish the monotonicity of the \projhamnospace, we combine SOS programming with the low-dimensionality of the EG update rule. To better illustrate our approach, we first prove the result in the unconstrained setting (\Cref{sec:warm up}), then show how to generalize it to the constrained setting (\Cref{sec:last-iterate EG}).
\vspace{-.1in}
\subsection{Warm Up: Unconstrained Case}
\label{sec:warm up}
As a warm-up, we consider the unconstrained setting where $\Z = \R^n$. Although the last-iterate convergence rate is known in the unconstrained setting due to~\citep{golowich_last_2020, gorbunov_extragradient_2021}, we provide a simpler proof that also permits a larger step size. Our analysis holds for any step size $\eta\in (0,\frac{1}{L})$, while the previous analysis requires $\eta\leq \frac{1}{\sqrt{2}L}$~\citep{gorbunov_extragradient_2021}.

Let $z_k$ be the $k$-th iterate of the EG method. 
In Theorem~\ref{thm:monotone hamiltonian at unconstrained}, we show that the \projham is monotone in the unconstrained setting.\footnote{In the unconstrained setting, the \projham is simply the  norm of the operator $r^{tan}_{(F,\R^n)}(z) = \InNorms{F(z)}$.} Our approach is to apply SOS programming to search for a certificate of non-negativity for $\InNorms{F(z_{k})}^2-\InNorms{F(z_{k+1})}^2$ for every $k$, over the semialgebraic set defined by the following polynomial constraints in variables $\left\{z_i[\ell],\eta F(z_{i})[\ell]\right\}_{i\in \{k,k+\half,k+1\}, \ell\in[n]}$:
\vspace{-.1in}
\begin{align}
 &z_{k+\frac{1}{2}}[\ell]- z_k[\ell]+\eta F(z_{k})[\ell]=0,\quad z_{k+1}[\ell] - z_k[\ell] + \eta F(z_{k+\frac{1}{2}})[\ell]=0,~\forall \ell\in [n], \quad\text{(EG Update)}&
 \notag\\
&\InNorms{\eta F(z_i)-\eta F(z_j)}^2-(\eta L)^2\InNorms{z_i-z_j}^2\leq 0 ,\quad \forall i,j\in \{k,k+\half,k+1\}, \qquad\text{(Lipschitzness)}&\notag\\
    &\InAngles{\eta F(z_i)-\eta F(z_j),z_j-z_i}\leq 0,\quad \forall i,j\in \{k,k+\half,k+1\}. \qquad\text{(Monotonicity)}&\notag
\end{align}
\vspace{-.25in}

We always multiply $F$ with $\eta$ in the constraints as it will be convenient later. We use $K$ to denote the set $\{k,k+\half,k+1\}$. To obtain a certificate of non-negativity, we apply SOS programming to search for a degree-2 SOS proof. More specifically, we want to find non-negative coefficients $\{\lambda^*_{i,j},\mu^*_{i,j}\}_{\substack{i>j, i,j\in K}}$ and degree-1 polynomials $\gamma_1^{(\ell)}(\bw)$ and $\gamma_2^{(\ell)}(\bw)$ in $\mathbb{R}[\bw]$ for each $\ell\in [n]$, where $\bw:=\{z_i[\ell],\eta F(z_{i})[\ell]\}_{i\in K, \ell\in[n]}$,
such that the following is an~SOS~polynomial:
\vspace{-.1in}
\begin{align}\label{eq:sos program unconstrained identity}
& \InNorms{\eta F(z_{k})}^2 - \InNorms{\eta F(z_{k+1})}^2+ \sum_{\substack{i> j \text{ and } i,j\in K}}\lambda^*_{i,j} \cdot \InParentheses{\InNorms{\eta F(z_i)-\eta F(z_j)}^2-(\eta L)^2\InNorms{z_i-z_j}^2 } \notag\\ 
&+\sum_{\substack{i> j \text{ and } i,j\in K}}\mu^*_{i,j} \cdot \InAngles{\eta F(z_i)-\eta F(z_j)),z_j-z_i}+\sum_{\ell\in[n]}\gamma_1^{(\ell)}(\bw)(z_{k+\frac{1}{2}}[\ell]- z_k[\ell]+\eta F(z_{k})[\ell])\notag\\
&+\sum_{\ell\in[n]}\gamma_2^{(\ell)}(\bw)(z_{k+1}[\ell] - z_k[\ell] + \eta F(z_{k+\frac{1}{2}})[\ell]).
\end{align}
\vspace{-.25in}

Due to constraints satisfied by the EG iterates, the non-negativity of Expression~\eqref{eq:sos program unconstrained identity} clearly implies that $\InNorms{F(z_{k})}^2-\InNorms{F(z_{k+1})}^2$ is non-negative. However, Expression~\eqref{eq:sos program unconstrained identity} is in fact an infinite family of polynomials rather than a single one. Expression~\eqref{eq:sos program unconstrained identity} corresponds to a different polynomial for every integer $n$. 
To directly search for the solution, we would need to solve an infinitely large SOS program, which is clearly infeasible. By exploring the symmetry in Expression~\eqref{eq:sos program unconstrained identity}, we show that it suffices to solve a constant size SOS program. Let us first expand Expression~\eqref{eq:sos program unconstrained identity} as follows:
\vspace{-.1in}
\begin{align}\label{eq:unconstrained expanded}
    & \sum_{\ell\in[n]}\Big( \left(\eta F(z_{k})[\ell]\right)^2 - \left(\eta F(z_{k+1})[\ell]\right)^2+ \sum_{\substack{i> j \text{ and } i,j\in K}}\lambda^*_{i,j}  \InParentheses{\InParentheses{\eta F(z_i)[\ell]-\eta F(z_j)[\ell]}^2-(\eta L)^2\InParentheses{z_i[\ell]-z_j[\ell]}^2 } \notag\\ 
& \qquad + \sum_{\substack{i> j \text{ and } i,j\in K}}\mu^*_{i,j}  \InParentheses{\eta F(z_i)[\ell]-\eta F(z_j)[\ell])}\InParentheses{z_j[\ell]-z_i[\ell]}+  \gamma_1^{(\ell)}(\bw)(z_{k+\frac{1}{2}}[\ell]- z_k[\ell]+\eta F(z_{k})[\ell])  \notag\\
& \qquad  +\gamma_2^{(\ell)}(\bw)(z_{k+1}[\ell] - z_k[\ell] + \eta F(z_{k+\frac{1}{2}})[\ell]) \Big).
\end{align}
\vspace{-.15in}

What we will argue next is that, due to the \emph{symmetry across coordinates}, it suffices to directly search for a single SOS proof that shows that each of the $n$ summands in Expression~\eqref{eq:unconstrained expanded} is an SOS polynomial. More specifically, we make use of the following two key properties. \setword{(i)}{property:unconstrained one} For any $\ell,\ell'\in [n]$, the $\ell$-th summand and $\ell'$-th summand are identical subject to a change of variable;\footnote{Simply replace $\{z_i[\ell]\}_{i\in K}$ and $\{\eta F(z_i)[\ell]\}_{i\in K}$ with $\{z_i[\ell']\}_{i\in K}$ and $\{\eta F(z_i)[\ell']\}_{i\in K}$.} \setword{(ii)}{property:unconstrained two} the $\ell$-th summand only depends on the coordinate $\ell$, i.e., variables in $\{z_i[\ell],\eta F(z_{i})[\ell]\}_{i\in K}$ and does not involve any other coordinates.\footnote{We mainly care about the polynomials arise from the constraints. Although $\gamma_1^{(\ell)}(\bw)$ and $\gamma_2^{(\ell)}(\bw)$ could depend on other coordinates, we show that it suffices to consider polynomials in $\{z_i[\ell],\eta F(z_{i})[\ell]\}_{i\in K}$.} 
We solve  the following~SOS~program, whose solution can be used to construct $\{\lambda^*_{i,j},\mu^*_{i,j}\}_{\substack{i>j, i,j\in K}}$ and $\{\gamma_1^{(\ell)}(\bw), \gamma_2^{(\ell)}(\bw)\}_{\ell\in[n]}$ so that each of the summands in Expression~\eqref{eq:unconstrained expanded} is an SOS polynomial.

\begin{figure}[h!]
\colorbox{MyGray}{
\begin{minipage}{\textwidth} {
\small
\noindent\textbf{Input Fixed Polynomials.} We use $\boldsymbol{x}$ to denote $(x_0,x_1,x_2)$ and $\by$ to denote $(y_0,y_1,y_2)$. Interpret $x_i$ as $z_{k+\frac{i}{2}}[\ell]$ and $y_i$ as $\eta F(z_{k+\frac{i}{2}})[\ell]$ for $0\leq i\leq 2$. Observe that  $h_1(\boldsymbol{x},\by)$ and $h_2(\boldsymbol{x},\by)$ come from the EG update rule on coordinate $\ell$. $g^{L}_{i,j}(\boldsymbol{x},\by)$ and $g^{m}_{i,j}(\boldsymbol{x},\by)$ come from the $\ell$-th coordinate's contribution in the Lipschitzness and monotonicity constraints.
\begin{itemize}
\itemsep0em 
\item $h_1(\boldsymbol{x},\by):=x_1-x_0+y_0$ and $h_2(\boldsymbol{x},\by):=x_2-x_0+y_1$. 
\item  $g^L_{i,j}(\boldsymbol{x},\by):=(y_i-y_j)^2-C\cdot (x_i-x_j)^2$ for any $0\leq j<i\leq 2$.\footnote{$C$ represents $(\eta L)^2$. Larger $C$ corresponds to a larger step size and makes the SOS program harder to satisfy. Through binary search, we find that the largest possible value of $C$ is $1$ while maintaining the feasibility of the SOS program.} 

\item  $g^m_{i,j}(\boldsymbol{x},\by):=(y_i-y_j)(x_j-x_i)$ for any $0\leq j<i\leq 2$.
\end{itemize}
\noindent\textbf{Decision Variables of the SOS Program:}
\begin{itemize}
\itemsep0em 
\item $p^L_{i,j} \ge 0$, and $p^m_{i,j} \ge 0$, for all $0\leq j<i\leq 2$.
\item $q_1(\boldsymbol{x},\by)$ and $q_2(\boldsymbol{x},\by)$ are two degree $1$ polynomials in $\mathbb{R}[\boldsymbol{x},\by]$.
\end{itemize}
\textbf{Constraints of the SOS Program:}
\begin{equation}\label{eq:unconstrained SOS}
\begin{array}{ll@{}ll}
\text{s.t.}  & \displaystyle 
y_0^2 - y_2^2 +\sum_{2\geq i>j\geq 0} p^L_{i,j}\cdot g^L_{i,j}(\boldsymbol{x},\by)+\sum_{2\geq i>j\geq 0}p^m_{i,j}\cdot  g^m_{i,j}(\boldsymbol{x},\by)\\
&
\qquad\qquad\qquad\qquad\qquad\qquad+q_1(\boldsymbol{x},\by)\cdot h_1(\boldsymbol{x},\by)+q_2(\boldsymbol{x},\by)\cdot h_2(\boldsymbol{x},\by)\in \sos[\boldsymbol{x},\by].

\end{array}
\end{equation}}
\end{minipage}} 
\caption{Our SOS program in the unconstrained setting.}\label{fig:sos program unconstrained}
\vspace{-.15in}
\end{figure}
\vspace{-.05in}

\notshow{
\begin{figure}[h!]
\colorbox{MyGray}{
\begin{minipage}{\textwidth} {
\small
\noindent\textbf{Input Fixed Polynomials.} 
\begin{itemize}
\item $S_= =\{g_i(\bx)\}_{i\in[M]}$.
\item $S_\leq =\{h_i(\bx)\}_{i\in[M']}$.

\item  $q_{i,j}(\boldsymbol{x},\by):=(y_i-y_j)(x_j-x_i)$ for any $0\leq j<i\leq 2$.
\end{itemize}
\noindent\textbf{Decision Variables of the SOS Program:}
\begin{itemize}
\item $\lambda_{i,j} \ge 0$, and $\mu_{i,j} \ge 0$, for all $0\leq j<i\leq 2$.
\item $\varphi_1(\boldsymbol{x},\by)$ and $\varphi_2(\boldsymbol{x},\by)$ are two degree $1$ polynomials in $\mathbb{R}[\boldsymbol{x},\by]$.
\end{itemize}
\textbf{Constraints of the SOS Program:}
\begin{equation*}
\begin{array}{ll@{}ll}
\text{s.t.}  & \displaystyle y_2^2-y_0^2+\sum_{2\geq i>j\geq 0} \lambda_{i,j}\cdot p_{i,j}(\boldsymbol{x},\by)+\sum_{2\geq i>j\geq 0}\mu_{i,j}\cdot  q_{i,j}(\boldsymbol{x},\by)\\&\qquad\qquad\qquad\qquad\qquad\qquad+\varphi_1(\boldsymbol{x},\by)u_1(\boldsymbol{x},\by)+\varphi_2(\boldsymbol{x},\by)u_2(\boldsymbol{x},\by)\in \sos[\boldsymbol{x},\by].
\end{array}
\end{equation*}}
\end{minipage}} 
\caption{Our SOS program in the unconstrained setting.}\label{fig:sos program unconstrained}
\end{figure}
}

The proof of the following theorem is based on a feasible solution to the SOS program in Figure~\ref{fig:sos program unconstrained}.

\begin{theorem}\label{thm:monotone hamiltonian at unconstrained}
Let $F: \R^n \rightarrow \R^n$ be a monotone and $L$-Lipschitz operator. Then for any $k \in \mathbb{N}$, the EG algorithm with step size $\eta \in (0,\frac{1}{L})$ satisfies $\|F(z_k)\|^2 \ge \|F(z_{k+1})\|^2$. 
\end{theorem}
\vspace{-.1in}
\begin{proof}
Since $F$ is monotone and $L$-Lipschitz, we have $$
   \InAngles{F(z_{k+1})-F(z_k)  , z_{k} - z_{k+1}}  \le 0$$ and $$\InNorms{F(z_{k+\frac{1}{2}})-F(z_{k+1})}^2 - L^2 \InNorms{z_{k+\frac{1}{2}} - z_{k+1}}^2 \le 0.$$
We simplify them using the update rule of EG and $\eta L<1$. In particular, we replace $z_k-z_{k+1}$ with $\eta F(z_{k+\half})$ and $z_{k+\frac{1}{2}} - z_{k+1}$ with $\eta F(z_{k+\frac{1}{2}}) - \eta F(z_k)$.

\begin{align}
   & \InAngles{ F(z_{k+1}) - F(z_k) ,F(z_{k+\frac{1}{2}})} \le 0, \label{eq:unconstrained-mon}\\
   &\InNorms{F(z_{k+\frac{1}{2}})-F(z_{k+1})}^2 -\InNorms{F(z_{k+\frac{1}{2}}) - F(z_k)}^2  \le 0 \label{eq:unconstrained-lip}.
\end{align}

In Proposition~\ref{prop:verify unconstrained} at Appendix~\ref{appx:warm up} we verify the following identity.
\begin{align*}
    \|F(z_k)\|^2 - \|F(z_{k+1})\|^2 + 2\cdot \LHSI (\ref{eq:unconstrained-mon}) + \LHSI (\ref{eq:unconstrained-lip}) = 0.
\end{align*}
Thus, $\|F(z_k)\|^2 - \|F(z_{k+1})\|^2 \ge 0$. 
\end{proof}

Corollary~\ref{cor:EG last-iterate unconstrained} is implied by combing Lemma~\ref{lem:hamiltonian to gap}, Lemma~\ref{lem:best iterate hamiltonian}, Theorem~\ref{thm:monotone hamiltonian at unconstrained} and the fact that $\eta \in \InParentheses{0,\frac{1}{L}}$.
\vspace{-.1in}
\begin{corollary}\label{cor:EG last-iterate unconstrained}
Let $F(\cdot):\R^n \rightarrow \R^n$ be a monotone and $L$-Lipshitz operator and $z^*\in \R^n$ be a solution to the variational inequality. For any $T \ge 1$, let $z_T$ be $T$-th iterate of the EG algorithm with constant step size $\eta \in (0,\frac{1}{L})$,
then
$\gap_{\R^n,F,D}(z_T)  \le   \frac{1}{\sqrt{ T}}\frac{3 D\|z_0-z^*\|}{\eta \sqrt{1-(\eta L)^2}}$.
\end{corollary}

\vspace{-.2in}
\subsection{Last-Iterate Convergence of EG with Arbitrary Convex Constraints}\label{sec:last-iterate EG}
We establish the last-iterate convergence rate of the EG algorithm in the constrained setting in this section. The plan is similar to the one in Section~\ref{sec:warm up}. First, we use the assistance of SOS programming to prove the monotonicity of the \projham (Theorem~\ref{thm:monotone hamiltonian}), then combine it with the best-iterate convergence guarantee from Lemma~\ref{lem:best iterate hamiltonian} to derive the last-iterate convergence rate  (Theorem~\ref{thm:EG last-iterate constrained}). 

Due to the constraints, proving the monotonicity of the tangent residual becomes much more challenging. The \projham in the constrained setting ( Definition~\ref{def:projected hamiltonian}) is significantly more complex than its counterpart in the unconstrained setting. In Lemma~\ref{lemma:property of tangent residual}, we introduce an auxiliary point $c(z)$ for every point $z$ that can be used to simplified the tangent residual.

\begin{lemma}\label{lemma:property of tangent residual}
Let $\Z \subseteq \R^n$ be a closed convex set and $F: \Z \rightarrow \R$ be an operator. For any $z \in \Z$, denote $c(z) := \Pi_{N(z)}\inteval{-F(z)}$ the projection of $-F(z)$ on the normal cone $N(z)$. Then we have 
\begin{itemize}
\item $r^{tan}(z) = \InNorms{F(z)+c(z)}$,
    \item $\InAngles{F(z) + c(z) , c(z)} = 0$,
    \item $\InAngles{F(z)+c(z), a} \ge 0$, $\forall a \in N(z)$.
\end{itemize}
\end{lemma}
\begin{proof}
According to the definition of $c(z)$, $r^{tan}(z) = \InNorms{F(z)+c(z)}$ follows from \Cref{lem:equivalent hamiltonian}.
Since $c(z) = \Pi_{N(z)}\inteval{-F(z)}$, we know that for all $a \in N(z)$,
\begin{align}\label{eq:proj}
    \InAngles{-F(z) - c(z), a - c(z)} \le 0.
\end{align}
Note that $c(z) \in N(z)$ and $N(z)$ is a cone. By substituting $a = 0$ and $a = 2\cdot c(z)$ in \eqref{eq:proj}, we get
\begin{align*}
    \InAngles{-F(z) - c(z), c(z)} = 0. 
\end{align*}
Therefore, for all $a \in N(z)$, we have
\begin{align*}
    \InAngles{-F(z) - c(z), a} =  \InAngles{-F(z) - c(z), a - c(z)} \le 0.
\end{align*}

\end{proof}

Next, we need to decide over which semialgebraic set that we want to certify the non-negativity of $r^{tan}(z_k)^2-r^{tan}(z_{k+1})^2$. Naturally, we would like to use all constraints of $\Z$, but there might be arbitrarily many of them. In the next paragraph, we argue how to reduce the number of constraints.

\paragraph{Reducing the Number of Constraints.} Suppose we are not given the description of $\Z\subseteq \R^n$, and we only observe one iteration of the EG algorithm.
In other words, we know $z_k$, $z_{k+\half}$, and $z_{k+1}$, as well as $F(z_k)$, $F(z_{k+\half})$, and $F(z_{k+1})$. 
To express the squared \projham at $z_k$ and the squared \projham at $z_{k+1}$, let us also assume that the vector $c_k=\Pi_{N(z_k)}\inteval{-F(z_k)}$ and $c_{k+1}=\Pi_{N(z_{k+1})}\inteval{-F(z_{k+1})}$, and according to \Cref{lemma:property of tangent residual}, we have $r^{tan}(z_k)^2 = {\InNorms{F(z_k) + c_k}^2}$, and $r^{tan}(z_{k+1})^2 = {\InNorms{F(z_{k+1}) + c_{k+1}}^2}$. Our plan is to derive a set of inequalities that must be satisfied by these vectors.
From this limited information, what can we learn about $\Z$? We can conclude that $\Z$ must lie in the intersection of the following halfspaces: 
(a) $\InAngles{c_k,z}\leq \InAngles{c_k,z_k}$. This is true because $c_k\in \normal(z_k)$. 
(b) $\InAngles{a_{k+\half},z}\geq \InAngles{a_{k+\half},z_{k+\half}}$, where $a_{k+\half}=-\InParentheses{z_k-\eta F(z_k)-z_{k+\half}}$. This is true because $z_{k+\half}=\Pi_\Z(z_k-\eta F(z_k))$, so $-a_{k+\half}\in \normal(z_{k+\half})$.
(c) $\InAngles{a_{k+1},z}\geq \InAngles{a_{k+1},z_{k+1}}$, where $a_{k+1}=-\InParentheses{z_k-\eta F(z_{k+\half})-z_{k+1}}$.
This is true because $z_{k+1}=\Pi_\Z(z_k-\eta F(z_{k+\half}))$, so $-a_{k+1}\in \normal(z_{k+1})$.
See Figure~\ref{fig:geometry of EG} 
for illustration. Additionally, due to our definition of $c_k$ and $c_{k+1}$ and Lemma~\ref{lemma:property of tangent residual}, we know that  (d) $\InAngles{\eta F(z_i) + \eta c_i , \eta c_i} = 0$ for  $i\in\{k,k+1\}$, and (e) 
$\InAngles{\eta F(z_{k+1})+\eta c_{k+1}, a_{k+1}} \leq 0$ as $-a_{k+1}\in N(z_{k+1})$.

\begin{figure}[ht]
\centering
\includegraphics[width=0.75\textwidth]{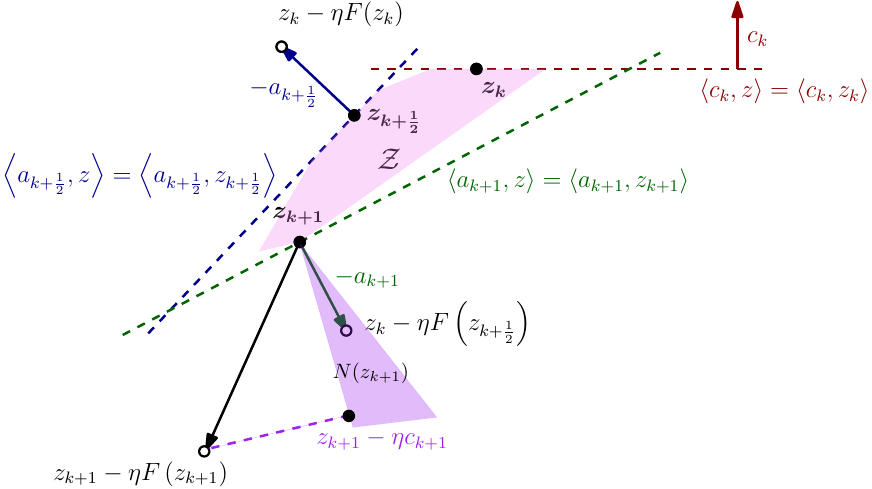}
\caption{
Reducing the number of constraints.}
\label{fig:geometry of EG}
\end{figure}

Clearly, for any $\Z$, the inequalities in (a) to (e) must hold, though there might be other inequalities that are also true. Our goal is to prove that the \projham is non-increasing even if only  inequalities (a) to (e) hold. If we can do so, then we prove that \projham is non-increasing for an arbitrary $\Z$.

\vspace{-.2in}
\paragraph{Formulation as SOS program.} 
Similar to the unconstrained case,
our plan is to search for a certificate of non-negativity of the following expression 
\begin{align}\label{eq:hardest instance main}
    \InNorms{F(z_k) + c_k}^2-\InNorms{F(z_{k+1})+c_{k+1}}^2
\end{align}
 over the semialgebraic set defined by the following polynomial constraints in variables $\left\{\left\{z_i[\ell],\eta F(z_{i})[\ell]\right\}_{i\in \{k,k+\half,k+1\}} \cup \{c_i[\ell]\}_{i\in{k,k+1}}\right\}_{ \ell\in[n]}$
\begin{align}
&\InNorms{\eta F(z_i)-\eta F(z_j)}^2-(\eta L)^2\InNorms{z_i-z_j}^2\leq 0 ,\quad &\forall i,j\in \{k,k+\half,k+1\}, \qquad\text{(Lipschitzness)}&\notag\\
&\InAngles{\eta F(z_i)-\eta F(z_j),z_j-z_i}\leq 0,\quad &\forall i,j\in \{k,k+\half,k+1\}, \qquad\text{(Monotonicity)}&\notag\\
&\InAngles{a_i,z_i-z_j}\leq 0,\quad &\forall i\in \{k+\half,k+1\},j\in \{k,k+\half,k+1\}, \qquad\InParentheses{-a_i \in \normal(z_i)}&\notag\\
&\InAngles{\eta c_i,z_j-z_i}\leq 0,\quad &\forall i\in \{k,k+1\},j\in \{k,k+\half,k+1\}, \qquad\InParentheses{c_i \in \normal(z_i)}&\notag\\
& \InAngles{\eta F(z_i) + \eta c_i , \eta c_i} = 0,\quad &\forall i\in \{k,k+1\}, \qquad\InParentheses{\Cref{lemma:property of tangent residual}}&\notag\\
& \InAngles{\eta F(z_{k+1})+\eta c_{k+1}, a_{k+1}} \leq 0,\quad & \qquad\InParentheses{\Cref{lemma:property of tangent residual}}&\notag.
\end{align}
Similar to Section~\ref{sec:EG last iterate everything}, we multiply the operators, $c_k$, and $c_{k+1}$ with $\eta$ for convenience.
Fortunately,
the dimensional-dependent Expression~\eqref{eq:hardest instance main} and semialgebraic set are symmetric across coordinates, and more specifically, satisfy the two key properties in the unconstrained case -- Property~\ref{property:unconstrained one} and~\ref{property:unconstrained two}.
Hence, we can represent all of the coordinates $\ell\geq 1$ with one coordinate in the SOS program, and we can form a constant size SOS program to search for a certificate of non-negativity for Expression~\eqref{eq:hardest instance main} as shown in Figure~\ref{fig:sos program constrained}.

In Theorem~\ref{thm:monotone hamiltonian}, we establish the monotonicity of the \projhamnospace. Our proof is based on the solution to the degree-2 SOS program concerning polynomials in $8$ variables (Figure~\ref{fig:sos program constrained}).

\begin{figure}[h!]
\colorbox{MyGray}{
\begin{minipage}{\textwidth} {
\small
\noindent\textbf{Input Fixed Polynomials.}
We use $\boldsymbol{x}$ to denote $(x_0,x_1,x_2)$, $\by$ to denote $(y_0,y_1,y_2)$ and $\boldsymbol{w}$ to denote $(w_0,w_2)$. 
Interpret $x_i$ as $z_{k+\frac{i}{2}}[\ell]$ and $y_i$ as $\eta F(z_{k+\frac{i}{2}})[\ell]$ for $0\leq i\leq 2$, $w_0$ as $\eta c_k[\ell]$ and $w_2$ as $\eta c_{k+1}[\ell]$.
Let $b_1 = -\InParentheses{x_0 -y_0-x_1}$ and $b_2=-\InParentheses{x_0-y_1-x_2}$.

\noindent\textbf{Origin of Constraints.}
$g^{L}_{i,j}(\boldsymbol{x},\by,\boldsymbol{w})$ and $g^{m}_{i,j}(\boldsymbol{x},\by,\boldsymbol{w})$ come from the $\ell$-th coordinate's contribution in the Lipschitzness and monotonicity constraints.
Similarly, $g^{b}_{i,j}(\boldsymbol{x},\by,\boldsymbol{w})$ and $g^{w}_{i,j}(\boldsymbol{x},\by,\boldsymbol{w})$ come from the $\ell$-th coordinate contribution of fact that $-a_i$ and $c_i$ are in the normal cone of $z_i$.
Finally, $h_i^w(\boldsymbol{x},\by, \boldsymbol{w})$ and $g^{r}(\boldsymbol{x},\by,\boldsymbol{w})$ comes from the $\ell$-th coordinate contribution due to the inequalities of Lemma~\ref{lemma:property of tangent residual}.
\begin{itemize}
\itemsep0em 
\item  $g^L_{i,j}(\boldsymbol{x},\by,\boldsymbol{w}):=(y_i-y_j)^2-C\cdot (x_i-x_j)^2$ for any $0\leq j<i\leq 2$.\footnote{$C$ represents $(\eta L)^2$.}

\item  $g^{m}_{i,j}(\boldsymbol{x},\by,\boldsymbol{w}):=(y_i-y_j)(x_j-x_i)$ for any $0\leq j<i\leq 2$.

\item  $g^{b}_{i,j}(\boldsymbol{x},\by,\boldsymbol{w}):=b_i\cdot (x_i-x_j)$ for any $i\in\{1,2\},0\leq j\leq 2$.

\item  $g^{w}_{i,j}(\boldsymbol{x},\by,\boldsymbol{w}):=w_i\cdot (x_j-x_i)$ for any $i\in\{0,2\},0\leq j\leq 2$.

\item  $g^{r}(\boldsymbol{x},\by,\boldsymbol{w})~:=(y_2+w_2)\cdot b_2$.

\item  $h^{w}_i(\boldsymbol{x},\by,\boldsymbol{w})~:=(y_i+w_i)\cdot w_i$ for any 
$i\in\{0,2\}$.

\end{itemize}
\noindent\textbf{Decision Variables of the SOS Program:}
\begin{itemize}
\itemsep0em 
\item $p^L_{i,j} \ge 0$, and $p^m_{i,j} \ge 0$, for all $0\leq j<i\leq 2$.
\item $p^b_{i,j}\geq 0$, for any $i\in\{1,2\}$, $0\leq j \leq 2$.
\item $p^w_{i,j}\geq 0$, for any $i\in\{0,2\}$, $0\leq j \leq 2$.
\item $p^r\geq 0$.
\item $q_0^w,q_2^w\in \R$.
\end{itemize}
\textbf{Constraints of the SOS Program:}
\begin{equation}\label{eq:constrained SOS}
\begin{array}{ll@{}ll}
\text{s.t.}  & \displaystyle 
(y_0+w_0)^2 - (y_2+w_2)^2 +\sum_{2\geq i>j\geq 0} p^L_{i,j}\cdot g^L_{i,j}(\boldsymbol{x},\by, \boldsymbol{w})+\sum_{2\geq i>j\geq 0}p^m_{i,j}\cdot  g^m_{i,j}(\boldsymbol{x},\by, \boldsymbol{w})&\\
&
\qquad+\sum_{i\in\{1,2\},2\geq j\geq 0}p^b_{i,j}\cdot g_{i,j}^{b}(\boldsymbol{x},\by,\boldsymbol{w})
+\sum_{i\in\{0,2\},2\geq j\geq 0}p^w_{i,j}\cdot g_{i,j}^{w}(\boldsymbol{x},\by,\boldsymbol{w}) & \quad\in SOS[\boldsymbol{x},\boldsymbol{y},\boldsymbol{w}]\\
&
\qquad +p^r\cdot g^r(\boldsymbol{x},\by,\boldsymbol{w}) +\sum_{i\in 
\{0,2\}}q^{w}_i\cdot h_{i}^{w}(\boldsymbol{x},\by,\boldsymbol{w})
&

\end{array}
\end{equation}}
\end{minipage}} 
\caption{Our SOS program in the constrained setting.}\label{fig:sos program constrained}
\vspace{-.15in}
\end{figure}

\vspace{-.1in}
\begin{theorem}\label{thm:monotone hamiltonian}
Let $\Z \subseteq \R^n$ be a closed convex set and $F:\Z \rightarrow \R^n$ be a monotone and $L$-Lipschitz operator. For any step size $\eta \in (0, \frac{1}{L}$) and any $z_k \in \Z$, the EG method update satisfies $r^{tan}_{(F,\Z)}(z_k) \ge r^{tan}_{(F,\Z)}(z_{k+1})$.
\end{theorem}
\begin{proof}
Let $c_k = \Pi_{\normal_\Z(z_k)}(-F(z_k))$ and  $c_{k+1} = \Pi_{\normal_\Z(z_{k+1})}(-F(z_{k+1}))$.
By Lemma~\ref{lemma:property of tangent residual} we have
\begin{align}
    \eta^2 r^{tan}(z_k)^2 - \eta^2 r^{tan}(z_{k+1})^2 =  \|\eta F(z_k)+\eta c_k\|^2- \|\eta F(z_{k+1})+\eta c_{k+1}\|^2 \label{eq:low deg LHS 1}
\end{align}
Combining the monotonicity and $L$-Lipschitzness of $F$ with the fact that $L \leq \frac{1}{\eta}$, we have
\begin{align}
&(-1)\cdot \left(\InNorms{z_{k+\half}-z_{k+1}}^2 - \InNorms{\eta F(z_{k+\half})-\eta F(z_{k+1})}^2\right)  \le 0 \label{eq:low deg LHS 2}, \\
   & (-2)\cdot \InAngles{ \eta F(z_{k+1}) - \eta F(z_k) , z_{k+1} - z_k)} \le 0, \label{eq:low deg LHS 3}.
\end{align}

Since $z_{k+\half}=\Pi_{\Z}\left(z_k-\eta F(z_k)\right)$ and $z_{k+1}=\Pi_{\Z}\left(z_k-\eta F(z_{k+\half})\right)$,
we can infer that $z_k-\eta F(z_k) - z_{k+\half}\in \normal(z_{k+\half})$ and $z_k-\eta F(z_{k+\half}) -
z_{k+1}\in \normal(z_{k+1})$,
which further implies
\begin{align}
   & (-2)\cdot \InAngles{ z_k - \eta F(z_k) - z_{k+\half},z_{k+\half} - z_{k+1} } \le 0, \label{eq:low deg LHS 4}\\
   & (-2)\cdot \InAngles{ z_k - \eta F(z_{k+\half}) - z_{k+1},z_{k+1} - z_k } \le 0, \label{eq:low deg LHS 5}\\
   & (-2)\cdot \InAngles{ \eta c_k,z_{k} - z_{k+\half} } \le 0, \label{eq:low deg LHS 6}.
\end{align}
Since $z_k - \eta F(z_{k+\half}) - z_{k+1}\in \normal(z_{k+1})$ and 
$c_{k+1}=\Pi_{\normal_\Z(z_{k+1})}(-F(z_{k+1}))$,
by Lemma~\ref{lemma:property of tangent residual} we have
\begin{align}
   & (-2)\cdot \InAngles{ \eta c_{k+1} + \eta F(z_{k+1}), z_k - \eta F(z_{k+\half}) - z_{k+1} } \leq 0 \label{eq:low deg LHS 7},\\
   & (-2)\cdot \InAngles{ \eta c_{k+1} + \eta F(z_{k+1}), -\eta c_{k+1} } = 0 \label{eq:low deg LHS 8}.
\end{align}

\textsc{Matlab} code for the verification of the following identity can be found at 
this \href{https://raw.githubusercontent.com/aroikonomou/code-verification-OGDA/main/eg_constrained_identity_verification.m}{link}.

\begin{align}
     &\text{Expression}~\eqref{eq:low deg LHS 1} + \LHSI~\eqref{eq:low deg LHS 2} + \LHSI~\eqref{eq:low deg LHS 3} \notag \\
     + & \LHSI~\eqref{eq:low deg LHS 4} 
     + \LHSI~\eqref{eq:low deg LHS 5} + \LHSI~\eqref{eq:low deg LHS 6} \notag \\
     + &\LHSI~\eqref{eq:low deg LHS 7} + \LHSI~\eqref{eq:low deg LHS 8} \notag\\
     =& \|\eta F(z_k) + \eta c_k - z_k + z_{k+\half}\|^2 \label{eq:low deg RHS 1} \\
     +& \|\eta F(z_{k+\half}) +\eta c_{k+1} - z_k + z_{k+1}\|^2 \label{eq:low deg RHS 2}\geq 0,
\end{align}
which concludes the proof.
\end{proof}

\vspace{-0.15in}
\begin{theorem}\label{thm:EG last-iterate constrained}
Let $\Z \subseteq \R^n$ be a closed convex set, $F(\cdot):\Z\rightarrow \R^n$ be a monotone and $L$-Lipschitz operator and $z^*\in \Z$ be the solution to the variational inequality. Then for any $T \ge 1$, $z_T$ produced by EG with any constant step size $\eta \in (0,\frac{1}{L})$ satisfies
\begin{itemize}
    \item $\gap(z_T)  \le   \frac{1}{\sqrt{T}}\frac{3D||z_0-z^*||}{\eta\sqrt{1-(\eta L)^2}}$,
    \item$ r^{nat}(z_T) \le \Ham(z_T) \le \frac{1}{\sqrt{T}}\frac{3||z_0-z^*||}{\eta\sqrt{1-(\eta L)^2}}$.
\end{itemize}

\end{theorem}
Theorem~\ref{thm:EG last-iterate constrained} is implied by combing~\Cref{lem:projected hamiltonian dominates natural residue}, \Cref{lem:hamiltonian to gap}, \Cref{lem:best iterate hamiltonian}, \Cref{thm:monotone hamiltonian} and the fact that $\eta\in \InParentheses{0,\frac{1}{L}}$. Choosing $\eta$ to be $\frac{1}{2L}$ and $D=O(\InNorms{z_0-z^*})$, then $\gap(z_T)= O(\frac{D^2L}{\sqrt{T}})$ matching the $\Omega(\frac{D^2L}{\sqrt{T}})$ lower bound for EG, OGDA, and more generally all p-SCLI algorithms~\citep{golowich_last_2020,golowich_tight_2020} in terms of the dependence on $D$, $L$, and $T$. Additionally, $r_{\Z,F,D}^{nat}(z_{T})= O\InParentheses{\frac{DL}{\sqrt{T}}}$, and $r_{\Z,F,D}^{tan}(z_{T})= O\InParentheses{\frac{DL}{\sqrt{T}}}$, so our upper bounds for both the natural residual and tangent residual also match the $\Omega\InParentheses{\frac{DL}{\sqrt{T}}}$ lower bounds with respect to natural residual and tangent residual for EG~\citep{golowich_last_2020}. This is because both the natural residual and tangent residual are equivalent to the norm of the operator, and \citep{golowich_last_2020} shows that in the unconstrained setting $\InNorms{F(z_T)}=\Omega\InParentheses{\frac{DL}{\sqrt{T}}}$.

\section{Last-Iterate Convergence of OGDA with Constant Step Size}\label{sec:OGDA}
In this section, we show that the OGDA algorithm with constant step size $\eta \in (0, \frac{1}{2L})$ has $O(\frac{1}{\sqrt{T}})$ last-iterate convergence rate with respect to the tangent residual or the gap function.
The analysis of the OGDA algorithm follows the same steps as in the analysis of the EG algorithm.
Compared to the EG algorithm, 
the last iterate convergence of OGDA follows by builds on the monotonicity and best-iterate convergence of the following potential function
\begin{align}
\Phi_k = \InNorms{F(z_k)-F(w_k)}^2 + r^{tan}(z_k)^2.
\end{align}
The potential function can be thought of as the tangent residual ($r^{tan}(z_k)^2$) and an extra correction term $\InNorms{F(z_k) - F(w_k)}^2$.
The potential function is discovered directly through SOS programming.
The SOS program was formulated by searching over linear combinations of $||F(z_k)||^2 - ||F(z_{k+1})||^2$,$||F(w_k)||^2 - ||F(w_{k+1})||^2$,$\InAngles{F(z_k),F(w_k)} - \InAngles{F(z_{k+1}),F(w_{k+1})}$ and $r^{tan}(z_k)-r^{tan}(z_{k+1})$,
under (i) the constraint that the linear combination is non-increasing,\footnote{To avoid finding the  trivial linear combination, i.e.,  all coefficients equal to $0$, we also use the objective function in the SOS program  to encourage a non-trivial solution if one exists by, for example, maximizing the sum of the coefficients of the linear combination.} and (ii) the constraints induced by properties of the operator $F(\cdot)$, the update rule of OGDA and the set $\Z$ (See \Cref{fig:sos program constrained} for a demonstration of the induced constraints of EG algorithm for solving a monotone VI over convex constraints). We then use the  linear combination output by the SOS program as the potential function in our analysis. We believe our heuristic for finding a potential function could be useful in other settings. In general, one can first choose a collection of basis functions that may be part of a potential function, then use SOS programming to search over all linear combinations of the basis functions subject to the constraint that the linear combination is non-negative to discover the potential function.
We postpone all technical details to \Cref{appx:OGDA}.

\subsection{Warm Up: Unconstrained Case}
In this section, we show that in the unconstrained setting the potential function $\Phi_k$ is non-increasing, which implies the $O(\frac{1}{\sqrt{T}})$ last-iterate convergence rate for OGDA with respect to the tangent residual, the natural residual, and the gap function.
\begin{theorem}\label{thm:monotone potential at unconstrained}
Let $F: \R^n \rightarrow \R^n$ be a monotone and $L$-Lipschitz operator. Then for any $k \in \mathbb{N}$, the OGDA algorithm with step size $\eta \in (0,\frac{1}{2L})$ satisfies $\InNorms{F(z_k) - F(w_k)}^2+ \InNorms{F(z_k)}^2 \ge \InNorms{F(z_{k+1}) - F(w_{k+1})}^2+ \InNorms{F(z_{k+1})}^2$. 
\end{theorem}
\begin{proof}
Since $F$ is monotone and $L$-Lipschitz, we have $
   \InAngles{F(z_{k+1})-F(z_k)  , z_{k} - z_{k+1}}  \le 0$ and $\InNorms{F(w_{k+1})-F(z_{k+1})}^2 - L^2 \InNorms{w_{k+1} - z_{k+1}}^2 \le 0$. 
We simplify them using the update rule of OGDA and $\eta^2 L^2<\frac{1}{4}$. In particular, we replace $z_k-z_{k+1}$ by $\eta F(w_{k+1})$ and $w_{k+1} - z_{k+1}$ with $\eta F(w_{k+1}) - \eta F(w_k)$.
\begin{align}
   & \InAngles{ F(z_{k+1}) - F(z_k) ,F(w_{k+1})} \le 0, \label{eq:OGDA unconstrained-mon}\\
   &\InNorms{F(w_{k+1})-F(z_{k+1})}^2 -\frac{1}{4}\InNorms{F(w_{k+1}) - F(w_k)}^2  \le 0 \label{eq:OGDA unconstrained-lip}.
\end{align}


\textsc{Matlab} code for the verification of the following identity can be found at 
this \href{https://raw.githubusercontent.com/aroikonomou/code-verification-OGDA/main/ogda_target2_identity_verification.m}{link}.

\begin{align}
    &\InNorms{F(z_k) - F(w_k)}^2+ \InNorms{F(z_k)}^2 - \InNorms{F(z_{k+1}) - F(w_{k+1})}^2- \InNorms{F(z_{k+1})}^2 \notag\\
    &+ 2\cdot \LHSI (\ref{eq:OGDA unconstrained-mon}) + 2 \cdot \LHSI (\ref{eq:OGDA unconstrained-lip}) \notag\\
    &=\frac{1}{2}\InNorms{F(w_k) + F(w_{k+1}) - 2F(z_k))}^2.\label{eq:identity OGDA unconstrained}
\end{align}
Thus, $\InNorms{F(z_k) - F(w_k)}^2+ \InNorms{F(z_k)}^2 \ge \InNorms{F(z_{k+1}) - F(w_{k+1})}^2+ \InNorms{F(z_{k+1})}^2$. 
\end{proof}

The following theorem is a 
combination of \Cref{cor:OGDA potential Best-iterate}, \Cref{thm:monotone potential at unconstrained}, \Cref{lem:OGDA relate w and z}, \Cref{lem:projected hamiltonian dominates natural residue} and \Cref{lem:hamiltonian to gap}.

\notshow{
\begin{theorem}\label{thm: last-iterate at unconstrained}
Let $F: \R^n \rightarrow \R$ be a monotone and $L$-Lipschitz operator. Let $z_0 = w_0 \in \R^n$ be arbitrary starting point and $\{z_k,w_k\}_{k\ge 0}$ be the iterates of the OGDA algorithm with any step size $\eta \in  (0,\frac{1}{2L})$. Denote $D_0 := \sqrt{\frac{4+ 6\eta^4 L^4}{\eta^2 L^2 (1-4\eta^2 L^2)}  \InNorms{z_0 - z^*}^2 +   \frac{16\eta^2 L^2 + 6\eta^4 L^4}{\eta^2 L^2(1- 4\eta^2L^2)} \InNorms{w_0 - z_0}^2} = O(\max\{\InNorms{z_0 - z^*}, \InNorms{w_0 - z_0}\})$. Then for any $T \ge 1$, 
\begin{itemize}
    \item $\gap_{\Z,F,D}(z_{T}) \le   \frac{DD_0 L}{\sqrt{ T}}$.
    \item $r_{\Z,F,D}^{nat}(z_{T}) \le r_{\Z,F,D}^{tan}(z_T) \le  \frac{D_0 L}{\sqrt{T}} $.
    \item $\gap_{\Z,F,D}(w_{T+1}) \le \sqrt{2}(2+\eta L)  \frac{DD_0 L}{\sqrt{ T}}$.
    \item $r_{\Z,F,D}^{nat}(w_{T+1}) \le r_{\Z,F,D}^{tan}(w_{T+1}) \le \sqrt{2}(2+\eta L) \frac{D_0 L}{\sqrt{T}} $.
\end{itemize}
\end{theorem}
}

\begin{theorem}\label{thm: last-iterate at unconstrained}
Let $F: \R^n \rightarrow \R$ be a monotone and $L$-Lipschitz operator. Let $z_0 = w_0 \in \R^n$ be arbitrary starting point and $\{z_k,w_k\}_{k\ge 0}$ be the iterates of the OGDA algorithm with any step size $\eta \in  (0,\frac{1}{2L})$. Denote $D_0 := \sqrt{\InParentheses{4+ 6\eta^4 L^4}  \InNorms{z_0 - z^*}^2 +   \InParentheses{16\eta^2 L^2 + 6\eta^4 L^4} \InNorms{w_0 - z_0}^2} = O(\max\{\InNorms{z_0 - z^*}, \InNorms{w_0 - z_0}\})$. Then for any $T \ge 1$, 
\begin{itemize}
    \item $\gap_{\Z,F,D}(z_{T}) \le  \frac{1}{\sqrt{T}}\cdot  \frac{DD_0 }{\eta \sqrt{1-4 (\eta L)^2)}}$.
    \item $r_{\Z,F,D}^{nat}(z_{T}) \le r_{\Z,F,D}^{tan}(z_T) \le \frac{1}{\sqrt{T}}\cdot \frac{D_0}{\eta\sqrt{1-4 (\eta L)^2}} $.
    \item $\gap_{\Z,F,D}(w_{T+1}) \le  \frac{1}{\sqrt{T}}\cdot \frac{\sqrt{2}(2+\eta L)\cdot DD_0 }{\eta \sqrt{ 1-4  (\eta L)^2}}$.
    \item $r_{\Z,F,D}^{nat}(w_{T+1}) \le r_{\Z,F,D}^{tan}(w_{T+1}) \le  \frac{1}{\sqrt{T}}\cdot \frac{\sqrt{2}(2+\eta L)\cdot D_0 }{\eta \sqrt{1-4 (\eta L)^2}} $.
\end{itemize}
\end{theorem}


\subsection{Last-Iterate Convergence of OGDA with Arbitrary Convex Constraints}

In this section,
we formally state the last-iterate convergence of OGDA algorithm with respect to the gap function, the natural residual and the tangent residual in the constrained setting.
All the details are postponed to Appendix~\ref{appx:OGDA}.
\begin{theorem}\label{thm: OGDA last-iterate at constrained main}
Let $\Z \subseteq \R^n$ be a closed convex set and $F: \Z \rightarrow \R$ be a monotone and $L$-Lipschitz operator. Let $z_0,w_0 \in \Z$ be arbitrary starting point and $\{z_k,w_k\}_{k\ge 0}$ be the iterates of the OGDA algorithm with any step size $\eta \in  (0,\frac{1}{2L})$. 
Let $D_0 := \sqrt{\InParentheses{4+ 6\eta^4 L^4}  \InNorms{z_0 - z^*}^2 +  \InParentheses{ 16\eta^2 L^2 + 6\eta^4 L^4} \InNorms{w_0 - z_0}^2} = O(\max\{\InNorms{z_0 - z^*}, \InNorms{w_0 - z_0}\})$.
Then for any $T \ge 1$, 
\begin{itemize}
    \item $\gap_{\Z,F,D}(z_{T}) \le \frac{DD_0 }{ \eta \cdot \sqrt{ T\cdot \left(1-4\cdot (\eta L)^2\right)}}$.
    \item $r_{\Z,F,D}^{nat}(z_{T}) \le r_{\Z,F,D}^{tan}(z_T) \le \frac{D_0 }{ \eta \cdot \sqrt{ T\cdot \left(1-4\cdot (\eta L)^2\right)}} $.
    \item $\gap_{\Z,F,D}(w_{T+1}) \le \frac{\sqrt{2}(2+\eta L)\cdot D\cdot D_0 }{ \eta \cdot \sqrt{ T\cdot \left(1-4\cdot (\eta L)^2\right)}}$.
    \item $r_{\Z,F,D}^{nat}(w_{T+1}) \le r_{\Z,F,D}^{tan}(w_{T+1}) \le  \frac{\sqrt{2}(2+\eta L) D_0 }{ \eta \cdot \sqrt{ T\cdot \left(1-4\cdot (\eta L)^2\right)}} $.
\end{itemize}
\notshow{
\begin{align*}
    &\gap_{\Z,F,D}(z_T) \le   \frac{DL}{\sqrt{ T}} \sqrt{\frac{4+ 6\eta^4 L^4}{\eta^2 L^2 (1-4\eta^2 L^2)}  \InNorms{z_0 - z^*}^2 +  \weiqiangnote{ \frac{6\eta^2 L^2 + 16\eta^4 L^4}{\eta^2 L^2(1- 4\eta^2L^2)} \InNorms{w_0 - z_0}^2}},\\
    &r_{\Z,F,D}^{nat}(z_T) \le r_{\Z,F,D}^{tan}(z_T) \le \frac{L}{\sqrt{ T}} \sqrt{\frac{4+ 6\eta^4 L^4}{\eta^2 L^2 (1-4\eta^2 L^2)}  \InNorms{z_0 - z^*}^2 +  \weiqiangnote{ \frac{6\eta^2 L^2 + 16\eta^4 L^4}{\eta^2 L^2(1- 4\eta^2L^2)} \InNorms{w_0 - z_0}^2}},\\
    &\gap_{\Z,F,D}(w_{T+1}) \le 2(2+\eta L) \frac{DL}{\sqrt{T}} \sqrt{\frac{4+ 6\eta^4 L^4}{\eta^2 L^2 (1-4\eta^2 L^2)}  \InNorms{z_0 - z^*}^2 + \weiqiangnote{ \frac{6\eta^2 L^2 + 16\eta^4 L^4}{\eta^2 L^2(1- 4\eta^2L^2)} \InNorms{w_0 - z_0}^2}},\\
    &r_{\Z,F,D}^{nat}(w_{T+1}) \le r_{\Z,F,D}^{tan}(w_{T+1}) \le 2(2+\eta L) \frac{L}{\sqrt{T}} \sqrt{\frac{4+ 6\eta^4 L^4}{\eta^2 L^2 (1-4\eta^2 L^2)}  \InNorms{z_0 - z^*}^2 + \weiqiangnote{ \frac{6\eta^2 L^2 + 16\eta^4 L^4}{\eta^2 L^2(1- 4\eta^2L^2)} \InNorms{w_0 - z_0}^2}}
\end{align*}
}
\end{theorem}

Setting $D = \max\{ \InNorms{z_0 - z^*}, \InNorms{w_0 - z_0}\}$, and $\eta = \frac{1}{2\sqrt{2}L}$, we have $\gap_{\Z,F,D}(z_T)$ (or $\gap_{\Z,F,D}(w_T))$ $= O \InParentheses{ \frac{D^2L}{\sqrt{T}}}$, which matches the lower bound of $\Omega(\frac{D^2L}{\sqrt{T}})$~by \cite{golowich_tight_2020}.

\bibliographystyle{plainnat}
\bibliography{references_new,ref}

\appendix
\section{Additional Preliminaries}\label{sec:additional prelim}
For $z\in \R^n$ and $D>0$, we use $\mathcal{B}(z,D)=\{z'\in R^n: \InNorms{z'-z}\leq R\}$ to denote the ball of radius $D$, centered at $z$.

\paragraph{Min-Max Saddle Points.} A special case of the variational inequality problem is the  constrained min-max problem $\min_{x\in \X} \max_{y\in \Y} f(x,y)$, where $\X$ and $\Y$ are closed convex sets in $\R^n$, and $f(\cdot,\cdot)$ is smooth, convex in $x$, and concave in $y$. It is well known that if one set $F(x,y) = \begin{pmatrix}
  \nabla_x f(x,y)\\
  -\nabla_y f(x,y)
\end{pmatrix}$,
then $F(x,y)$ is a monotone and Lipschitz operator~\citep{facchinei_finite-dimensional_2007}.

\paragraph{Equilibria of Monotone Games. } Monotone games are a large class of multi-player games that include many common and well-studied class of games such as bilinear games, $\lambda$-cocoercive games~\citep{lin_finite-time_2021}, zero-sum polymatrix games~\citep{cai_minmax_2011, cai_zero-sum_2016}, and zero-sum socially-concave games~\citep{even-dar_convergence_2009}. Besides, the min-max saddle point problem is a special case of two-player monotone games. We include the definition of monotone games here and remind readers that  finding a Nash Equilibrium of a monotone game is exactly the same as finding a solution to a monotone variational inequality. 

A continuous game $\mathcal{G}$ is  denoted as $(\mathcal{N}, (\X_i)_{i \in [N]}, (f_i)_{i \in [N]})$ where there are $N$ players $\mathcal{N} = \{1,\cdots,N\}$. Player $i\in\mathcal{N}$ chooses action from a closed convex set $\X_i \in \R^{n_i}$ such that $\X :=\Pi_{i\in\mathcal{N}}\X_i \in \R^n$ and wants to minimize its cost function $f_i: \X \rightarrow \R$. For each player $i$, we denote $x_{-i}$ the vector of actions of all the other players. A \emph{Nash Equilibrium} of game $\mathcal{G}$ is an action profile $x^* \in \X$ such that $f_i(x^*) \le f_i(x_i',x^*_{-i})$ for any $x_i' \in \X_i$. Let $F(x) = (\nabla_{x_i} f_i(x), \cdots, \nabla_{x_N} f_N(x)) \in \R^n$. We say $\mathcal{G}$ is \emph{monotone} if $\InAngles{F(x) - F(x'), x - x'} \ge 0$ for any $x, x' \in \X$. 

\notshow{
\begin{figure}[h!]
\colorbox{MyGray}{
\begin{minipage}{\textwidth} {
\small
\noindent\textbf{Input Fixed Polynomials.} 
\begin{itemize}
\item $S_= =\{g_i(\bx)\}_{i\in[M]}$.
\item $S_\leq =\{h_i(\bx)\}_{i\in[M']}$.

\item  $q_{i,j}(\boldsymbol{x},\by):=(y_i-y_j)(x_j-x_i)$ for any $0\leq j<i\leq 2$.
\end{itemize}
\noindent\textbf{Objective of the SOS Program:}
\noindent\textbf{Find polynomials:}
\begin{itemize}
\item $\lambda_{i,j} \ge 0$, and $\mu_{i,j} \ge 0$, for all $0\leq j<i\leq 2$.
\item $\varphi_1(\boldsymbol{x},\by)$ and $\varphi_2(\boldsymbol{x},\by)$ are two degree $1$ polynomials in $\mathbb{R}[\boldsymbol{x},\by]$.
\end{itemize}
\textbf{Constraints of the SOS Program:}
\begin{equation*}
\begin{array}{ll@{}ll}
\text{s.t.}  & \displaystyle y_2^2-y_0^2+\sum_{2\geq i>j\geq 0} \lambda_{i,j}\cdot p_{i,j}(\boldsymbol{x},\by)+\sum_{2\geq i>j\geq 0}\mu_{i,j}\cdot  q_{i,j}(\boldsymbol{x},\by)\\&\qquad\qquad\qquad\qquad\qquad\qquad+\varphi_1(\boldsymbol{x},\by)u_1(\boldsymbol{x},\by)+\varphi_2(\boldsymbol{x},\by)u_2(\boldsymbol{x},\by)\in \sos[\boldsymbol{x},\by].
\end{array}
\end{equation*}}
\end{minipage}} 
\caption{Our SOS program in the unconstrained setting.}\label{fig:template sos 2}
\end{figure}
}

In Lemma~\ref{lem:equivalent hamiltonian} we present several equivalent formulations of the \projhamnospace. 

\begin{lemma}\label{lem:equivalent hamiltonian}
Let $\Z$ be a closed convex set and $F:\Z\rightarrow \R^n$ be an operator. Denote $\normal_\Z(z)$ the normal cone of $z$ and $J_\Z(z) := \{z\} + T_\Z(z)$, where $T_\Z(z) = \{ z'\in \R^n: \InAngles{z',a} \le 0, \forall a \in \normal_{\Z}(z)\}$ is the tangent cone of $z$. Then all of the following quantities are equivalent: 
\begin{enumerate}
    \item $\sqrt{ \|F(z)\|^2 - \max_{\substack{a\in \unitnormal_\Z(z),\\ \InAngles{F(z),a}\le 0}}\langle F(z), a \rangle ^2}$
    \item $\min_{\substack{a\in \unitnormal_\Z(z),\\ \InAngles{F(z),a}\le 0}} \InNorms{F(z)-\InAngles{F(z),a}\cdot a}$
    \item $\displaystyle\InNorms{\Pi_{T_\Z(z)}\inteval{-F(z)}}$
    \item $\displaystyle\InNorms{\Pi_{J_\Z(z)}\inteval{z - F(z)} - z}$
    \item $\displaystyle\InNorms{-F(z) - \Pi_{\normal_\Z(z)}\inteval{-F(z)}}$
    \item $\displaystyle\min_{a \in \normal_{\Z}(z)} \InNorms{ F(z) + a}$
\end{enumerate}
\end{lemma}

\begin{proof}
\noindent{\bf (quantity 1 = quantity 2).}
Observe that
$$
\min_{\substack{a\in \unitnormal_\Z(z),\\ \InAngles{F(z),a}\le 0}}\InNorms{F(z)-\InAngles{F(z),a}\cdot a}^2
= \InNorms{F(z)}^2
- \max_{\substack{a\in \unitnormal_\Z(z),\\ \InAngles{F(z),a}\le 0}}\InAngles{F(z),a}^2\cdot\InParentheses{2-\InNorms{a}^2}.
$$


Therefore, it is enough to show that $\max_{\substack{a\in \unitnormal_\Z(z),\\ \InAngles{F(z),a}\le 0}}\InAngles{F(z),a}^2\cdot\InParentheses{2-\InNorms{a}^2} =\max_{\substack{a\in \unitnormal_\Z(z),\\ \InAngles{F(z),a}\le 0}}\langle F(z), a \rangle ^2$.
If $\unitnormal_\Z(z)=\{(0,\ldots,0)\}$,
then the equality holds trivially.
Now we assume that $\{(0,\ldots,0)\} \subsetneq \unitnormal_\Z(z)$ and consider any $a\in \unitnormal_\Z(z)\backslash (0,\ldots,0)$.
Let $c\in\left[1,\frac{1}{\InNorms{a}}\right]$.
By Definition~\ref{def:normal cone}, $\InNorms{a}\leq 1$,
which implies that $c\cdot a \in \unitnormal_\Z(z)$.
We try to maximize the following objective
$$
\InAngles{F(z),c\cdot a}^2\cdot\InParentheses{2-c^2\InNorms{a}^2}
=\frac{\InAngles{F(z), a}^2}{\InNorms{a}^2}\cdot c^2\InNorms{a}^2\cdot\InParentheses{2-c^2\InNorms{a}^2}.
$$

One can easily verify that function $c^2\InNorms{a}^2\cdot\InParentheses{2-c^2\InNorms{a}^2}$ is maximized when $c^2\InNorms{a}^2=1\Leftrightarrow c=\frac{1}{\InNorms{a}}$.
Thus when $\{(0,\ldots,0)\} \subsetneq \unitnormal_\Z(z)$,
\begin{align*}
\max_{\substack{a\in \unitnormal_\Z(z),\\ \InAngles{F(z),a}\le 0}}\InAngles{F(z),a}^2\cdot\InParentheses{2-\InNorms{a}^2}
= &
\max_{\substack{a\in \unitnormal_\Z(z),\\ \InAngles{F(z),a}\le 0,\\ \InNorms{a}=1}}\InAngles{F(z),a}^2\cdot\InParentheses{2-\InNorms{a}^2}\\
= &
\max_{\substack{a\in \unitnormal_\Z(z),\\ \InAngles{F(z),a}\le 0,\\ \InNorms{a}=1}}\InAngles{F(z),a}^2\\
= &
\max_{\substack{a\in \unitnormal_\Z(z),\\ \InAngles{F(z),a}\le 0}}\InAngles{F(z),a}^2,
\end{align*}
which concludes the proof.

\noindent{\bf (quantity 3 = quantity 4). }By definition, $J_\Z(z) = \{z\} + T_\Z(z) $. Thus we have 
\begin{align*}
    \InNorms{\Pi_{J_\Z(z)}\inteval{z - F(z)} - z} = \InNorms{\Pi_{T_\Z(z)}\inteval{- F(z)}}.
\end{align*}

\noindent{\bf (quantity 4 = quantity 5).} By definition, the tangent cone $T_\Z(z)$ is the polar cone of the normal cone $\normal_\Z(z)$. Since $\normal_\Z(z)$ is a closed convex cone, by Moreau's decomposition theorem, we have for any vector $x \in \R^n$, 
\begin{align*}
    x = \Pi_{\normal_\Z(z)}(x) + \Pi_{T_\Z(z)}(x), \qquad \InAngles{\Pi_{\normal_\Z(z)}(x), \Pi_{T_\Z(z)}(x)} = 0.
\end{align*}
Thus it is clear that we have 
\begin{align*}
    \InNorms{\Pi_{J_\Z(z)}\inteval{z - F(z)} - z} &= \InNorms{\Pi_{T_\Z(z)}\inteval{- F(z)}} \\
    & = \InNorms{-F(z) - \Pi_{\normal_\Z(z)}\inteval{-F(z)}}.
\end{align*}

\noindent{\bf (quantity 5 = quantity 6). }
Denote $a^* := \Pi_{\normal_\Z(z)}\inteval{-F(z)}$. By definition of projection, we have 
\begin{align*}
    a^* = \argmin_{a \in \normal_\Z(z)} \InNorms{F(z)+ a}^2.
\end{align*}
Thus 
\begin{align*}
    \InNorms{-F(z) - \Pi_{\normal_\Z(z)}\inteval{-F(z)}}^2 = \InNorms{F(z) + a^*}^2 = \min_{a \in \normal_\Z(z)} \InNorms{F(z)+ a}^2.
\end{align*}

\noindent{\bf (quantity 6 = quantity 2).}
For any fix non-zero $a\in \normal_\Z(z)$,  (i) if $\InAngles{F(z),a}>0$, then $ \InNorms{F(z)+ a}^2\geq \InNorms{F(z)}^2$, and (ii) if  $\InAngles{F(z),a}\leq 0$, $\InNorms{F(z)+ a}^2\geq \InNorms{F(z)-\InAngles{F(z),\frac{a}{\InNorms{a}}}\cdot \frac{a}{\InNorms{a}}}^2$, as
$$\min_{r\geq 0} \InNorms{F(z)+r\cdot a}^2=\InNorms{F(z)-\InAngles{F(z),\frac{a}{\InNorms{a}}}\cdot \frac{a}{\InNorms{a}}}^2.$$
Hence, $$\min_{a \in \normal_\Z(z)} \InNorms{F(z)+ a}^2=\min_{\substack{a'\in\unitnormal_\Z(z)\\ \InAngles{F(z),a'}\leq 0}}\InNorms{F(z)-\InAngles{F(z),a'} \cdot a'}^2
$$
The first equality is because for any $a\in \normal_\Z(z)$, there exists $a'\in \unitnormal_\Z(z)$ so that  $\InNorms{F(z)+a}^2\geq \InNorms{F(z)-\InAngles{F(z),a'}}^2$. 
\end{proof}

In the following Lemma, we show a useful property of the tangent residual that we use repeatedly.
\begin{lemma}\label{lem:upper bound tangent residual}
Let $\Z \subseteq \R^n$ be a closed convex set and $F: \Z \rightarrow \R$ be an operator. Let $\eta > 0$ and $z_1, z_2, z_3 \in \Z$ be three points such that $ z_1 = \Pi_\Z [z_2 - \eta F(z_3)]$, then we have 
\begin{align*}
    \Ham(z_1) \le \InNorms{\frac{z_2 - z_1}{\eta} + F(z_1) - F(z_3)}.
\end{align*}
\end{lemma}
\begin{proof}
If $z_1 = z_2 - \eta F(z_3)$, then the lemma holds since
\begin{align*}
    \Ham(z_1) \le \InNorms{F(z_1)} = \InNorms{F(z_3) + F(z_1) - F(z_3)} = \InNorms{\frac{z_2-z_1}{\eta} + F(z_1) - F(z_3)}.
\end{align*}
For the rest of the proof, we assume that $z_1 \ne z_2-\eta F(z_3)$. 
Since $z_2 - \eta F(z_3) - z_1 \in \normal(z_1)$, we have 
\begin{align*}
    \InAngles{z_2 -\eta F(z_3) - z_1, z - z_1} \le 0, \quad \forall z \in \Z.
\end{align*}
Define $a := \frac{z_2 -\eta F(z_3) - z_1}{\InNorms{z_2 -\eta F(z_3) - z_1}}$. We thus know $a \in \unitnormal(z_1)$. Let 
$$a_\perp :=
\begin{cases}
\frac{F(z_1)-  \langle a, F(z_1)\rangle \cdot a}{\InNorms{F(z_1)-  \langle a, F(z_1)\rangle \cdot a}} \qquad & \text{if $\InNorms{F(z_1)-  \langle a, F(z_1)\rangle \cdot a}\neq 0$,}\\
\InParentheses{0,\ldots,0} & \text{otherwise.}
\end{cases}
$$
Observe that $\langle a, a\rangle = 1$, $\langle a_\perp, a\rangle = 0$ and $F(z_1) =\langle a, F(z_1)\rangle a + \langle a_\perp, F(z_1)\rangle a_\perp $. Thus
\begin{align}
&0=\langle a_\perp, a\rangle =\langle a_\perp, z_2 -  \eta F(z_3) -z_1\rangle\nonumber \\
\Leftrightarrow &\langle a_\perp, F(z_3)\rangle = \frac{\langle a_\perp, z_2 -z_1\rangle}{\eta}.
\label{eq:penpedicular part-1}
\end{align}
Moreover, the fact that
$\InAngles{a, z_1 - z_2 + \eta F(z_3)} \le 0$ implies $ \InAngles{a, F(z_3)} \le \frac{\InAngles{a, z_2 - z_1}}{\eta}
$, which further implies that
 \begin{align}
&\InAngles{a, F(z_1)} = \InAngles{a, F(z_3)+ F(z_1) - F(z_3)} \leq  \InAngles{a, \frac{z_2-z_1}{\eta} + F(z_1) - F(z_3)}. 
\label{eq:normal part-1}
\end{align}
Combining the definition of $\Ham(z_1)$, Equation~\eqref{eq:penpedicular part-1}, and Equation~\eqref{eq:normal part-1} we have 
\begin{align*}
    \Ham(z_1)^2 &\le ||F(z_1)||^2 - \InAngles{a, F(z_1)}^2\cdot \mathbbm{1}[\InAngles{a, F(z_1)} \leq 0] \\
    & = \InAngles{a_\perp, F(z_1)}^2 + \InAngles{a, F(z_1)}^2\cdot \mathbbm{1}[\InAngles{a, F(z_1)} > 0] \\
    & = \InAngles{a_\perp, F(z_3) + F(z_1)-F(z_3)}^2 + \InAngles{a, F(z_1)}^2\cdot \mathbbm{1}[\InAngles{a, F(z_1)} > 0] \\
    & \le \InAngles{a_\perp, \frac{z_2 - z_1}{\eta} + F(z_1) - F(z_3) }^2 + \InAngles{a, \frac{z_2 - z_1}{\eta} + F(z_1) - F(z_3) }^2 \\
    & \leq \InNorms{\frac{z_2 - z_1}{\eta} + F(z_1) - F(z_3)}^2. \qedhere
\end{align*}
\end{proof}

\notshow{
\begin{lemma}\label{lem:equivalent hamiltonian}
Given an operator $F:\Z\rightarrow \R^n$ and a closed convex set $\Z$,
then
$$H_{(F,\Z)}(z)
=\min_{\substack{a\in \unitnormal_\Z(z),\\ \InAngles{F(z),a}\le 0}} \InNorms{F(z)-\InAngles{F(z),a}\cdot a}^2. 
$$
\end{lemma}
\begin{proof}
Observe that
$$
\min_{\substack{a\in \unitnormal_\Z(z),\\ \InAngles{F(z),a}\le 0}}\InNorms{F(z)-\InAngles{F(z),a}\cdot a}^2
= \InNorms{F(z)}^2
- \max_{\substack{a\in \unitnormal_\Z(z),\\ \InAngles{F(z),a}\le 0}}\InAngles{F(z),a}^2\cdot\InParentheses{2-\InNorms{a}^2}.
$$
We remind readers that
$H_{(F,\Z)}(z):=
\|F(z)\|^2 - \max_{\substack{a\in \unitnormal_\Z(z),\\ \InAngles{F(z),a}\le 0}}\langle F(z), a \rangle ^2
$.
Therefore, it is enough to show that $\max_{\substack{a\in \unitnormal_\Z(z),\\ \InAngles{F(z),a}\le 0}}\InAngles{F(z),a}^2\cdot\InParentheses{2-\InNorms{a}^2} =\max_{\substack{a\in \unitnormal_\Z(z),\\ \InAngles{F(z),a}\le 0}}\langle F(z), a \rangle ^2$.
If $\unitnormal_\Z(z)=\{(0,\ldots,0)\}$,
then the equality holds trivially.
Now we assume that $\{(0,\ldots,0)\} \subsetneq \unitnormal_\Z(z)$ and consider any $a\in \unitnormal_\Z(z)\backslash (0,\ldots,0)$.
Let $c\in\left[1,\frac{1}{\InNorms{a}}\right]$.
By Definition~\ref{def:normal cone}, $\InNorms{a}\leq 1$,
which implies that $c\cdot a \in \unitnormal_\Z(z)$.
We try to maximize the following objective
$$
\InAngles{F(z),c\cdot a}^2\cdot\InParentheses{2-c^2\InNorms{a}^2}
=\frac{\InAngles{F(z), a}^2}{\InNorms{a}^2}\cdot c^2\InNorms{a}^2\cdot\InParentheses{2-c^2\InNorms{a}^2}.
$$

One can easily verify that function $c^2\InNorms{a}^2\cdot\InParentheses{2-c^2\InNorms{a}^2}$ is maximized when $c^2\InNorms{a}^2=1\Leftrightarrow c=\frac{1}{\InNorms{a}}$.
Thus when $\{(0,\ldots,0)\} \subsetneq \unitnormal_\Z(z)$,
\begin{align*}
\max_{\substack{a\in \unitnormal_\Z(z),\\ \InAngles{F(z),a}\le 0}}\InAngles{F(z),a}^2\cdot\InParentheses{2-\InNorms{a}^2}
= &
\max_{\substack{a\in \unitnormal_\Z(z),\\ \InAngles{F(z),a}\le 0,\\ \InNorms{a}=1}}\InAngles{F(z),a}^2\cdot\InParentheses{2-\InNorms{a}^2}\\
= &
\max_{\substack{a\in \unitnormal_\Z(z),\\ \InAngles{F(z),a}\le 0,\\ \InNorms{a}=1}}\InAngles{F(z),a}^2\\
= &
\max_{\substack{a\in \unitnormal_\Z(z),\\ \InAngles{F(z),a}\le 0}}\InAngles{F(z),a}^2,
\end{align*}
which concludes the proof.
\end{proof}
}

\notshow{
In Lemma~\ref{lem:equivalent hamiltonian-2} we present another equivalent formulation of the \projham. It shows that the \projham can be seen as the squared norm of $-F(z)$ projected onto the intersection of all the tight halfspaces at $z$. 
\begin{lemma}\label{lem:equivalent hamiltonian-2}
Let $\Z$ be a closed convex set and $F:\Z\rightarrow \R^n$ be an operator. Denote $J_\Z(z) := z + T_\Z(z)$, where $T_\Z(z) = \{ z'\in \R^n: \InAngles{z',a} \le 0, \forall a \in \normal_{\Z}(z)\}$ is the tangent cone of $z$. Then all of the following quantities are equivalent: 
\argyrisnote{
\begin{itemize}
    \item $r^{tan}_{(F,\Z)}(z)$
    \item $\displaystyle\InNorms{\Pi_{J_\Z(z)}\inteval{z - F(z)} - z}$
    \item $\displaystyle\InNorms{-F(z) - \Pi_{\normal_\Z(z)}\inteval{-F(z)}}$
    \item $\displaystyle\min_{a \in \normal_{\Z}(z)} \InNorms{ F(z) + a}$
\end{itemize}
}
\end{lemma}
\begin{proof}
By definition, we have 
\begin{align*}
    \InNorms{\Pi_{J_\Z(z)}\inteval{z - F(z)} - z}^2 = \InNorms{\Pi_{T_\Z(z)}\inteval{- F(z)}}^2.
\end{align*}
By definition, the tangent cone $T_\Z(z)$ is the polar cone of the normal cone $\normal_\Z(z)$. Since $\normal_\Z(z)$ is a closed convex cone, by Moreau's decomposition theorem, we have for any vector $x \in \R^n$, 
\begin{align*}
    x = \Pi_{\normal_\Z(z)}(x) + \Pi_{T_\Z(z)}(x), \qquad \InAngles{\Pi_{\normal_\Z(z)}(x), \Pi_{T_\Z(z)}(x)} = 0.
\end{align*}
Thus it is clear that we have 
\begin{align*}
    \InNorms{\Pi_{J_\Z(z)}\inteval{z - F(z)} - z}^2 &= \InNorms{\Pi_{T_\Z(z)}\inteval{- F(z)}}^2 \\
    & = \InNorms{-F(z) - \Pi_{\normal_\Z(z)}\inteval{-F(z)}}^2.
\end{align*}

Now it suffices to show the following holds \[\min_{a \in \normal_\Z(z)} \InNorms{F(z)+ a}^2 = \InNorms{-F(z) - \Pi_{\normal_\Z(z)}\inteval{-F(z)}}^2 = \argyrisnote{r^{tan}_{F,\Z}(z)^2}.\]
Denote $a^* := \Pi_{\normal_\Z(z)}\inteval{-F(z)}$. By definition of projection, we have 
\begin{align*}
    a^* = \argmin_{a \in \normal_\Z(z)} \InNorms{F(z)+ a}^2.
\end{align*}
Thus 
\begin{align*}
    \InNorms{-F(z) - \Pi_{\normal_\Z(z)}\inteval{-F(z)}}^2 = \InNorms{F(z) + a^*}^2 = \min_{a \in \normal_\Z(z)} \InNorms{F(z)+ a}^2.
\end{align*}

For any fix non-zero $a\in \normal_\Z(z)$,  (i) if $\InAngles{F(z),a}>0$, then $ \InNorms{F(z)+ a}^2\geq \InNorms{F(z)}^2$, and (ii) if  $\InAngles{F(z),a}\leq 0$, $\InNorms{F(z)+ a}^2\geq \InNorms{F(z)-\InAngles{F(z),\frac{a}{\InNorms{a}}}\cdot \frac{a}{\InNorms{a}}}^2$, as
$$\min_{r\geq 0} \InNorms{F(z)+r\cdot a}^2=\InNorms{F(z)-\InAngles{F(z),\frac{a}{\InNorms{a}}}\cdot \frac{a}{\InNorms{a}}}^2.$$
Hence, $$\min_{a \in \normal_\Z(z)} \InNorms{F(z)+ a}^2=\min_{\substack{a'\in\\unitnormal_\Z(z)\\ \InAngles{F(z),a'}\leq 0}}\InNorms{F(z)-\InAngles{F(z),a'} \cdot a'}^2=\argyrisnote{r^{tan}_{(F,\Z)}(z)^2}.$$
The first equality is because for any $a\in \normal_\Z(z)$, there exists $a'\in \\unitnormal_\Z(z)$ so that  $\InNorms{F(z)+a}^2\geq \InNorms{F(z)-\InAngles{F(z),a'}}^2$; the second equality follows from Lemma~\ref{lem:equivalent hamiltonian}.

\notshow{
By definition of projection, we have 
\begin{align*}
    \InAngles{-F(z) - a^*, a - a^*} \le 0, \forall a \in \normal_{\Z}(z).
\end{align*}
By taking $a = 0$ and $a = 2a^*$ in the above inequality, we get
\begin{align*}
    \InAngles{ a^*,  a^* + F(z)} = \InNorms{a^*}^2 + \InAngles{a^*, F(z)} = 0.
\end{align*}

\paragraph{\bf Case 1: $\InAngles{a, F(z)} \ge 0, \forall a \in \normal_{\Z}(z)$.} In this case, it is clear that $a^* = 0$. Thus $$\min_{a \in \normal_{\Z}(z)} \InNorms{a + F(z)}^2 = \InNorms{F(z)}^2 = \argyrisnote{r^{tan}_{F,\Z}(z)^2}.$$

\paragraph{\bf Case 2: there exits $a \in \normal_{\Z}(z)$ such that $\InAngles{a, F(z)} < 0$.} Thus we know $a^* \ne 0$. Define set $S := \{a \in C_{\Z}(z): \InAngles{a, F(z)} \leq 0, a \ne 0\}$. Based on the optimality condition above, we have $\InNorms{a^*} = -\InAngles{\frac{a^*}{\InNorms{a^*}}, F(z)}$. Then we have

\begin{align*}
     \min_{a \in S} \InNorms{a + F(z)}^2 &= \min_{a \in S} \InNorms{\frac{a}{\InNorms{a}} \InNorms{a} + F(z)}^2 \\
     &= \min_{a \in S} \InNorms{- \InAngles{\frac{a}{\InNorms{a}}, F(z)} \frac{a}{\InNorms{a}} + F(z)}^2 \\
     & = \min_{a \in \unitnormal_{\Z}(z), \InAngles{a,F(z)}\leq 0} \InNorms{F(z) - \InAngles{a,F(z)}a}^2 \\
     & = H_{(F,\Z)}(z). \tag{\Cref{lem:equivalent hamiltonian}}
\end{align*}
This completes the proof.}

\end{proof}
}
\begin{prevproof}{Lemma}{lem:hamiltonian to gap}
If $\InAngles{a,F(z)} \ge 0$ for all $a \in \unitnormal(z)$, then we have $\Ham(z) = \|F(z)\|$. Thus for any $z' \in \Z$, by Cauchy-Schwarz inequality, we have
\begin{align*}
    \langle F(z),z-z' \rangle \le \|F(z)\| \|z-z'\| \le D  \cdot \Ham(z).
\end{align*}
Otherwise there exists $a \in \unitnormal(z)$ such that $\InNorms{a} = 1$, $\InAngles{a,F(z)} < 0$ and $\Ham(z) = \sqrt{\|F(z)\|^2 - \InAngles{a,F(z)}^2} = \|F(z)-\InAngles{a,F(z)}a\|$. Then for any $z'\in \Z$, we have
\begin{align*}
    \InAngles{F(z),z-z'} &= \InAngles{F(z)-\InAngles{a,F(z)}a,z-z'} + \InAngles{a,F(z)}\cdot \InAngles{a,z-z'}\\
    &\le \InAngles{F(z)-\InAngles{a,F(z)}a,z-z'}\\
    &\le \|F(z)-\InAngles{a,F(z)}a\| \|z-z'\| \\
    &\le D \cdot \Ham(z),
\end{align*}
where we use $\InAngles{a,F(z)} < 0$ and $\InAngles{a,z- z'} \ge 0$ in the first inequality and Cauchy-Schwarz inequality in the second inequality. 

If $\Z' = \X' \times \Y'$ and $F(x,y) = \begin{pmatrix}
  \nabla_x f(x,y)\\
  -\nabla_y f(x,y)
\end{pmatrix}$ for a convex-concave function $f$  then
\begin{align*}
    \dg^{\X',\Y'}_f(z) &= \max_{y'\in \Y'} f(x,y')- \min_{x'\in \X'} f(x',y) \\
    &= \max_{y'\in \Y'} (f(x,y') - f(x,y)) - \min_{x'\in \X'} (f(x',y)-f(x,y)), \\
    &\le \max_{y'\in \Y'} \InAngles{\nabla_y f(x,y),y'-y} +\max_{x'\in \X'} \InAngles{\nabla_x f(x,y),x-x'},\\
    &= \max_{z'\in \Z'} \InAngles{F(z),z-z'},\\
    &\le D \sqrt{2} \cdot \Ham(z), 
\end{align*}
where we use the fact that $f$ is a convex-concave function in the first inequality and $\|z-z'\| = \sqrt{\InNorms{x-x'}^2+\InNorms{y-y'}^2} \le \sqrt{2}D$ in the second inequality.
\end{prevproof}

\notshow{
\begin{proposition}\label{prop:compactness of normals}
$\unitnormal_{\mathcal{Z}}(z)$ as defined in Definition~\ref{def:normal cone} is a compact set for any $z\in \mathcal{Z}$.
\end{proposition}
\begin{proof}
We fix $z\in \Z$ for the rest of the proof.
First we show that $\unitnormal_\Z(z)$ is a compact set.
Let $\{a_t\in \unitnormal_\Z(z)\}_{t\geq 0}$ be any Cauchy Sequence and $a^* = \lim_{t\rightarrow +\infty}a_t$.
In order to show that $\unitnormal_\Z(z)$ is compact, it is enough to show that $a^*\in \unitnormal_\Z(z)$.
Note that if $a^*\notin \unitnormal_\Z(z)$ then $\|a\| \neq 1$ or there exists $\hz\in Z: \InAngles{a^*,\hz}>\InAngles{a^*,z}$.

Assume towards a contradiction that $\|a^*\| - 1 = \epsilon > 0$.
There exists sufficiently large $t^*> 0$ s.t. $\|a^* - a_{t^*}\| \leq \epsilon/2$. Using triangle inequality $\|a_{t^*}\| \geq \|a^*\| - \epsilon/2 = 1+\epsilon/2$ a contradiction since $a_{t^*}\in \unitnormal_\Z(z)$.
We can similarly show that assuming $\|a^*\| > 1$ leads to a contradictions.

Assume now that there exists $\hz\in Z: \InAngles{a^*,\hz - z}=\epsilon > 0$.
There exists sufficiently large $t^*> 0$ s.t. $\|a^* - a_{t^*}\| \leq \frac{\epsilon}{2\|\hz-z\|} $.
Thus $\InAngles{a_{t^*},\hz - z} \geq \InAngles{a^*,\hz - z} - \epsilon/2 = \epsilon/2$,
a contradiction.

Therefore $a^*\in \unitnormal_\Z(z)$,
which implies that set $\unitnormal_\Z(z)$ is compact.
\end{proof}
}

\section{Missing Proofs from Section~\ref{sec:best iterate}}\label{appx:best iterate}

\begin{prevproof}{Lemma}{lem:EG Best-iterate}
By Pythagorean inequality,
\begin{align}
   \|z_{k+1}-z^*\|^2 &\leq  \|z_k-\eta F(z_{k+\half})-z^*\|^2 - \|z_k-\eta F(z_{k+\half})-z_{k+1}\|^2 \notag\\
   &= \|z_k-z^*\|^2-\|z_k-z_{k+1}\|^2+2\eta\langle F(z_{k+\half}), z^*-z_{k+1}\rangle\notag\\
   &= \|z_k-z^*\|^2-\|z_k-z_{k+1}\|^2+2\eta\langle F(z_{k+\half}), z^*-z_{k+\half}\rangle+2\eta\langle F(z_{k+\half}), z_{k+\half}-z_{k+1}\rangle. \label{eq:EG first step}
\end{align}
 
We first use monotonicity of $F(\cdot)$ to argue that $\langle F(z_{k+\half}), z^*-z_{k+\half}\rangle\leq 0$.
 
\begin{fact}\label{fact:monotonicity implication}
  For all $z\in \Z$, $\langle F(z), z^*-z\rangle\leq 0$.
\end{fact}
\begin{proof}
\begin{align*}
     0 &\leq \langle F(z^*)- F(z), z^*-z\rangle \qquad &(\text{monotonicity of $F(\cdot)$})\\
     &=  \langle F(z^*), z^*-z\rangle -  \langle F(z), z^*-z\rangle &\\
     & \leq - \langle F(z), z^*-z\rangle \qquad &(\text{optimality of $z^*$ and $z\in \Z$})
 \end{align*}
\end{proof}

We can simplify \Cref{eq:EG first step} using \Cref{fact:monotonicity implication}:
 \begin{align}
        \|z_{k+1}-z^*\|^2 &\leq \|z_k-z^*\|^2-\|z_k-z_{k+1}\|^2+2\eta\langle F(z_{k+\half}), z_{k+\half}-z_{k+1}\rangle \notag\\
        &= \|z_k-z^*\|^2-\|z_k-z_{k+\half}\|^2-\|z_{k+\half}-z_{k+1}\|^2-2\langle z_k-\eta F(z_{k+\half})-z_{k+\half}, z_{k+\half}-z_{k+1}\rangle \notag\\
        &= \|z_k-z^*\|^2-\|z_k-z_{k+\half}\|^2-\|z_{k+\half}-z_{k+1}\|^2 \notag\\ 
        &\qquad \qquad -2\langle z_k-\eta F(z_{k})-z_{k+\half}, z_{k+\half}-z_{k+1}\rangle - 2\langle \eta F(z_{k})-\eta F(z_{k+\half}), z_{k+\half}-z_{k+1}\rangle \notag\\
        &\leq \|z_k-z^*\|^2-\|z_k-z_{k+\half}\|^2-\|z_{k+\half}-z_{k+1}\|^2 - 2\eta \langle  F(z_{k})- F(z_{k+\half}), z_{k+\half}-z_{k+1} \rangle \notag 
 \end{align}
 The last inequality is because $\langle z_k-\eta F(z_{k})-z_{k+\half}, z_{k+\half}-z_{k+1}\rangle \geq 0$, which follows from the that fact that $z_{k+\half}=\Pi_\Z[z_k-\eta F(z_{k})]$ and $z_{k+1}\in \Z$.
 
 Finally, since $F(\cdot)$ is $L$-Lipschitz, we know that \begin{equation*}
     -\langle  F(z_{k})- F(z_{k+\half}), z_{k+\half}-z_{k+1}\rangle \leq \| F(z_{k})- F(z_{k+\half})\|\cdot \|z_{k+\half}-z_{k+1}\|\leq L\| z_{k}-z_{k+\half}\|\cdot \|z_{k+\half}-z_{k+1}\|.
 \end{equation*}
 So we can further simplify the inequality as follows:
 \begin{align}
     \left\|z_{k+1}-z^*\right\|^2 \leq & \|z_k-z^*\|^2-\|z_k-z_{k+\half}\|^2-\|z_{k+\half}-z_{k+1}\|^2 - 2\eta \langle  F(z_{k})- F(z_{k+\half}), z_{k+\half}-z_{k+1}\rangle \notag \\
    \leq & \|z_k-z^*\|^2-\|z_k-z_{k+\half}\|^2-\|z_{k+\half}-z_{k+1}\|^2 + 2\eta L \| z_{k}-z_{k+\half}\|\cdot \|z_{k+\half}-z_{k+1}\| \notag \\
    \leq & \|z_k-z^*\|^2-(1-\eta^2 L^2)\|z_k-z_{k+\half}\|^2 \notag
 \end{align}
 
 Hence, \begin{align*}
     \left\|z_k-z^*\right\|^2\geq  \left\|z_{k+1}-z^*\right\|^2+(1-\eta^2 L^2) \|z_k-z_{k+\half}\|^2. 
 \end{align*} 
\end{prevproof}

\begin{prevproof}{Lemma}{lem:bound hamiltonian by distance}
We need the following fact for our proof.
\begin{fact}\label{fact:close}
$\InNorms{z_{k+\half} - z_{k+1}}\leq \eta L \InNorms{z_k-z_{k+\half}}$. Moreover, when $\eta L < 1$, $\InNorms{z_{k+\half} - z_{k+1}}\leq \frac{\InNorms{z_k-z_{k+1}}}{1 - \eta L}$.
\end{fact}

\begin{proof}
Recall that $z_{k+\frac{1}{2}} = \Pi_\Z\left[z_k-\eta F(z_k)\right]$ and $z_{k+1} = \Pi_\Z\left[z_k-\eta F(z_{k+\frac{1}{2}})\right]$.
By the non-expansiveness of the projection operator and the $L$-Lipschitzness of operator $F$,
we have that $\InNorms{z_{k+\half}-z_{k+1}}\leq \InNorms{\eta (F(z_{k+\half})- F(z_k))}\leq \eta L\InNorms{z_k - z_{k+\half}}$.

Finally, by the triangle inequality $$\InNorms{z_k-z_{k+1}}\geq \InNorms{z_k-z_{k+\half}} -\InNorms{z_{k+\half}-z_{k+1}}\geq\InParentheses{1 - \eta L}\InNorms{z_k-z_{k+\half}}.$$
\end{proof}

Now we prove Lemma~\ref{lem:bound hamiltonian by distance}.
By the $L$-Lipschitzness of operator $F$ we have
\begin{align}
\|F(z_{k+1})-F(z_{k+\half})\|\leq L \|z_{k+1}-z_{k+\half}\| \leq \eta L^2\|z_k-z_{k+\half}\|.\label{eq:lipschitz distance}
\end{align}
Recall that $z_{k+1} = \Pi_\Z\left[z_k-\eta F(z_{k+\frac{1}{2}})\right]$. Using \Cref{lem:upper bound tangent residual}, we have
\begin{align}
\Ham(z_{k+1})
\leq& \left\|\frac{z_k-z_{k+1}}{\eta}+F(z_{k+1})- F(z_{k+\frac{1}{2}})\right\| \notag \\
\leq & \frac{\|z_k-z_{k+1}\|}{\eta} + \|F(z_{k+1})- F(z_{k+\frac{1}{2}})\| \notag \\
 \leq & \frac{\|z_k-z_{k+1}\| + (\eta L)^2\|z_k-z_{k+\half}\|}{\eta} \notag\\ 
\leq & \frac{||z_k-z_{k+\half}|| + ||z_{k+\half}-z_{k+1}|| + (\eta L)^2||z_k-z_{k+\half}||}{\eta} \notag\\
\leq &\left(1+ \eta L + (\eta L)^2 \right)\frac{||z_k-z_{k+\half}||}{\eta}. \notag
\end{align}
The second and the fourth inequality follow from the triangle inequality.
The third inequality follows from Equation~\eqref{eq:lipschitz distance}.
In the final inequality we 
use $||z_{k+\half}-z_{k+1}||\leq \eta L ||z_k-z_{k+\half}||$ by Fact~\ref{fact:close}.
\notshow{
Note that if $z_{k+1} - z_k + \eta F(z_{k+\half}) =0 $,
then by Definition~\ref{def:projected hamiltonian} vector $(0,\ldots,0)\in \unitnormal(z_{k+1})$ and the proof still holds. Finally,
$$\sqrt{H(z_{k+1})} \leq \left(1+ \frac{(\eta L )^2}{1-\eta L} \right)\frac{||z_k-z_{k+1}||}{\eta}$$ follows from the combination of  Fact~\ref{fact:close} and Inequality~\eqref{eq:last iterate strongly monotone}.}
\notshow{
Since $z_{k+1} = \Pi_\Z\left[z_k-\eta F(z_{k+\frac{1}{2}})\right]$, for every $z\in \Z$ we have:
$$\langle z_k - \eta F(z_{k+\half}) - z_{k+1} , z_{k+1} \rangle \geq \langle z_k - \eta F(z_{k+\half}) - z_{k+1} , z \rangle.
$$
For the rest of the proof we assume that $z_k - \eta F(z_{k+\half}) - z_{k+1}\neq 0$, we show how to remove this assumption at the end of the proof.
Thus vector  $a = \frac{z_k - \eta F(z_{k+\half}) - z_{k+1} }{\InNorms{z_k - \eta F(z_{k+\half}) - z_{k+1}}} \in \unitnormal(z_{k+1})$.
Let

$$a_\perp =
\begin{cases}
\frac{F(z_{k+1})-  \langle a, F(z_{k+1})\rangle \cdot a}{\InNorms{F(z_{k+1})-  \langle a, F(z_{k+1})\rangle \cdot a}} \qquad & \text{if $\InNorms{F(z_{k+1})-  \langle a, F(z_{k+1})\rangle \cdot a}\neq 0$}\\
\InParentheses{0,\ldots,0} & \text{otherwise}
\end{cases}.
$$

Observe that $\langle a, a\rangle = 1$, $\langle a_\perp, a\rangle = 0$ and $F(z_{k+1}) =\langle a, F(z_{k+1})\rangle a + \langle a_\perp, F(z_{k+1})\rangle a_\perp $.

By Definition~\ref{def:projected hamiltonian} we have 

\begin{align}
\Ham(z_{k+1})\leq & \sqrt{||F(z_{k+1})||^2 - \langle a, F(z_{k+1})\rangle^2\cdot \mathbbm{1}[\langle a, F(z_{k+1})\rangle \leq 0]}
\nonumber\\
=&
\sqrt{\langle a_\perp, F(z_{k+1})\rangle^2
+
\langle a, F(z_{k+1})\rangle^2\cdot \mathbbm{1}[\langle a, F(z_{k+1})\rangle > 0]}.    \label{eq:bound hamiltonian by distance}
\end{align}

We first present a few inequalities that will help us to prove the claim. Let $\epsilon = F(z_{k+1})- F(z_{k+\half})$.
By Claim~\ref{fact:close} we have $\|z_{k+\half} - z_{k+1}\|\leq \eta L \|z_k-z_{k+\half}\|$ and 
by the $L$-Lipschitzness of operator $F$ we have
\begin{align}
\|\epsilon\|=\|F(z_{k+1})-F(z_{k+\half})\|\leq L \|z_{k+1}-z_{k+\half}\| \leq \eta L^2\|z_k-z_{k+\half}\|.\label{eq:lipschitz distance}
\end{align}

Note that
\begin{align}
&0=\langle a_\perp, a\rangle =\langle a_\perp, z_{k+1} -z_k +  \eta F(z_{k+\half})\rangle\nonumber \\
\Leftrightarrow &\langle a_\perp, F(z_{k+\half})\rangle = \frac{\langle a_\perp, z_k -z_{k+1}\rangle}{\eta}.
\label{eq:penpedicular part}
\end{align}
Moreover, the fact that
$0\geq\langle a, z_{k+1} - z_k + \eta F(z_{k+\half}) \rangle
\Leftrightarrow \frac{\langle a, z_k - z_{k+1} \rangle }{\eta}\geq\langle a, F(z_{k+\half}) \rangle
\label{eq:regular part a}
$, which further imiplies that
 \begin{align}
&\langle a, F(z_{k+1}) \rangle = \langle a, F(z_{k+\half})+ \epsilon \rangle \leq  \langle a, \frac{z_k-z_{k+1}}{\eta} \rangle + \langle a, \epsilon \rangle. 
\label{eq:normal part}
\end{align}

Combining Equation~\eqref{eq:bound hamiltonian by distance}, Equation~\eqref{eq:penpedicular part} and Equation~\eqref{eq:normal part} we have 

\begin{align}
\Ham(z_{k+1})
\leq& \left\|\frac{z_k-z_{k+1}}{\eta}+\epsilon\right\| \notag \\
\leq & \frac{\|z_k-z_{k+1}\|}{\eta} + \|\epsilon\| \notag \\
 \leq & \frac{\|z_k-z_{k+1}\| + (\eta L)^2\|z_k-z_{k+\half}\|}{\eta} \label{eq:last iterate strongly monotone} \\
\leq & \frac{||z_k-z_{k+\half}|| + ||z_{k+\half}-z_{k+1}|| + (\eta L)^2||z_k-z_{k+\half}||}{\eta} \notag\\
\leq &\left(1+ \eta L + (\eta L)^2 \right)\frac{||z_k-z_{k+\half}||}{\eta}. \notag
\end{align}
The first inequality is because (i) $\langle a_\perp, F(z_{k+1})\rangle=\langle a_\perp, \frac{z_k-z_{k+1}}{\eta}+\epsilon \rangle$ and (ii) $\langle a, F(z_{k+1})\rangle^2\cdot \ind[\langle a, F(z_{k+1})\rangle > 0]\leq \langle a, \frac{z_k-z_{k+1}}{\eta}+\epsilon \rangle^2$.
The second inequality is because $a$ and $a_\perp$ are both unit vectors, and they are orthogonal to each other when $\InNorms{F(z_{k+1})-  \langle a, F(z_{k+1})\rangle \cdot a}\neq 0$, and because $a_\perp$ is $\InParentheses{0,\ldots,0}$ otherwise. The third and the fifth inequality follow from the triangle inequality.
The fourth inequality follows from Equation~\eqref{eq:lipschitz distance}.
In the final inequality we 
use $||z_{k+\half}-z_{k+1}||\leq \eta L ||z_k-z_{k+\half}||$ by Fact~\ref{fact:close}.
\notshow{
Note that if $z_{k+1} - z_k + \eta F(z_{k+\half}) =0 $,
then by Definition~\ref{def:projected hamiltonian} vector $(0,\ldots,0)\in \unitnormal(z_{k+1})$ and the proof still holds. Finally,
$$\sqrt{H(z_{k+1})} \leq \left(1+ \frac{(\eta L )^2}{1-\eta L} \right)\frac{||z_k-z_{k+1}||}{\eta}$$ follows from the combination of  Fact~\ref{fact:close} and Inequality~\eqref{eq:last iterate strongly monotone}.}
}
\end{prevproof}

\begin{prevproof}{Lemma}{lem:best iterate hamiltonian}
By Lemma~\ref{lem:EG Best-iterate} we have 
$$
     \left\|z_0-z^*\right\|^2\geq  \left\|z_{T+1}-z^*\right\|^2+(1-\eta^2 L^2) \sum_{k=0}^T \|z_k-z_{k+\half}\|^2
\geq 
(1-\eta^2 L^2) \sum_{k=0}^T \|z_k-z_{k+\half}\|^2
$$
Thus there exists a $t^*\in [T]$ such that $\|z_{t^*}-z_{t^*+\half}\|^2 \leq \frac{\|z_0-z^*\|^2}{T(1-\eta ^2L^2)}$.
We conclude the proof by applying Lemma~\ref{lem:bound hamiltonian by distance}. \end{prevproof}

\section{Missing Proofs from Section~\ref{sec:warm up}}\label{appx:warm up}
\begin{proposition}\label{prop:verify unconstrained}
\begin{align*}
    \|F(z_k)\|^2 - \|F(z_{k+1})\|^2 &+ 2\cdot \InAngles{ F(z_{k+1}) - F(z_k) ,F(z_{k+\frac{1}{2}})} \\
    &+ \InParentheses{\InNorms{F(z_{k+\frac{1}{2}})-F(z_{k+1})}^2 -\InNorms{F(z_{k+\frac{1}{2}}) - F(z_k)}^2} = 0.
\end{align*}
\end{proposition}

\begin{proof}
Expanding the LHS of the equation in the statement we can verify that
\begin{align*}
    \|F(z_k)\|^2 - \|F(z_{k+1})\|^2 &+ 2\cdot \InAngles{ F(z_{k+1}) ,F(z_{k+\frac{1}{2}})} -  2\cdot \InAngles{ F(z_k) ,F(z_{k+\frac{1}{2}})} \\
    &+ \InNorms{F(z_{k+\frac{1}{2}})}^2 - 2\cdot \InAngles{ F(z_{k+1}) ,F(z_{k+\frac{1}{2}})}+\InNorms{F(z_{k+1})}^2 \\
    &-\InNorms{F(z_{k+\frac{1}{2}})}^2 + 2\cdot \InAngles{ F(z_k) ,F(z_{k+\frac{1}{2}})} - \InNorms{F(z_k)}^2 = 0.
\end{align*}
\end{proof}

\section{Non-Monotonicity of Several Standard Performance Measures}\label{sec:counterexamples}

We conduct numerical experiments by trying to find saddle points in constrained bilinear games using EG,
and verified that the following performance measures are not monotone:
the (squared) natural residual, $\InNorms{z_k - z_{k+\half}}^2$, $\InNorms{z_{k} - z_{k+1}}^2$, $\max_{z \in \Z} \InAngles{F(z), z_k - z}$, $\max_{z \in \Z} \InAngles{F(z_k), z_k - z}$. 

All of our counterexamples are constructed by trying to find a saddle point in bilinear games of the following form:

\begin{align}
    \min_{x\in \X} \max_{y \in \Y} x^\top A y - b^\top x - c^\top y\label{game:bilinear counterexample}
\end{align}
where $\X, \Y \subseteq \R^2$,
$A$ is a $2\times2$ matrix and $b,c$ are $2$-dimensional column vectors.
All of the instances of the bilinear game considered in this section have $\X, \Y = [0,10]^2$.
We denote by $\Z = \X\times \Y$ and by $F(x,y) =\begin{pmatrix} Ay - b\\ -A^\top x+c\end{pmatrix}:\Z\rightarrow \R^n$.
We remind readers that finding a saddle point of bilinear game \eqref{game:bilinear counterexample}, is equivalent to solving the monotone VI with operator $F(z)$ on set $\Z$.

\subsection{Non-Monotonicity of the Natural Residual and its Variants}

\paragraph{Performance Measure: Natural Residual.} Let $A = \begin{bmatrix}
1 & 2 \\
1 & 1
\end{bmatrix}$, $b=c = \begin{bmatrix}
1 \\ 1
\end{bmatrix}$.
Running the EG method on the corresponding VI problem with step-size $\eta = 0.1$ starting at $z_0 = (0.3108455, 0.4825575, 0.4621875, 0.5768655)^T$ has the following trajectory:
\begin{align*}
    &z_1 = (0.24923465, 0.47967569, 0.43497808, 0.57458145)^T,\\
    &z_2 = (0.19396855, 0.48164918, 0.40193211, 0.56061753)^T.
\end{align*}
Thus we have
\begin{align*}
    r^{nat}(z_0)^2 =& 0.15170013184049996, \\
    r^{nat}(z_1)^2 =& 0.13617654362050116, \\
    r^{nat}(z_2)^2 =& 0.16125792556139756.
\end{align*}
It is clear that the natural residual is not monotone.

\paragraph{Performance Measure: $\InNorms{z_k - z_{k+\half}}^2$.} Note that the norm of the operator mapping defined in~\citep{diakonikolas_halpern_2020} is exactly $\frac{1}{\eta}\cdot \InNorms{z_k - z_{k+\half}}$. Let $A = \begin{bmatrix}
0.50676631 & 0.15042569\\
0.46897595 & 0.96748026
\end{bmatrix} $, $b=c = \begin{bmatrix}
1 \\ 1
\end{bmatrix}$.
Running the EG method on the corresponding VI problem with step-size $\eta = 0.1$ starting at $z_0 = (2.35037432, 0.00333996, 1.70547279, 0.71065999)^T$ has the following trajectory:
\begin{align*}
    z_{\half} =& (2.35325656, 0 , 1.72473848, 0.64633879)^T, \\
    z_1 =& (2.35324779, 0, 1.72472791, 0.64605901)^T, \\
    z_{1+\half} =& (2.35612601, 0, 1.74398258, 0.58145791)^T \\
    z_2 =& (2.35612201, 0 , 1.74412844, 0.5815012)^T, \\
    z_{2+\half} =& (2.35898819, 0, 1.76352876, 0.51694333)^T.
\end{align*}
Thus we have
\begin{align*}
    \InNorms{z_0 - z_{\half}}^2 =& 0.00452784581555656, \\
    \InNorms{z_1 - z_{1+\half}}^2 =& 0.004552329544896258, \\
    \InNorms{z_2 - z_{2+\half}}^2 =& 0.004552306444552208.
\end{align*}
It is clear that the $\InNorms{z_k -z_{k+\half}}^2$ is not monotone.

\paragraph{Performance Measure: $\InNorms{z_k - z_{k+1}}^2$.} Let $A = \begin{bmatrix}
0.50676631 & 0.15042569\\
0.46897595 & 0.96748026
\end{bmatrix} $, $b=c = \begin{bmatrix}
1 \\ 1
\end{bmatrix}$.
Running the EG method on the corresponding VI problem with step-size $\eta = 0.1$ starting at $z_0 = (2.37003485, 0, 1.84327237, 0.25934775)^T$ has the following trajectory:
\begin{align*}
    z_1 =& (2.37267186, 0, 1.86351397, 0.1950396)^T,\\
    z_2 =& (2.37524308, 0, 1.88388624, 0.13077023)^T, \\
    z_3 =& (2.37774149, 0.00426125, 1.90438549, 0.06653856)^T.
\end{align*}
\notshow{
Thus we have
\begin{align*}
    &\InNorms{z_\half - z_{1}}^2 = 2.278045513420713 \times 10^{-8}, \\
    &\InNorms{z_{1+\half} - z_{2}}^2 = 2.2716975319731723 \times 10^{-8}, \\
    &\InNorms{z_{2+\half} - z_{3}}^2 = 1.8180916070316384 \times 10^{-5}.
\end{align*}
It is clear that the $\InNorms{z_{k+\half} -z_{k+1}}^2$ is not monotone.
}
Thus we have
\begin{align*}
    \InNorms{z_0 - z_{1}}^2 =& 0.004552214685275266, \\
    \InNorms{z_1 - z_{2}}^2 =& 0.004552191904998012, \\
    \InNorms{z_2 - z_{3}}^2 =& 0.004570327450598002.
\end{align*}
It is clear that the $\InNorms{z_{k} -z_{k+1}}^2$ is not monotone.

\subsection{Non-Monotonicity of the Gap Functions and its Variant}
\notshow{
We show that $\max_{z \in \Z}\InAngles{F(z), z_k - z}$ and $\max_{z \in \Z}\InAngles{F(z_k), z_k - z}$ of the EG updates are not monotone. We consider the following bilinear game over the 4-dimension unit cube $\Z = [0,1]^4$:
\begin{align*}
    \min_{x \in [0,1]^2} \max_{y \in [0,1]^2} x^\top Ay.
\end{align*}
Denote $z = (x,y)$. Note that the above problem is equivalent to the the monotone VI with $F(z):= (Ay, -A^\top x)^\top$. Since for any $z, z'\in \Z$, $\InAngles{F(z)-F(z'), z -z'} = 0$, we know 
\begin{align*}
    \InAngles{F(z), z - z'} = \InAngles{F(z'), z - z'}.
\end{align*}
Thus we have 
\begin{align*}
    \max_{z \in \Z}\InAngles{F(z), z_k - z} = \max_{z \in \Z}\InAngles{F(z_k), z_k - z} = \max_{y \in [0,1]^2} x_k^\top A y - \min_{x \in [0,1]^2} x^\top A y_k.
\end{align*}
We denote $\displaystyle \gap(z_k):= \max_{y \in [0,1]^2} x_k^\top A y - \min_{x \in [0,1]^2} x^\top A y_k$ as the duality gap at $z_k$. It suffices to show $\gap(z_k)$ is not monotone. 
}
\paragraph{Performance Measure: Gap Function and $\max_{z \in \Z}\InAngles{F(z), z_k - z}$.} 

Let $A = \begin{bmatrix}
-0.21025101 & 0.22360196\\
0.40667685 &-0.2922158
\end{bmatrix}$, $b=c = \begin{bmatrix}
0 \\ 0
\end{bmatrix}$.
One can easily verify that $\InAngles{F(z), z_k - z}=\InAngles{F(z_k), z_k - z}$,
which further implies that  $\max_{z \in \Z}\InAngles{F(z), z_k - z}=  \max_{z \in \Z}\InAngles{F(z_k), z_k - z}=\gap(z_k)$,
which implies that non-monotonicity of the gap function implies non-monotonicity of $\max_{z \in \Z}\InAngles{F(z), z_k - z}$.
Running the EG method on the corresponding VI problem with step-size $\eta = 0.1$ starting at $z_0 = (0.53095379, 0.29084076, 0.62132986, 0.49440498)$ has the following trajectory:
\begin{align*}
    &z_1 = (0.53290086, 0.28009156, 0.62151204, 0.4981395)^T, \\
    &z_2 = (0.5347502,  0.26947398, 0.62122195, 0.50222691)^T.
\end{align*}
One can easily verify that
\begin{align*}
    \gap(z_0) =& 0.6046398415472187, \\
    \gap(z_1) =& 0.58462873354003214, \\
    \gap(z_2) =& 0.5914026255469654.
\end{align*}
It is clear that the duality gap is not monotone.

\notshow{
\argyrisnote{Argyris: I would suggest not to include the following paragraph and figure.}
We also run EG on the same instance with starting point $x_0 = (0.5, 0.5)$ and $y_0 = (0.5, 0.5)$, and plot the duality gap in each round as shown in \Cref{fig:duality gap nonnomotone}.

\begin{figure}[h]
\centering
\includegraphics[width=0.5\textwidth]{COLT/non_monotone.png}
\caption{Numerical results of the non-monotonicity of the duality gap in EG}
\label{fig:duality gap nonnomotone}
\end{figure}
}

\section{Optimistic Gradient Descent Ascent Algorithm}\label{appx:OGDA}

Let $\Z \subseteq \R^n$ be a closed convex set and $F : \Z \rightarrow \R$ be an operator.  Let $z_k$ and $w_k$ be the $k$-th iterate of the Optimistic Gradient Descent Ascent algorithm (OGDA) algorithm. Let $z_0, w_0$ be arbitrary point in $\Z$ and $\{z_k, w_k\}_{k\ge 0}$ be the iterated of the OGDA algorithm. The update rule for any $k \ge 0$ is as follows:
\begin{nalign}
     w_{k+1} &= \Pi_\Z\left[z_k-\eta F(w_k)\right]=\arg \min_{z\in \Z} \| z-\left(z_k-\eta F(w_k)\right)\|\\
     z_{k+1} & = \Pi_\Z\left[z_k-\eta F(w_{k+1})\right]= \arg \min_{z\in \Z} \left \| z-\left(z_k-\eta F(w_{k+1})\right)\right \|
\end{nalign}
We prove last-iterate convergence for OGDA with respect to the gap function, natural residual and tangent residual in \Cref{thm: OGDA last-iterate at constrained} at Section~\ref{sec:combine OGDA}.
The last-iterate convergence proof for OGDA is a simple extension of the proof for EG.
The last-iterate convergence for the performance measures we mentioned follow from the last-iterate convergence of the following monotonically decreasing potential function:
\begin{align}\label{eq:potential OGDA}
\Phi_k = \InNorms{F(z_k)-F(w_k)}^2 + r^{tan}(z_k)^2
\end{align}

In Section~\ref{sec:best iterate OGDA} we show that OGDA enjoys last-iterate convergence with respect to the quantity $\InNorms{z_k - w_{k+1}}$ \citep{wei_last-iterate_2021,hsieh_convergence_2019} and in Section~\ref{sec:best to potential} we show how to upper bound the potential function $\Phi_k$ by the best-iterate.
In Section~\ref{sec:OGDA monotone} we show that the potential function $\Phi_k$ is monotonically decreasing across iterates and finally in Section~\ref{sec:combine OGDA} we show how to translate the last-iterate convergence with respect to the potential function $\Phi_k$ to last-iterate convergence of the performance measures of interest.

\subsection{Best-Iterate Convergence of OGDA with Constant Step Size}\label{sec:best iterate OGDA}

Best-iterate convergence guarantees for OGDA are known \citep{wei_last-iterate_2021} and can easily be derived by \cite{hsieh_convergence_2019}.
We include the proof here for completeness. 
\begin{lemma}\label{lem:OGDA Best-iterate}
Let $\Z \subseteq \R^n$ be a closed convex set, $F: \Z \rightarrow \R$ be a monotone and $L$-Lipschitz operator, and $z^*$ be a saddle point. Let $z_0, w_0 \in \Z$ be arbitrary starting points and $\{z_k,w_k\}_{k\ge 0}$ be the iterates of the OGDA algorithm with any step size $\eta \in  (0,\frac{1}{2L})$. Then for all $T \ge 1$, 
\begin{equation}\label{eq:best-iterate-1}
     \sum_{k=0}^T \InNorms{z_k - w_{k+1}}^2 \le \frac{1-2\eta^2 L^2}{1- 4\eta^2 L^2} \InNorms{z_0-z^*}^2 + \frac{2\eta^2 L^2}{1 - 4\eta^2 L^2} \InNorms{w_0 - z_0}^2.
\end{equation}
\end{lemma}
\begin{prevproof}{Lemma}{lem:OGDA Best-iterate}
In order to upper bound $\sum_{k=0}^T \InNorms{w_k - w_{k+1}}^2$, we first relate the quantity $\InNorms{w_k-w_{k+1}}^2$ to the weighted sum of $\{\InNorms{z_t - w_{t+1}}^2 \}_{0 \le t \le k}$.
\begin{lemma}\label{Lem: best-induction}
For all $k \ge 0$, 
\begin{align}\label{eq: best-induction-1}
    \InNorms{w_k-w_{k+1}}^2 \le 2(2\eta^2 L^2)^k\InNorms{w_0 - z_0}^2 + \sum_{t=0}^{k} 2 (2\eta^2 L^2)^{t} \InNorms{z_{k-t}- w_{k+1 -t}}^2.
\end{align}
Moreover, for all $T \ge 0$,
\begin{align}\label{eq: best-induction-2}
    \sum_{k = 0}^T \InNorms{w_k - w_{k+1}}^2 \le \frac{2}{1 - 2\eta^2 L^2} \InParentheses{\InNorms{w_0 - z_0}^2  + \sum_{k=0}^T \InNorms{z_k - w_{k+1}}^2}.
\end{align}
\end{lemma}
\begin{proof}
We first prove \Cref{eq: best-induction-1} by induction. Note that for all $k \ge 0$, we have 
\begin{align}
    \InNorms{w_k - w_{k+1}}^2 &= \InNorms{w_k - z_k + z_k - w_{k+1}}^2 \notag\\
    & \le 2\InNorms{w_k - z_k}^2 + 2\InNorms{z_k - w_{k+1}}^2. \label{eq:OGDA-5}
\end{align}
The inequality follows from the fact that $(a+b)^2\leq 2a^2 + 2b^2$.
Thus \Cref{eq: best-induction-1} holds for the base case $k = 0$.  
For the sake of induction, we assume that \Cref{eq: best-induction-1} holds for some $k -1 \ge 0$. Using the update rule of OGDA, the non-expansiveness of the projection operator, and the $L$-Lipschitzness of $F$, for all $k \ge 1$ we have 
\begin{align}\label{eq:OGDA-9}
    \InNorms{w_k - z_k}^2 \le  \eta^2 \InNorms{F(w_{k-1}) - F(w_k)}^2 \le \eta^2 L^2 \InNorms{w_{k-1} - w_k}^2.
\end{align}
Combining \Cref{eq:OGDA-5}, \Cref{eq:OGDA-9}, and the induction assumption, we have 
\begin{align*}
    \InNorms{w_k - w_{k+1}}^2 &\le 2\InNorms{w_k - z_k}^2 + 2\InNorms{z_k - w_{k+1}}^2\notag \\ 
    & \le 2\eta^2 L^2 \InNorms{w_{k-1} - w_k}^2 + 2\InNorms{z_k - w_{k+1}}^2 \\
    &\le 2\eta^2 L^2 \InParentheses{ 2 (2\eta^2 L^2)^{k-1} \InNorms{w_0-z_0}^2 +  \sum_{t=0}^{k-1} 2 (2\eta^2 L^2)^{t} \InNorms{z_{k-1-t}- w_{k -t}}^2} + 2\InNorms{z_k - w_{k+1}}^2 \\
    & = 2 (2\eta^2 L^2)^{k} \InNorms{w_0-z_0}^2 + \sum_{t=1}^{k} 2 (2\eta^2 L^2)^{t} \InNorms{z_{k-t}- w_{k+1 -t}}^2 + 2\InNorms{z_k - w_{k+1}}^2 \\
    & = 2 (2\eta^2 L^2)^{k} \InNorms{w_0-z_0}^2 + \sum_{t=0}^{k} 2 (2\eta^2 L^2)^{t} \InNorms{z_{k-t}- w_{k+1 -t}}^2. 
\end{align*}
This completes the proof of \Cref{eq: best-induction-1}.

Summing \Cref{eq: best-induction-1} with $k = 0, 1, \cdots, T$, we have 
\begin{align*}
    \sum_{k = 0}^T \InNorms{w_k - w_{k+1}}^2 &\le \sum_{k = 0}^T2 (2\eta^2 L^2)^{k} \InNorms{w_0-z_0}^2 +   \sum_{k = 0}^T \sum_{t=0}^{k} 2 (2\eta^2 L^2)^{t} \InNorms{z_{k-t}- w_{k+1 -t}}^2 \\
    & = \sum_{k = 0}^T2 (2\eta^2 L^2)^{k} \InNorms{w_0-z_0}^2 +  \sum_{k = 0}^T \InParentheses{ \sum_{t = 0}^{T-k} 2 (2\eta^2 L^2)^{t}} \cdot \InNorms{z_{k}- w_{k+1}}^2 \\
    & \le \frac{2}{1- 2 \eta^2 L^2} \InParentheses{ \InNorms{w_0-z_0}^2 + \sum_{k = 0}^T \InNorms{z_{k}- w_{k+1}}^2}.
\end{align*}
This completes the proof of \Cref{eq: best-induction-2}.
\end{proof}

Back to the proof of Lemma~\ref{lem:OGDA Best-iterate}.
For all $k \ge 0$, we have
\begin{align}
    \InNorms{z_{k+1} - z^*}^2 &= \InNorms{z_{k+1} - z_k + z_k - z^*}^2 \notag \\
    &= \InNorms{z_k-z^*}^2 + \InNorms{z_{k+1} - z_k}^2  + 2\InAngles{z_{k+1} - z_k, z_k - z^*} \notag\\
    &= \InNorms{z_k-z^*}^2 - \InNorms{z_{k+1} - z_k}^2  + 2\InAngles{z_{k+1} - z_k, z_{k+1} - z^*}\notag\\
    & \le \InNorms{z_k-z^*}^2 - \InNorms{z_{k+1} - z_k}^2  - 2\eta\InAngles{F(w_{k+1}), z_{k+1} - z^*}.\label{eq:OGDA-1}
\end{align}
The last inequality follows from $\InAngles{z_{k+1} - z_k + \eta F(w_{k+1}), z_{k+1}-z^*} \le 0$ as $z_{k+1}  = \Pi_\Z\left[z_k-\eta F(w_{k+1})\right]$.

Similarly, for all $k \ge 0$, we have 
\begin{align}
    \InNorms{z_{k+1} - w_{k+1}}^2 &= \InNorms{z_{k+1} - z_k + z_k - w_{k+1}}^2 \notag\\
    &= \InNorms{z_{k+1} - z_k}^2 + \InNorms{z_k -w_{k+1}}^2 + 2\InAngles{z_k -w_{k+1}, z_{k+1} - z_k}\notag\\
    &= \InNorms{z_{k+1} - z_k}^2 - \InNorms{z_k -w_{k+1}}^2 + 2\InAngles{z_k -w_{k+1}, z_{k+1} - w_{k+1}}\notag\\
    & \le \InNorms{z_{k+1} - z_k}^2 - \InNorms{z_k -w_{k+1}}^2  + 2\eta\InAngles{F(w_k), z_{k+1} - w_{k+1}}.\label{eq:OGDA-2}
\end{align}
The last inequality follows from $\InAngles{z_k - \eta F(w_k) - w_{k+1}, z_{k+1}-w_{k+1}} \le 0$ as $w_{k+1} = \Pi_\Z\left[z_k-\eta F(w_k)\right]$.

\notshow{
\begin{fact}\label{fact:monotonicity implication}
  For all $z\in \Z$, $\langle F(z), z^*-z\rangle\leq 0$.
\end{fact}
\begin{proof}
\begin{align*}
     0 &\leq \langle F(z^*)- F(z), z^*-z\rangle \qquad &(\text{monotonicity of $F(\cdot)$})\\
     &=  \langle F(z^*), z^*-z\rangle -  \langle F(z), z^*-z\rangle &\\
     & \leq - \langle F(z), z^*-z\rangle. \qquad &(\text{optimality of $z^*$ and $z\in \Z$})
 \end{align*}
\end{proof}
}
We can further simplify \Cref{eq:OGDA-1} using \Cref{fact:monotonicity implication}:
\begin{align}
    \InNorms{z_{k+1} - z^*}^2 &\le \InNorms{z_k-z^*}^2 - \InNorms{z_{k+1} - z_k}^2  - 2\eta\InAngles{F(w_{k+1}), z_{k+1} - z^*}\notag\\
    & = \InNorms{z_k-z^*}^2 - \InNorms{z_{k+1} - z_k}^2 - 2\eta\InAngles{F(w_{k+1}), z_{k+1} - w_{k+1}} + 2\eta\InAngles{F(w_{k+1}), z^* - w_{k+1}} \notag \\
    &\le \InNorms{z_k-z^*}^2 - \InNorms{z_{k+1} - z_k}^2 - 2\eta\InAngles{F(w_{k+1}), z_{k+1} - w_{k+1}}. \label{eq:OGDA-3}
\end{align}

Summing \Cref{eq:OGDA-2} and \Cref{eq:OGDA-3}, we get
\begin{align}
    \InNorms{z_{k+1} - z^*}^2 &\le \InNorms{z_k - z^*}^2 - \InNorms{z_k - w_{k+1}}^2 - \InNorms{z_{k+1} - w_{k+1}}^2 + 2\eta \InAngles{F(w_k) - F(w_{k+1}), z_{k+1} - w_{k+1}}\notag\\
    & \le \InNorms{z_k - z^*}^2 - \InNorms{z_k - w_{k+1}}^2 - \InNorms{z_{k+1} - w_{k+1}}^2 + 2\eta \InNorms{F(w_k) - F(w_{k+1})}\InNorms{z_{k+1} - w_{k+1}} \notag\\
    & \le \InNorms{z_k - z^*}^2 - \InNorms{z_k - w_{k+1}}^2 - \InNorms{z_{k+1} - w_{k+1}}^2 + 2\eta L \InNorms{w_k - w_{k+1}}\InNorms{z_{k+1} - w_{k+1}} \notag \\
    & \le \InNorms{z_k - z^*}^2 - \InNorms{z_k - w_{k+1}}^2 + \eta^2 L^2 \InNorms{w_k - w_{k+1}}^2, \label{eq:OGDA-4}
\end{align}
where we use Cauchy-Schwarz inequality in the second inequality and $L$-Lipschitzness of $F(\cdot)$ in the third inequality. In the last inequality, we optimize the quadratic function in $\InNorms{z_{k+1}-w_{k+1}}$.

Summing \Cref{eq:OGDA-4} for $k = 0, 1, \cdots, T$ and using \Cref{Lem: best-induction}, we get 
\begin{align}
    \InNorms{z_{T+1}-z^*}^2 &\le \InNorms{z_0 - z^*}^2 - \sum_{k=0}^T \InNorms{z_k - w_{k+1}}^2 + \eta^2 L^2 \sum_{k=0}^T \InNorms{w_k-w_{k+1}}^2  \notag\\ 
    &\le \InNorms{z_0 - z^*}^2 - \sum_{k=0}^T \InNorms{z_k - w_{k+1}}^2 + \frac{2 \eta^2 L^2}{1 - 2\eta^2 L^2} \InParentheses{\InNorms{w_0 - z_0}^2 +  \sum_{k=0}^T  \InNorms{z_k - w_{k+1}}^2 } \tag{\Cref{Lem: best-induction}}\\
    &= \InNorms{z_0 - z^*}^2 - \frac{1 - 4\eta^2 L^2 }{1 - 2\eta^2 L^2} \sum_{k=0}^T  \InNorms{z_k - w_{k+1}}^2 + \frac{2 \eta^2 L^2}{1 - 2\eta^2 L^2} \InNorms{w_0 - z_0}^2  . \notag 
\end{align}
Since $\eta^2 L^2 < \frac{1}{4}$, we complete the proof by rearranging the above inequality.
\end{prevproof}

\subsection{Best-Iterate of $\Phi_k$}\label{sec:best to potential}
In this section,
we use Lemma~\ref{lem:OGDA Best-iterate} to show that there exists $t^*\in[T]$ such that $\Phi_{t^*}=O\InParentheses{\frac{1}{T}}$.


\begin{lemma}\label{lem:OGDA potential Best-iterate}
Let $\Z \subseteq \R^n$ be a closed convex set, $F: \Z \rightarrow \R$ be a monotone and $L$-Lipschitz operator, and $z^*$ be a saddle point. Let $z_0, w_0 \in \Z$ be arbitrary starting point and $\{z_k,w_k\}_{k\ge 0}$ be the iterates of the OGDA algorithm with any step size $\eta \in  (0,\frac{1}{2L})$. Then for all $T \ge 1$, 

$$ \sum_{k=1}^{T} \InParentheses{ \InNorms{\eta F(z_k) - \eta F(w_k)}^2 + \eta^2\Ham(z_k)^2} \le \frac{ 4 + 6\eta^4 L^4 }{1 - 4\eta^2 L^2} \InNorms{z_0-z^*}^2 +    
\frac{16\eta^2L^2 + 6\eta^4 L^4}{1-4\eta^2 L^2} \InNorms{w_0-z_0}^2
.$$
Moreover, when $w_0 = z_0$
$$ \sum_{k=1}^{T} \InParentheses{ \InNorms{\eta F(z_k) - \eta F(w_k)}^2 + \eta^2\Ham(z_k)^2} \le \frac{ 4 + 6\eta^4 L^4 }{1 - 4\eta^2 L^2} \InNorms{z_0-z^*}^2.$$
\end{lemma}

\notshow{
We first provide the following upper bound of the squared tangent residual.
\begin{lemma}\label{lemma:upperbound hamiltonian}
For all $k \ge 1$, $\Ham(z_k)^2 \le \InNorms{\frac{z_{k-1}-z_k}{\eta} + F(z_k) - F(w_k)}^2$.
\end{lemma}
\begin{proof}
Fix any $k \ge 1$. By the update rule of OGDA, we know $z_k = \Pi_\Z\left[z_{k-1} - F(w_k) \right]$. If $z_k = z_{k-1} - \eta F(w_k)$, then \Cref{lemma:upperbound hamiltonian} holds since 
\begin{align*}
    \Ham(z_k)^2 \le \InNorms{F(z_k)}^2 = \InNorms{F(w_k) + F(z_k) - F(w_k)}^2 = \InNorms{\frac{z_{k-1}-z_k}{\eta} + F(z_k) - F(w_k)}^2.
\end{align*}
For the rest of the proof, we assume that $z_k \ne z_{k-1}-\eta F(w_k)$. 
Since $z_{k-1} - \eta F(w_k) - z_k \in \normal(z_k)$, we have 
\begin{align*}
    \InAngles{z_{k-1} -\eta F(w_k) - z_k, z - z_k} \le 0, \quad \forall z \in \Z.
\end{align*}
Define $a := \frac{z_{k-1} -\eta F(w_k) - z_k}{\InNorms{z_{k-1} -\eta F(w_k) - z_k}}$. We thus know $a \in \unitnormal(z_k)$. Let 
$$a_\perp :=
\begin{cases}
\frac{F(z_{k})-  \langle a, F(z_{k})\rangle \cdot a}{\InNorms{F(z_{k})-  \langle a, F(z_{k})\rangle \cdot a}} \qquad & \text{if $\InNorms{F(z_{k})-  \langle a, F(z_{k})\rangle \cdot a}\neq 0$,}\\
\InParentheses{0,\ldots,0} & \text{otherwise.}
\end{cases}
$$
Observe that $\langle a, a\rangle = 1$, $\langle a_\perp, a\rangle = 0$ and $F(z_{k}) =\langle a, F(z_{k})\rangle a + \langle a_\perp, F(z_{k})\rangle a_\perp $. Thus
\begin{align}
&0=\langle a_\perp, a\rangle =\langle a_\perp, z_{k-1} -  \eta F(w_k) -z_k\rangle\nonumber \\
\Leftrightarrow &\langle a_\perp, F(w_k)\rangle = \frac{\langle a_\perp, z_{k-1} -z_k\rangle}{\eta}.
\label{eq:penpedicular part OGDA}
\end{align}
Moreover, the fact that
$\InAngles{a, z_{k} - z_{k-1} + \eta F(w_k)} \le 0$ implies $ \InAngles{a, F(w_k)} \le \frac{\InAngles{a, z_{k-1} - z_k}}{\eta}
$, which further implies that
 \begin{align}
&\InAngles{a, F(z_k)} = \InAngles{a, F(w_k)+ F(z_k) - F(w_k)} \leq  \InAngles{a, \frac{z_{k-1}-z_{k}}{\eta} + F(z_k) - F(w_k)}. 
\label{eq:normal part OGDA}
\end{align}
Combining the definition of $\Ham(z_k)$, Equation~\eqref{eq:penpedicular part OGDA}, and Equation~\eqref{eq:normal part OGDA} we have 
\begin{align*}
    \Ham(z_k)^2 &\le ||F(z_{k})||^2 - \InAngles{a, F(z_{k})}^2\cdot \mathbbm{1}[\InAngles{a, F(z_{k})} \leq 0] \\
    & = \InAngles{a_\perp, F(z_{k})}^2 + \InAngles{a, F(z_k)}^2\cdot \mathbbm{1}[\InAngles{a, F(z_k)} > 0] \\
    & \le \InAngles{a_\perp, \frac{z_{k-1} - z_k}{\eta} + F(z_k) - F(w_k) }^2 + \InAngles{a, \frac{z_{k-1} - z_k}{\eta} + F(z_k) - F(w_k) }^2 \\
    & \leq \InNorms{\frac{z_{k-1} - z_k}{\eta} + F(z_k) - F(w_k)}^2. \qedhere
\end{align*}
\end{proof}
}
\begin{prevproof}{Lemma}{lem:OGDA potential Best-iterate}
For all $k\ge 1$, we have
\begin{align}
    \InNorms{\eta F(z_k) - \eta F(w_k)}^2 &\le \eta^2 L^2 \InNorms{z_k - w_k}^2 \tag{$L$-Lipschitzness of $F$} \\
    & \le \eta^4 L^4 \InNorms{w_{k-1} - w_k}^2 \tag{\Cref{eq:OGDA-9}}.
\end{align}
Using \Cref{lem:upper bound tangent residual} with the fact that $z_k = \Pi_\Z [z_{k-1} - \eta F(w_k)]$, we have for all $k \ge 1$,
\begin{align}
    \eta^2\Ham(z_k)^2 
    &\le \InNorms{z_{k-1} - z_k + \eta F(z_k) - \eta F(w_k)}^2 \notag\\
    &\le 2\InNorms{z_{k-1} - z_k}^2 +  2\eta^2\InNorms{F(z_k) - F(w_k)}^2 \notag \\
    &\le 2\InNorms{z_{k-1} - w_k + w_k -  z_k}^2 +  2\eta^2 L^2 \InNorms{w_k - z_k}^2 \tag{$L$-Lipschitzness of $F$}\\
    &\le 4\InNorms{z_{k-1} - w_k}^2 + (4 + 2\eta^2 L^2)\InNorms{w_k -  z_k}^2 \notag \\
    &\le 4\InNorms{z_{k-1} - w_k}^2 + (4 + 2\eta^2 L^2)\eta^2 L^2 \InNorms{w_{k-1}-w_k}^2 \tag{\Cref{eq:OGDA-9}}.
\end{align}

Summing the above inequalities with $k = 1, \cdots, T$ and using \Cref{lem:OGDA Best-iterate} and \Cref{Lem: best-induction}, we have
\begin{align*}
    &\sum_{k=1}^{T}\InParentheses{\InNorms{\eta F(z_k) - \eta F(w_k)}^2 + \eta^2\Ham(z_k)^2  }\\
    &\le
    4\sum_{k=0}^{T-1}{\InNorms{z_{k} - w_{k+1}}^2} + (4 + 3\eta^2 L^2 ) \eta^2 L^2 \sum_{k=0}^{T-1}\InNorms{w_{k}-w_{k+1}}^2 \\
    &\leq  \frac{2(4 + 3\eta^2 L^2 ) \eta^2 L^2}{1 - 2\eta^2 L^2} \InNorms{w_0 - z_0}^2
    +
    \InParentheses{4+\frac{2(4 + 3\eta^2 L^2 ) \eta^2 L^2}{1 - 2\eta^2 L^2}} \sum_{k=0}^{T-1} \InNorms{z_k - w_{k+1}}^2 \\
    &\leq \frac{2(4 + 3\eta^2 L^2 ) \eta^2 L^2}{1 - 2\eta^2 L^2} \InNorms{w_0 - z_0}^2
    +\InParentheses{\frac{8\eta^2 L^2}{1 - 4\eta^2 L^2} + \frac{4(4 + 3\eta^2 L^2 ) \eta^4 L^4}{\InParentheses{1 - 2\eta^2 L^2}\cdot \InParentheses{1 - 4\eta^2 L^2}}} \InNorms{w_0-z_0}^2 \\
    &\quad  +\InParentheses{\frac{4-8\eta^2 L^2}{1- 4\eta^2 L^2} + \frac{2(4 + 3\eta^2 L^2 ) \eta^2 L^2}{1 - 4\eta^2 L^2}} \InNorms{z_0-z^*}^2
    \\
    &= \frac{16\eta^2L^2 + 6\eta^4 L^4}{1-4\eta^2 L^2} \InNorms{w_0-z_0}^2 +\frac{ 4 + 6\eta^4 L^4 }{1 - 4\eta^2 L^2} \InNorms{z_0-z^*}^2,
\end{align*}
which concludes the proof.
\end{prevproof}

\begin{corollary}\label{cor:OGDA potential Best-iterate}
Let $\Z \subseteq \R^n$ be a closed convex set, $F: \Z \rightarrow \R$ be a monotone and $L$-Lipschitz operator, and $z^*$ be a saddle point. Let $z_0, w_0 \in \Z$ be arbitrary starting point and $\{z_k,w_k\}_{k\ge 0}$ be the iterates of the OGDA algorithm with any step size $\eta \in  (0,\frac{1}{2L})$. Then for all $T \ge 1$, there exists $t^*\in [T]$ such that 
$$   \InNorms{\eta F(z_{t^*}) - \eta F(w_{t^*})}^2 + \eta^2\Ham(z_{t^*})^2 \le \frac{1}{T}\frac{ 4 + 6\eta^4 L^4 }{1 - 4\eta^2 L^2} \InNorms{z_0-z^*}^2 +    \frac{1}{T}\frac{16\eta^2 L^2 + 6\eta^4 L^4}{1-4\eta^2 L^2}\InNorms{w_0 - z_0}^2.$$
Moreover, when $w_0 = z_0$
$$   \InNorms{\eta F(z_{t^*}) - \eta F(w_{t^*})}^2 + \eta^2\Ham(z_{t^*})^2 \le \frac{1}{T}\frac{ 4 + 6\eta^4 L^4 }{1 - 4\eta^2 L^2} \InNorms{z_0-z^*}^2.$$
\end{corollary}

\subsection{Monotonicity of the Potential}\label{sec:OGDA monotone}

In this section we show that the potential function $\Phi_k$ is monotonically decreasing across iterates of OGDA. We only include the simplified proof discovered using a degree 2 SOS program. The original proof is based on a higher degree SOS program and can be found in an earlier version of this paper \citep{v2-last-iterate} and at this \href{https://arxiv.org/abs/2204.09228v2}{link}.

\begin{theorem}\label{thm:monotone potential at constrained}
Let $\Z \subseteq \R^n$ be a closed convex set and $F: \Z \rightarrow \R$ be a monotone and $L$-Lipschitz operator. Then for any $z_k, w_k \in \Z$, the OGDA algorithm with any step size $\eta \in  (0,\frac{1}{2L})$ produces $w_{k+1}, z_{k+1} \in \Z$ that satisfy $\InNorms{F(z_k) - F(w_k)}^2+ \Ham(z_k)^2 \ge \InNorms{F(z_{k+1}) - F(w_{k+1})}^2+ \Ham(z_{k+1})^2$. 
\end{theorem}

\begin{proof}
Let $c_k = \Pi_{\normal_\Z(z_k)}(-F(z_k))$ and  $c_{k+1} = \Pi_{\normal_\Z(z_{k+1})}(-F(z_{k+1}))$.
\Cref{lemma:property of tangent residual} implies that
\begin{align}
    &\eta^2 r^{tan}(z_k)^2 + \eta^2 \InNorms{F(z_k) - F(w_k)}^2 - \InParentheses{\eta^2 r^{tan}(z_{k+1})^2+\InNorms{F(z_{k+1}) - F(w_{k+1})}^2} \notag\\
    &=  \|\eta F(z_k)+\eta c_k\|^2 + \eta^2 \InNorms{F(z_k) - F(w_k)}^2 \notag\\
    &\quad- \InParentheses{\|\eta F(z_{k+1})+\eta c_{k+1}\|^2+\InNorms{F(z_{k+1}) - F(w_{k+1})}^2 }\label{eq:target OGDA}
\end{align}

Since $F$ is monotone and $L$-Lipschitz, and $\eta \in (0,\frac{1}{2L})$, we have 
\begin{align}
    (-2)\cdot \InParentheses{\InAngles{\eta F(z_{k+1}) - \eta F(z_k), z_{k+1} - z_k}} &\le 0, \label{eq:cons-mon} \\
    (-2)\cdot \InParentheses{\frac{1}{4} \InNorms{z_{k+1} - w_{k+1}}^2- \InNorms{\eta F(z_{k+1}) - \eta F(w_{k+1})}^2} & \le 0.\label{eq:cons-lip}
\end{align}
Since $w_{k+1} = \Pi_\Z \inteval{z_k - \eta F(w_k)}$ and $z_{k+1} = \Pi_\Z \inteval{z_k - \eta F(w_{k+1})}$,
we have that $z_k - \eta F(w_k) - w_{k+1} \in N(w_{k+1})$ and $z_k - \eta F(w_{k+1}) - z_{k+1} \in N(z_{k+1})$. Thus we have
\begin{align}
    (-1)\cdot\InAngles{z_k - \eta F(w_k) - w_{k+1}, w_{k+1} - z_{k+1}} & \le 0, \label{eq:cons-proj1} \\
    (-2)\cdot \InAngles{z_k - \eta F(w_{k+1}) - z_{k+1}, z_{k+1}-z_k} & \le 0 \label{eq:cons-proj2}.
\end{align}
Since $c(z_k) \in N(z_k)$, we have that
\begin{align}
    (-1)\cdot \InAngles{\eta c(z_k), z_k - w_{k+1}} & \le 0, \label{eq:cons-ak-1} \\
    (-1)\cdot \InAngles{\eta c(z_k), z_k - z_{k+1}} & \le 0. \label{eq:cons-ak-2}
\end{align}
According to \Cref{lemma:property of tangent residual} and the fact that $z_k - \eta F(w_{k+1}) - z_{k+1} \in N(z_{k+1})$, $c_{k+1}\in \Pi_{\normal(z_{k+1})}\InParentheses{-F(z_{k+1})}$ we have 

\begin{align}
    \InParentheses{-2}\cdot \InAngles{\eta c(z_{k+1}) + \eta F(z_{k+1}), z_k - \eta F(w_{k+1}) - z_{k+1} } \le 0, \label{eq:cons-ak+1}\\
    (-2)\cdot \InAngles{\eta c(z_{k+1}) + \eta F(z_{k+1}), -c(z_{k+1})} = 0,\label{eq:cons-ak-3}.
\end{align}

\textsc{Matlab} code for the verification of the following identity can be found at 
this \href{https://raw.githubusercontent.com/aroikonomou/code-verification-OGDA/main/ogda_constrained_identity_verification.m}{link}.

\begin{align}
    &\text{Expression}~\eqref{eq:target OGDA} + \LHSI~\eqref{eq:cons-mon} +  \LHSI~\eqref{eq:cons-lip} + \LHSI~\eqref{eq:cons-proj1} \notag\\
    & \quad + \LHSI~\eqref{eq:cons-proj2} + \LHSI~\eqref{eq:cons-ak-1} + \LHSI~\eqref{eq:cons-ak-2} \notag\\
    & \quad + \LHSI~\eqref{eq:cons-ak-3} + \LHSI~\eqref{eq:cons-ak+1}\notag \\
    & = \InNorms{\frac{w_{k+1}-z_{k+1}}{2} + \eta F(w_{k})  - \eta F(z_k)}^2  \\
    &+ \InNorms{\eta F(z_k) + \eta c(z_k) - z_k + \frac{w_{k+1}+z_{k+1}}{2}}^2 \\
    &+ \InNorms{z_k - \eta F(w_{k+1})- z_{k+1} - \eta c(z_{k+1})}^2 \\
    &\ge 0\notag.
\end{align}
Thus, $\InNorms{F(z_k) - F(w_k)}^2+ \Ham(z_k)^2 \ge \InNorms{F(z_{k+1}) - F(w_{k+1})}^2+ \Ham(z_{k+1})^2$.
\end{proof}

\subsection{Combining Everything}\label{sec:combine OGDA}

In this section,
we combine the results of the previous sections and show that $\Phi_T=O\left(\frac{1}{T}\right)$ and we show the last-iterate convergence rate for performance measures of iterest.

\begin{lemma}\label{lem:OGDA relate w and z}
Let $\Z \subseteq \R^n$ be a closed convex set and $F: \Z \rightarrow \R$ be a monotone and $L$-Lipschitz operator. Let $z_0,w_0 \in \Z$ be arbitrary starting point and $\{z_k,w_k\}_{k\ge 0}$ be the iterates of the OGDA algorithm with any step size $\eta \in  (0,\frac{1}{2L})$. Then for any $k \ge 0$, 
\begin{align*}
    r^{tan}_{(F,\Z)}(w_{k+1}) \le \sqrt{2}(2+\eta L) \sqrt{r^{tan}_{(F,\Z)}(z_k)^2 + \InNorms{F(w_k) - F(z_k)}^2}.
\end{align*}
\end{lemma}
\begin{proof}
Since $w_{k+1} = \Pi_\Z [z_k - F(w_k)]$, by using \Cref{lem:upper bound tangent residual} we have
\begin{align*}
    r^{tan}_{(F,\Z)}(w_{k+1}) &\le \InNorms{\frac{z_k - w_{k+1}}{\eta} + F(w_{k+1}) - F(w_{k})} \\
    & \le \InNorms{\frac{z_k - w_{k+1}}{\eta}} + \InNorms{F(w_k) - F(z_k)} + \InNorms{F(z_k) - F(w_{k+1})} \\
    & \le \frac{1+\eta L}{\eta }\InNorms{z_k - w_{k+1}} + \InNorms{F(w_k) - F(z_k)}. \tag{$L$-Lipschitzness of $F$}
\end{align*}
Using \Cref{lem:projected hamiltonian dominates natural residue} and the non-expansiveness of the projection mapping, we have
\begin{align*}
    \InNorms{z_k - w_{k+1}} &\le \InNorms{z_k - \Pi_\Z[z_k - \eta F(z_k)]} + \InNorms{\Pi_\Z[z_k - \eta F(z_k)] - w_{k+1}} \\
    &= r^{nat}_{(\eta F,\Z)}(z_k) + \InNorms{\Pi_\Z[z_k - \eta F(z_k)] - \Pi_\Z[z_k - \eta F(w_k)]} \\
    & \le r^{tan}_{(\eta F,\Z)}(z_k) +\eta \InNorms{F(z_k) - F(w_k)} \\
    & = \eta r^{tan}_{(F,\Z)}(z_k) +\eta \InNorms{F(z_k) - F(w_k)}.
\end{align*}
Combing the above two inequalities, we have 
\begin{align*}
    r^{tan}_{(F,\Z)}(w_{k+1}) &\le (1+\eta L) r^{tan}_{(F,\Z)}(z_k) + (2+\eta L) \InNorms{F(w_k) - F(z_k)} \\
    & \le \sqrt{2}(2+\eta L) \sqrt{r^{tan}_{(F,\Z)}(z_k)^2 + \InNorms{F(w_k) - F(z_k)}^2}. \tag{$a+b \le \sqrt{2} \sqrt{a^2+b^2}$}
\end{align*}
\end{proof}

Combining \Cref{cor:OGDA potential Best-iterate}, \Cref{thm:monotone potential at constrained}, \Cref{lem:OGDA relate w and z}, \Cref{lem:projected hamiltonian dominates natural residue} and \Cref{lem:hamiltonian to gap} we get  $\mathcal{O}(\frac{1}{\sqrt{T}})$ last-iterate convergence in terms of the tangent residual, natural residual and gap function for both $z_T$ and $w_{T+1}$. 
The result is formally stated in \Cref{thm: OGDA last-iterate at constrained}.

\begin{theorem}\label{thm: OGDA last-iterate at constrained}
Let $\Z \subseteq \R^n$ be a closed convex set and $F: \Z \rightarrow \R$ be a monotone and $L$-Lipschitz operator. Let $z_0,w_0 \in \Z$ be arbitrary starting point and $\{z_k,w_k\}_{k\ge 0}$ be the iterates of the OGDA algorithm with any step size $\eta \in  (0,\frac{1}{2L})$. 
Let $D_0 := \sqrt{\InParentheses{4+ 6\eta^4 L^4}  \InNorms{z_0 - z^*}^2 +  \InParentheses{16\eta^2 L^2 + 6\eta^4 L^4} \InNorms{w_0 - z_0}^2} = O(\max\{\InNorms{z_0 - z^*}, \InNorms{w_0 - z_0}\})$.
Then for any $T \ge 1$, 
\begin{itemize}
    \item $\gap_{\Z,F,D}(z_{T}) \le \frac{1}{\sqrt{T}}\cdot \frac{DD_0 }{ \eta  \sqrt{ 1-4\cdot (\eta L)^2}}$.
    \item $r_{\Z,F,D}^{nat}(z_{T}) \le r_{\Z,F,D}^{tan}(z_T) \le \frac{1}{\sqrt{T}}\cdot \frac{D_0 }{ \eta \cdot \sqrt{ 1-4\cdot (\eta L)^2)}} $.
    \item $\gap_{\Z,F,D}(w_{T+1}) \le \frac{1}{\sqrt{T}}\cdot \frac{\sqrt{2}(2+\eta L)\cdot D\cdot D_0 }{ \eta \cdot \sqrt{ 1-4\cdot (\eta L)^2}}$.
    \item $r_{\Z,F,D}^{nat}(w_{T+1}) \le r_{\Z,F,D}^{tan}(w_{T+1}) \le \frac{1}{\sqrt{T}}\cdot \frac{\sqrt{2}(2+\eta L) D_0 }{ \eta \cdot \sqrt{ 1-4\cdot (\eta L)^2}} $.
\end{itemize}
\notshow{
\begin{align*}
    &\gap_{\Z,F,D}(z_T) \le   \frac{DL}{\sqrt{ T}} \sqrt{\frac{4+ 6\eta^4 L^4}{\eta^2 L^2 (1-4\eta^2 L^2)}  \InNorms{z_0 - z^*}^2 +  \weiqiangnote{ \frac{6\eta^2 L^2 + 16\eta^4 L^4}{\eta^2 L^2(1- 4\eta^2L^2)} \InNorms{w_0 - z_0}^2}},\\
    &r_{\Z,F,D}^{nat}(z_T) \le r_{\Z,F,D}^{tan}(z_T) \le \frac{L}{\sqrt{ T}} \sqrt{\frac{4+ 6\eta^4 L^4}{\eta^2 L^2 (1-4\eta^2 L^2)}  \InNorms{z_0 - z^*}^2 +  \weiqiangnote{ \frac{6\eta^2 L^2 + 16\eta^4 L^4}{\eta^2 L^2(1- 4\eta^2L^2)} \InNorms{w_0 - z_0}^2}},\\
    &\gap_{\Z,F,D}(w_{T+1}) \le 2(2+\eta L) \frac{DL}{\sqrt{T}} \sqrt{\frac{4+ 6\eta^4 L^4}{\eta^2 L^2 (1-4\eta^2 L^2)}  \InNorms{z_0 - z^*}^2 + \weiqiangnote{ \frac{6\eta^2 L^2 + 16\eta^4 L^4}{\eta^2 L^2(1- 4\eta^2L^2)} \InNorms{w_0 - z_0}^2}},\\
    &r_{\Z,F,D}^{nat}(w_{T+1}) \le r_{\Z,F,D}^{tan}(w_{T+1}) \le 2(2+\eta L) \frac{L}{\sqrt{T}} \sqrt{\frac{4+ 6\eta^4 L^4}{\eta^2 L^2 (1-4\eta^2 L^2)}  \InNorms{z_0 - z^*}^2 + \weiqiangnote{ \frac{6\eta^2 L^2 + 16\eta^4 L^4}{\eta^2 L^2(1- 4\eta^2L^2)} \InNorms{w_0 - z_0}^2}}
\end{align*}
}
\end{theorem}

\notshow{
\begin{theorem}\label{thm: OGDA last-iterate at constrained}
Let $\Z \subseteq \R^n$ be a closed convex set and $F: \Z \rightarrow \R$ be a monotone and $L$-Lipschitz operator. Let $z_0,w_0 \in \Z$ be arbitrary starting point and $\{z_k,w_k\}_{k\ge 0}$ be the iterates of the OGDA algorithm with any step size $\eta \in  (0,\frac{1}{2L})$. Then for any $T \ge 1$, 
\begin{align*}
    &\gap_{\Z,F,D}(z_T) \le   \frac{DL}{\sqrt{ T}} \sqrt{\frac{4+ 6\eta^4 L^4}{\eta^2 L^2 (1-4\eta^2 L^2)}  \InNorms{z_0 - z^*}^2 + \weiqiangnote{ \frac{6\eta^2 L^2 + 16\eta^4 L^4}{\eta^2 L^2(1- 4\eta^2L^2)} \InNorms{w_0 - z_0}^2}}\\
    &\gap_{\Z,F,D}(w_{T+1}) \le 2(2+\eta L) \frac{DL}{\sqrt{T}} \sqrt{\frac{4+ 6\eta^4 L^4}{\eta^2 L^2 (1-4\eta^2 L^2)}  \InNorms{z_0 - z^*}^2 + \weiqiangnote{ \frac{6\eta^2 L^2 + 16\eta^4 L^4}{\eta^2 L^2(1- 4\eta^2L^2)} \InNorms{w_0 - z_0}^2}}.
\end{align*}
\end{theorem}

Setting $D = \max\{ \InNorms{z_0 - z^*}, \InNorms{w_0 - z_0}\}$, and $\eta = \frac{1}{2\sqrt{2}L}$, we have $\gap_{\Z,F,D}(z_T) = O \InParentheses{ \frac{D^2L}{\sqrt{T}}}$, which matches the lower bound $\Omega(\frac{D^2L}{\sqrt{T}})$ by \cite{golowich_tight_2020} up to constant factor.

}

\section{Agnostic to \Cref{lemma:property of tangent residual} Proof for Monotonicity of Tangent Residual of EG }\label{sec:high degree last-iterate EG}

Let $a = \max_{\substack{a\in \unitnormal_\Z(z),\\ \InAngles{F(z),a}\le 0}}\langle a, F(z) $.
By \Cref{def:projected hamiltonian},
a natural formulation of the squared tangent residual is $r^{tan}(z)^2 = \InNorms{F(z)}^2 - \InAngles{a,F(z)}^2 $,
which is a degree-4 formulation of the tangent residual with respect to our $\{a,F(z)\}$.
In this section,
we show how to prove monotonicity of tangent residual of EG with arbitrary convex constraints,
while being agnostic to the degree two formulation of tangent residual as shown in \Cref{lemma:property of tangent residual} with the use of a higher-degree SOS program.

\paragraph{Reducing the Number of Constraints.} 
Our reasoning behind reducing the number of constraints follows similar to the corresponding paragraph in Section~\ref{sec:last-iterate EG} with minor modifications that we list here for completeness.
Suppose we are not given the description of $\Z$, and we only observe one iteration of the EG algorithm. In other words, we know $z_k$, $z_{k+\half}$, and $z_{k+1}$, as well as $F(z_k)$, $F(z_{k+\half})$, and $F(z_{k+1})$. To compute the squared \projham at $z_k$, let us also assume that the unit vector $-a_k\in \unitnormal(z_k)$ satisfies $r^{tan}_{(F,\Z)}(z_k)^2=\InNorms{F(z_k)}^2-\InAngles{F(z_k),-a_k}^2$. From this limited information, what can we learn about $\Z$? We can conclude that $\Z$ must lie in the intersection of the following halfspaces: (a) $\InAngles{a_k,z}\geq b_k$, where $b_k=\InAngles{a_k,z_k}$. This is true because $-a_k\in \unitnormal(z_k)$. (b) $\InAngles{a_{k+\half},z}\geq b_{k+\half}$, where $a_{k+\half}=\frac{z_{k+\half}-z_k+\eta F(z_k)}{\InNorms{z_{k+\half}-z_k+\eta F(z_k)}}$ and $b_{k+\half}=\InAngles{a_{k+\half},z_{k+\half}}$. This is true because $z_{k+\half}=\Pi_\Z(z_k-\eta F(z_k))$, so $\InAngles{z_{k+\half}-z_k+\eta F(z_k), z-z_{k+\half}}\geq 0$ for all $z\in \Z$. (c) $\InAngles{a_{k+1},z}\geq b_{k+1}$, where $a_{k+1}=\frac{z_{k+1}-z_k+\eta F(z_{k+\half})}{\InNorms{z_{k+1}-z_k+\eta F(z_{k+\half})}}$ and $b_{k+1}=\InAngles{a_{k+1},z_{k+1}}$. This is true because $z_{k+1}=\Pi_\Z(z_k-\eta F(z_{k+\half}))$, so $\InAngles{z_{k+1}-z_k+\eta F(z_{k+\half}), z-z_{k+1}}\geq 0$ for all $z\in \Z$.

The ``hardest instance'' of $\Z$ that is consistent with our knowledge of $z_k$, $z_{k+\half}$, and $z_{k+1}$ is when  $\Z$ is exactly the intersection of these three halfspaces. In such case, the squared \projham of $z_{k+1}$ is $\InNorms{F(z_{k+1})}^2-\InAngles{F(z_{k+1}),a_{k+1}}^2\cdot \ind[\InAngles{F(z_{k+1}),a_{k+1}}\geq 0]$, and it is an upper bound of $r^{tan}_{(F,\Z)}(z_{k+1})^2$ for any other consistent $\Z$. Our goal is to prove the \projham is non-increasing even in the "hardest case", that is, to prove the non-negativity of 
\vspace{-.1in}
\begin{align}\label{eq:hardest instance high degree}
    \InNorms{F(z_k)}^2-\InAngles{F(z_k),a_k}^2-\left(\InNorms{F(z_{k+1})}^2-\InAngles{F(z_{k+1}),a_{k+1}}^2\cdot \ind[\InAngles{F(z_{k+1}),a_{k+1}}\geq 0]\right)
\end{align}

\vspace{-.2in}
\paragraph{Low-dimensionality of an EG Update.} As there are only three hyperplanes $\InAngles{a_i,z}\geq b_i$ for $i\in\{k,k+\half,k+1\}$ involved, we can choose a new basis, so that $a_{k+1}=(1,0,\ldots, 0)$, $a_{k+\half}=(\theta_1,\theta_2,0,\ldots, 0)$, and $a_{k}=(\sigma_1,\sigma_2,\sigma_3,0,\ldots, 0)$. As $a_{k+\half}$ and $z_{k+\half}-z_k+\eta F(z_k)$ are co-directed, and $a_{k+1}$ and $z_{k+1}-z_k+\eta F(z_{k+\half})$ are co-directed, \emph{an important property of this change of basis is that the EG update from $z_k$ to $z_{k+1}$ is unconstrained in all coordinates $\ell\geq 4$.} More specifically,
\vspace{-.1in}
$$z_{k+\frac{1}{2}}[\ell]- z_k[\ell]+\eta F(z_{k})[\ell]=0,\quad z_{k+1}[\ell] - z_k[\ell] + \eta F(z_{k+\frac{1}{2}})[\ell]=0,~\forall \ell\geq 4.$$ 

\vspace{-.1in}
 Hence, we can represent all of the coordinates $\ell\geq 4$ with one coordinate in the SOS program similar to the unconstrained case. We still need to keep the first three dimensions, but now we only face a problem in dimension $4$ rather than in  dimension $n$, and we can form a constant size SOS program to search for a certificate of non-negativity for Expression~\eqref{eq:hardest instance high degree}.

In Lemma~\ref{lem:reduce to cone}, we further simplify the instance that we need to consider. In particular, we argue that it is w.l.o.g. to assume that (1) $a_{k+1}$, $a_{k+\half}$, and $a_{k}$ are linear independent and (2) the intersection of the three halfspaces forms a cone, i.e., $b_k=b_{k+\half}=b_{k+1}=0$. Both assumption (1) and (2) reduce the number of variables we need to consider in the SOS program, so a low degree SOS proof is more likely to exist.
To maximally reduce the number of variables, we only included the minimal number of constraints that suffice to derive an SOS proof.

\vspace{-.1in}
\begin{lemma}[Simplification Procedure]\label{lem:reduce to cone}
Let $\mathcal{I}$ be a variational inequality problem for a closed convex set $\Z\subseteq \R^n$ and a monotone and $L$-Lipschitz operator $F:\Z\rightarrow{\R}^n$.
Suppose the EG algorithm has a constant step size $\eta$. Let $z_k$ be the $k$-th iteration of the EG algorithm, $z_{k+\half}$ be the $(k+\half)$-th iteration as defined in~\eqref{def:k+1/2-th step}, and $z_{k+1}$ be the $(k+1)$-th iteration as defined in~\eqref{def:k+1-th step}.

Then either $r^{tan}_{(F,\Z)}(z_k) \geq r^{tan}_{(F,\Z)}(z_{k+1})$, or there exist vectors $\ba_k,\ba_{k+\half},\ba_{k+1}$,$\bz_k$,$\bz_{k+\half}$,$\bz_{k+1}$,
$\bF(\bz_k)$, $\bF(\bz_{k+\half})$, $\bF(\bz_{k+1})\in \R^{N}$ with $N\leq n+5$ that satisfy the following conditions.
\begin{enumerate}
    \item $\ba_k=(\beta_1,\beta_2,1,0,\ldots,0),\ba_{k+\half}=(\alpha, 1, 0,\ldots,0)$, and $\ba_{k+1}=(1,0,\ldots,0)$ for some $\alpha,\beta_1,\beta_2\in \R$. 
    \item 
$\InNorms{\bF(\bz_k)-\frac{\InAngles{F(\bz_k),\ba_k}\cdot \ba_k}{\InNorms{\ba_k}^2}}^2 
-
\InNorms{
\bF(\bz_{k+1})-\frac{\InAngles{\bF(\bz_{k+1}),\ba_{k+1}}\cdot\ba_{k+1}}{\InNorms{\ba_{k+1}}^2}\mathbbm{1}[\InAngles{\bF(\bz_{k+1}),\ba_{k+1}}\ge 0]}^2<0$.
\item Additionally, $\langle \ba_i,\bz_j \rangle \geq 0$ and $\langle \ba_i,\bz_i \rangle = 0$  for all  $i,j \in\{k,k+\half,k+1\}$.
$\ba_{k+\half}$ and $\bz_{k+\half} -\bz_k +\eta \bF(\bz_k)$ are co-directed, i.e., they are colinear and have the same direction,
and
$\ba_{k+1}$ and $\bz_{k+1}-\bz_k + \eta \bF(\bz_{k+\half})$ are co-directed.
\item \begin{align}
\InNorms{\bF(\bz_{k+1}) - \bF(\bz_{k+\half})}^2 \leq&  L^2\InNorms{\bz_{k+1} - \bz_{k+\half}}^2\label{eq:reduced Lipschitz}\\
\InAngles{\bF(\bz_{k+1}) - \bF(\bz_k),\bz_{k+1}-\bz_k} \geq&0 \label{eq:reduced monotone}\\
\InAngles{\ba_k, \bF(\bz_k) }\geq& 0. \label{eq: reduced gradient plane}
\end{align}
\end{enumerate}
\end{lemma}
\vspace{-.1in}

\begin{prevproof}{Lemma}{lem:reduce to cone}
For our proof, it will be more convenient to work with terms $\Ham(z_k)^2, \Ham(z_{k+1})^2$ rather than $\Ham(z_k), \Ham(z_{k+1})$.
Since $\Ham(z_k),\Ham(z_{k+1})\geq 0$, then $\Ham(z_k) - \Ham(z_{k+1}) \geq 0$ iff $\Ham(z_k)^2 - \Ham(z_{k+1})^2 \geq 0$.
For the rest of the proof,
we refer to the property that $\ba_k=(\beta_1,\beta_2,1,0,\ldots,0),\ba_{k+\half}=(\alpha, 1, 0,\ldots,0),\ba_{k+1}=(1,0,\ldots,0)$ as the form property. 

Recall that the $k$-th update of EG is as follows
$z_{k+\frac{1}{2}} = \Pi_\Z\left[z_k-\eta F(z_k)\right]$ and $z_{k+1}  = \Pi_\Z\left[z_k-\eta F(z_{k+\frac{1}{2}})\right]$.
We define the following vectors:
\begin{align}
    -a_k &\in \argmin_{\substack{a\in \unitnormal_\Z(z_k),\\ \InAngles{F(z_k),a}\le 0}}
\|F(z_k) - \InAngles{F(z_k),a}\cdot a\|^2, \label{dfn:a_k} \\
    a_{k+\half} &= z_{k+\half} - z_k + \eta F(z_k), \label{dfn:a_k+1/2}\\
    a_{k+1} &= z_{k+1} - z_k + \eta F(z_{k+\half}).\label{dfn:a_k+1}
\end{align}
For now, let us assume that $a_k$, $a_{k+\half}$, and $a_{k+1}$ satisfy  (i) the form property, and (ii) $\InAngles{a_k,z_k}
=\InAngles{a_{k+\half},z_{k+\half}}
=\InAngles{a_{k+1},z_{k+1}}=0$. 
We use this simple case as the basis of our construction, and will remove these assumptions later.

We set $\ba_k=a_k,\ba_{k+\half}=a_{k+\half},\ba_{k+1}=a_{k+1}$,
$\bz_k=z_k,\bz_{k+\half}=z_{k+\half},\bz_{k+1}=z_{k+1}$, $\bF(\bz_k) = F(z_k)$, $\bF(\bz_{k+\half}) = F(z_{k+\half})$ and $\bF(\bz_{k+1}) = F(z_{k+1})$.
We first argue that Property 4 holds. Since $F(\cdot)$ is monotone and $L$-Lipschitz, Inequality~\eqref{eq:reduced Lipschitz} and~\eqref{eq:reduced monotone} are satisfied.
In addition,
Inequality~\eqref{eq: reduced gradient plane} is satisfied due to the definition of $a_k$.

Next, we show that property 3 holds. By the definition of $\ba_{k+\half}$ (or $\ba_{k+1}$), it is clear that $\ba_{k+\half}$ (or $\ba_{k+1}$) and $ \bz_{k+\half}-\bz_k + \eta \bF(\bz_k) $ (or $\bz_{k+1}-\bz_k + \eta \bF(\bz_{k+\half})$) are co-directed. 
As  
$-a_k \in \unitnormal_\Z(z_k)$, 
\begin{align}
    \InAngles{\ba_k ,\bz }\geq& \InAngles{\ba_k ,\bz_k }= 0, \quad  \bz \in \{\bz_k,\bz_{k+\half},\bz_{k+1}\}.
\end{align}
According to the update rule of the EG algorithm (Equation~\eqref{def:k+1/2-th step} and~\eqref{def:k+1-th step})
, Equation (\ref{dfn:a_k+1/2}), and Equation (\ref{dfn:a_k+1}), we know that for all $z \in \Z$, 
\begin{align}
    &\InAngles{a_{k+\half} ,z - z_{k+\half} } = \InAngles{ z_{k+\half} - z_k + \eta F(z_k), z-z_{k+\half}} \ge 0,\label{eq:project midpoint} \\
    &\InAngles{a_{k+1} ,z - z_{k+1} } = \InAngles{ z_{k+1} - z_k + \eta F(z_{k+\half}), z-z_{k+1}} \ge 0,\label{eq:project endpoint}
\end{align}
which implies that for any $i\in \{k+\half,k+1\}, j \in \{k,k+\half,k+1\}$, $\InAngles{\ba_i ,\bz_j} \geq \InAngles{\ba_i , \bz_i }=0$.

Finally, we verify Property 2.
By Equation~\eqref{eq:project endpoint}, $-\frac{a_{k+1}}{\|a_{k+1}\|} \in \unitnormal(z_{k+1})$,
which in combination with Lemma~\ref{lem:equivalent hamiltonian} implies
\begin{align*}
    &\InNorms{\bF(\bz_{k+1})-\frac{\InAngles{\bF(\bz_{k+1}),\ba_{k+1}}\cdot \ba_{k+1}}{\InNorms{\ba_{k+1}}^2}\mathbbm{1}[\InAngles{\bF(\bz_{k+1}),\ba_{k+1}}\ge 0]}^2 \\
    \geq& \min_{\substack{a\in \unitnormal_\Z(z_{k+1}),\\ \InAngles{F(z_{k+1}),a}\le 0}}
\|F(z_{k+1}) - \InAngles{F(z_{k+1}),a}\cdot a\|^2\\
=& \Ham(z_{k+1})^2.
\end{align*}
According to Lemma~\ref{lem:equivalent hamiltonian} and Equation (\ref{dfn:a_k}), we know that
$\Ham(z_k)^2  = \|F(z_k) - \InAngles{F(z_k),a_k}\cdot a_k\|^2 = \InNorms{\bF(\bz_k) - \frac{\InAngles{\bF(\bz_k),\ba_k}\cdot \ba_k}{\InNorms{\ba_k}^2}}^2$.
If $\Ham(z_k)^2 - \Ham(z_{k+1})^2 < 0$, then
\begin{align}
0&>\Ham(z_k)^2 - \Ham(z_{k+1})^2  \notag\\
&\geq
\InNorms{\bF(\bz_k) - \frac{\InAngles{\bF(\bz_k),\ba_k}\cdot \ba_k}{\InNorms{\ba_k}^2} }^2
-
\InNorms{\bF(\bz_{k+1})-\frac{\InAngles{\bF(\bz_{k+1}),\ba_{k+1}}\cdot \ba_{k+1}}{\InNorms{\ba_{k+1}}^2}\mathbbm{1}[\InAngles{\bF(\bz_{k+1}),\ba_{k+1}}\ge 0]}^2.\notag\\
&\qquad\qquad\qquad \label{eq:bound hamiltonian}
\end{align}
Observe that when vectors $a_k$ and $a_{k+1}$ satisfy the form property, then $\InNorms{\ba_k}=1$, and $\InNorms{\ba_{k+1}}\geq 1$ and Equation~\eqref{eq:bound hamiltonian} is well-defined.

This completes the proof for the case, where (i) vectors $a_k$, $a_{k+\half}$, and $a_{k+1}$ satisfy the form property,
and (ii) $\InAngles{a_k,z_k}
=\InAngles{a_{k+\half},z_{k+\half}}
=\InAngles{a_{k+1},z_{k+1}}=0$. Our next step is to remove the assumptions. We first show how to modify the construction so that for any $a_k,a_{k+\half}, a_{k+1}$, we can construct vectors $\ha_k,\ha_{k+\half}, \ha_{k+1}, \hz_k,\hz_{k+\half}, \hz_{k+1}, \hF(\hz_k),\hF(\hz_{k+\half}),\hF(\hz_{k+1})\in \mathbb{R}^{n+5}$, that satisfy, among other properties, (a)$\InAngles{\ha_k,\hz_k}=\InAngles{\ha_{k+\half},\hz_{k+\half}}=\InAngles{\ha_{k+1},\hz_{k+1}}=0$, and (b) vectors $\ha_k$, $\ha_{k+\half}$, $\ha_{k+1}$ are linear independent. In our final step, we  choose a proper basis to construct vectors $\ba_k,\ba_{k+\half},\ba_{k+1}$,$\bz_k$,$\bz_{k+\half}$,$\bz_{k+1}$,
$\bF(\bz_k)$,$\bF(\bz_{k+\half})$, $\bF(\bz_{k+1})\in \mathbb{R}^{n+5}$ that satisfy all four properties in the statement of Lemma~\ref{lem:reduce to cone}.  

We now present the construction of vectors $\ha_k$, $\ha_{k+\half}$, $\ha_{k+1}$, $\hz_k$, $\hz_{k+\half}$, $\hz_{k+1}$, $\hF(\hz_k)$, $\hF(\hz_{k+\half})$, $\hF(\hz_{k+1})$. High-levelly speaking, we introduce five dummy dimensions.
The purpose of the first dummy dimension is to ensure  property (a). We use the remaining four dummy dimensions to ensure that the newly created vectors $\ha_k$, $\ha_{k+\half}$ and $\ha_{k+1}$ are linearly independent satisfying property (b).
More specifically for parameters $\ell, \epsilon>0$ that we determine later,
we define $\hz_k$, $\hz_{k+\half}$, $\hz_{k+1}$, $\hF(\hz_k)$, $\hF(\hz_{k+\half})$, $\hF(\hz_{k+1})$, $\ha_k$, $\ha_{k+\frac{1}{2}}$, $\ha_{k+1}$ as follows
\begin{align}
    &\hz_i := (-\epsilon^{-1},0,0,0,0,z_i) \qquad\qquad\qquad\qquad\forall i\in\{k,k+\half,k+1\}\\
    &\widehat{F}(\hz_k) :=
    (\frac{\epsilon}{\eta}\cdot \InAngles{a_{k+\half},z_{k+\half}},0,\frac{\epsilon}{\eta},0,\frac{\ell\epsilon}{\eta}, F(z_k)) \\
    &\widehat{F}(\hz_{k+\half}) :=
    (\frac{\epsilon}{\eta}\cdot \InAngles{a_{k+1},z_{k+1}},0,0,\frac{\epsilon}{\eta},\frac{\ell\epsilon}{\eta}, F(z_{k+\half})) \\ 
    &\widehat{F}(\hz_{k+1}) :=
    (\frac{\epsilon}{\eta}\cdot \InAngles{a_{k+1},z_{k+1}},0,0,\frac{\epsilon}{\eta}, \frac{\ell\epsilon}{\eta}, F(z_{k+1}))\\
    &\ha_k := (\epsilon\cdot \InAngles{a_k,z_k},\epsilon,0,0,\ell\epsilon, a_k)\\ 
    &\ha_{k+\half} := (\epsilon\cdot\InAngles{a_{k+\half},z_{k+\half}},0,\epsilon,0,\ell\epsilon,a_{k+\half})\\
    &\ha_{k+1} := (\epsilon\cdot\InAngles{a_{k+1},z_{k+1}},0,0,\epsilon,\ell\epsilon,a_{k+1})
\end{align}

\notshow{
Since the first four coordinates of each point in $\hZ$ are fixed,
it is easy to see that $\widehat{F}$ over $\hZ$ is monotone and $L$-Lipschitz.
We define $\hz_k$ to be $(-1,0,0,0,z_k)$. It is not hard to see by executing EG on $\hZ$ for $\hF$, $\hz_{k+\half}=(-1,0,0,0,z_{k+\half})$ and $\hz_{k+1}=(-1,0,0,0,z_{k+1})$.}
Clearly, $\ha_k$, $\ha_{k+\frac{1}{2}}$, $\ha_{k+1}$ are linear independent,
$\ha_{k+1}$ and $\hz_{k+1} -\hz_k +\eta \hF(\hz_{k+\half})$ are co-directed,
and $\ha_{k+\half}$ and $\hz_{k+\half} -\hz_k +\eta \hF(\hz_k)$ are co-directed.
Note that the following inequalites hold. 
By Equation~\eqref{dfn:a_k}-\eqref{dfn:a_k+1}, it is clear that $-a_k\in \unitnormal(z_k)\subseteq \normal(z_k)$, $-a_{k+\half}\in \normal(z_{k+\half})$ and $-a_{k+1}\in \normal(z_{k+1})$, which further implies
\begin{align}
&\langle \ha_i,\hz_j \rangle = \langle a_i, z_j \rangle - \langle a_i, z_i \rangle \geq 0,
\qquad  \forall i,j\in\{k,k+\half,k+1\} \label{eq:before limit start}\\
&\langle \ha_i,\hz_i \rangle = 0,
\qquad\qquad \qquad\qquad\qquad~~  \forall i\in\{k,k+\half,k+1\}\\
& 
\InNorms{\hF(\hz_{k+1})-\hF(\hz_{k+\half})}^2=\InNorms{F(z_{k+1})-F(z_{k+\half})}^2\leq  L^2\InNorms{z_{k+1} - z_{k+\half}}^2=\InNorms{\hz_{k+1} - \hz_{k+\half}}^2
\\
&\InAngles{\hF(\hz_{k+1})-\hF(\hz_k), \hz_{k+1} - \hz_k}
=
\InAngles{F(z_{k+1})-F(z_k), z_{k+1} - z_k}\geq 0
\end{align}

Moreover, $\InAngles{\ha_k, \hF(\hz_k)}=
\frac{\epsilon^2}{\eta}\cdot\left(\InAngles{a_{k},z_{k}}\InAngles{a_{k+\half},z_{k+\half}}+\ell^2\right) + \InAngles{a_k, F(z_k)}$. We choose $\ell$ to be sufficiently large so that $\InAngles{a_{k},z_{k}}\InAngles{a_{k+\half},z_{k+\half}}+\ell^2\geq 0$. Hence, for our choice of $\ell$, \begin{equation}\label{eq:limit eq start}
    \InAngles{\ha_k, \hF(\hz_k)}\geq \InAngles{a_k, F(z_k)}\geq 0.
\end{equation}

We define function \begin{align*}
\widehat{H}(\hz_k) &:=
\InNorms{\hF(\hz_k)-\frac{\InAngles{\hF(\hz_k),\ha_k}}{\InNorms{\ha_k}^2}\cdot\ha_k}^2\\
&= \InNorms{\hF(\hz_k)-\frac{\frac{\epsilon^2}{\eta}\cdot\left(\InAngles{a_{k},z_{k}}\InAngles{a_{k+\half},z_{k+\half}}+\ell^2\right) + \InAngles{a_k, F(z_k)}}{\epsilon^2(\InAngles{a_k,z_k}^2+1+\ell^2)+\InNorms{a_k}^2}\cdot \ha_k}^2.
\end{align*}

Define $f(\epsilon):= \frac{\frac{\epsilon^2}{\eta}\cdot\left(\InAngles{a_{k},z_{k}}\InAngles{a_{k+\half},z_{k+\half}}+\ell^2\right) + \InAngles{a_k, F(z_k)}}{\epsilon^2(\InAngles{a_k,z_k}^2+1+\ell^2)+\InNorms{a_k}^2}$. We can simplify $\widehat{H}(\hz_k)$ to be
\begin{align*}
\epsilon^2\cdot \InNorms{ 
    \left(\frac{\InAngles{a_{k+\half},z_{k+\half}}}{\eta}-f(\epsilon)\InAngles{a_k,z_k} ,-f(\epsilon),\frac{1}{\eta},0,\frac{\ell}{\eta}-f(\epsilon)\ell \right)}^2+\InNorms{F(z_k)-f(\epsilon)\cdot a_k}^2.
\end{align*}
When $a_k=(0,\ldots,0)$,
 $\Ham(z_k)^2=\InNorms{F(z_k)}^2$ (Equation~\eqref{dfn:a_k}) and $f(\epsilon)=\frac{\ell^2}{\eta(1+\ell^2)}$ for any $\epsilon>0$. Therefore, $\lim_{\epsilon\rightarrow 0^+}\widehat{H}(\hz_k)=\InNorms{F(z_k)}^2=\Ham(z_k)^2$, when $a_k=(0,\ldots,0)$. 
When $a_k\neq(0,\ldots,0)$, $\Ham(z_k)^2=\InNorms{F(z_k)-\frac{\InAngles{a_k,F(z_k)}}{\InNorms{a_k}^2}\cdot a_k}^2$ (Equation~\eqref{dfn:a_k}) and $\lim_{\epsilon\rightarrow 0^+}f(\epsilon) = \frac{\InAngles{a_k,F(z_k)}}{\InNorms{a_k}^2}$, so $\lim_{\epsilon\rightarrow 0^+}\widehat{H}(\hz_k)=\InNorms{F(z_k)-\frac{\InAngles{a_k,F(z_k)}}{\InNorms{a_k}^2}\cdot a_k}^2=\Ham(z_k)^2$.

\notshow{
\begin{align*}
&\widehat{H}(\hz_k)=\InNorms{\hF(\hz_k)-\frac{\InAngles{\hF(\hz_k),\ha_k}}{\InNorms{\ha_k}^2}\cdot\ha_k}^2=\InNorms{\hF(\hz_k)-\frac{\epsilon^2\InParentheses{\InAngles{a_k,z_k}^2+1}+\InAngles{a_k,F(z_k)}}{\epsilon^2\InParentheses{\InAngles{a_k,z_k}^2+1}+ \InNorms{a_k}^2}\cdot\ha_k}^2 \\
=&
\epsilon^2\frac{\InParentheses{\InAngles{a_k,z_k}^2+1}\InParentheses{\InNorms{a_k}^2-\InAngles{a_k,F(z_k)}}^2}{\InParentheses{\epsilon^2\InParentheses{\InAngles{a_k,z_k}^2+1}+ \InNorms{a_k}^2}^2}  +
\InNorms{F(z_k)-\frac{\epsilon^2\InParentheses{\InAngles{a_k,z_k}^2+1}+\InAngles{a_k,F(z_k)}}{\epsilon^2\InParentheses{\InAngles{a_k,z_k}^2+1}+ \InNorms{a_k}^2}\cdot a_k}^2
\end{align*}
Clearly, when $a_k\neq (0,\ldots,0)$, $\widehat{H}(\hz_k)$ is continuous with respect to $\epsilon$ and $\lim_{\epsilon\rightarrow 0^+}\widehat{H}(\hz_k)$ = $
\InNorms{F(z_k)-\frac{\InAngles{a_k,F(z_k)}}{\InNorms{a_k}^2}\cdot a_k}^2$ = $\Ham(z_k)^2$.}

Similarly, we define function 
\begin{align*}
    \widehat{H}(\hz_{k+1}) &:= \InNorms{\hF(\hz_{k+1})-\frac{\InAngles{\hF(\hz_{k+1}),\ha_{k+1}}\cdot \ha_{k+1}}{\InNorms{\ha_{k+1}}^2}\mathbbm{1}\left[\InAngles{\hF(\hz_{k+1}),\ha_{k+1}}\ge 0\right]}^2\\
    &=\InNorms{\hF(\hz_{k+1})-\frac{\frac{\epsilon^2}{\eta}\cdot\left(\InAngles{a_{k+1},z_{k+1}}^2+1+\ell^2\right)+\InAngles{F(z_{k+1}),a_{k+1}}}{\epsilon^2(\InAngles{a_{k+1},z_{k+1}}^2+1+\ell^2)+\InNorms{a_{k+1}}^2}\cdot \ha_{k+1}\ind\left[\InAngles{\hF(\hz_{k+1}),\ha_{k+1}}\geq 0\right]}^2.
\end{align*}

Define $g(\epsilon):=\frac{\frac{\epsilon^2}{\eta}\cdot\left(\InAngles{a_{k+1},z_{k+1}}^2+1+\ell^2\right)+\InAngles{F(z_{k+1}), a_{k+1}}}{\epsilon^2(\InAngles{a_{k+1},z_{k+1}}^2+1+\ell^2)+\InNorms{a_{k+1}}^2}$, and we can simplify $\widehat{H}(\hz_{k+1})$ to be 
\begin{align*}
    \epsilon^2\left(\frac{1}{\eta}-g(\epsilon)\ind\left[\InAngles{\hF(\hz_{k+1}),\ha_{k+1}}\geq 0\right]\right)^2 &\cdot\InNorms{\InParentheses{\InAngles{a_{k+1},z_{k+1}},0,0,1,\ell}}^2\\
   &+\InNorms{F(z_{k+1})-g(\epsilon)\cdot a_{k+1}\ind\left[\InAngles{\hF(\hz_{k+1}),\ha_{k+1}}\geq 0\right]}^2.
\end{align*}

When $a_{k+1}=(0,\ldots, 0)$, we have $ g(\epsilon)=\frac{1}{\eta}$ and $\InAngles{\hF(\hz_{k+1}),\ha_{k+1}}\geq 0$, so $\widehat{H}(\hz_{k+1})=\InNorms{F(z_{k+1})}^2\geq \Ham(z_{k+1})^2$. When $a_{k+1}\neq (0,\ldots, 0)$, we have $\lim_{\epsilon\rightarrow 0^+} g(\epsilon)=\frac{\InAngles{F(z_{k+1}),a_{k+1}}}{\InNorms{a_{k+1}}^2}$ and $\lim_{\epsilon\rightarrow 0^+} \ind\left[\InAngles{\hF(\hz_{k+1}),\ha_{k+1}}\geq 0\right]= \ind\left[\InAngles{F(z_{k+1}),a_{k+1}}\geq 0\right]$,\footnote{This is because $\InAngles{\hF(\hz_{k+1}),\ha_{k+1}}$ is never smaller than $\InAngles{F(z_{k+1}),a_{k+1}}$, and the function $\ind[x\geq 0]$ is right continuous.} hence $$\lim_{\epsilon\rightarrow 0^+}\widehat{H}(\hz_{k+1})=\InNorms{F(z_{k+1})-\frac{\InAngles{F(z_{k+1}),a_{k+1}}}{\InNorms{a_{k+1}}^2}\cdot a_{k+1}\ind\left[\InAngles{F(z_{k+1}),a_{k+1}}\geq 0\right]}^2\geq \Ham(z_{k+1})^2.$$ The last inequality is because $\frac{-a_{k+1}}{\InNorms{a_{k+1}}}\in \unitnormal_{\mathcal{Z}}(z_{k+1})$ and Lemma~\ref{lem:equivalent hamiltonian}.

\notshow{
$\InNorms{\hF(\hz_{k+1}) - \hF(\hz_{k+\half})}^2$
 are continuous with respect to $\epsilon$ and 
\begin{align}
\lim_{\epsilon\rightarrow 0^+}\InNorms{\hF(\hz_{k+1}) - \hF(\hz_{k+\half})}^2 = \InNorms{F(z_{k+1}) - F(z_{k+\half})}^2,\notag\\
\lim_{\epsilon\rightarrow 0^+}
\widehat{H}(\hz_{k+1}) 
\geq \Ham(z_{k+1})^2 \label{eq:lim at end}.
\end{align}

Observe that if $z_{k+\half}=z_{k+1}$,
then $z_{k+1}=\prod_\Z\InParentheses{z_k-\eta F(z_{k+\half})}=\prod_{PP}\InParentheses{z_k-\eta F(z_{k+\half})}$ and
$H_{(F,\Z)}(z_{k+1})\geq H_{(F,\Z)}(z_{k+1})$,
a contradiction.
Since by assumption we have $\Ham(z_{k+1})^2 - \Ham(z_k)^2>0$,
then by choosing sufficiently small $\epsilon>0$,
we can ensure that 
\begin{align}
\left|\InNorms{\hF(\hz_{k+1}) - \hF(\hz_{k+\half})}^2 - \InNorms{F(z_{k+1}) - F(z_{k+\half})}^2\right| \leq L^2\InNorms{\hz_{k+1}-\hz_{k+\half}}^2, \label{eq:limit eq start}\\
\widehat{H}(\hz_{k}) - \Ham(z_k)^2
= \widehat{H}(\hz_{k}) - \lim_{\epsilon\rightarrow 0^+}
\widehat{H}(\hz_{k}))  \leq  \frac{\Ham(z_{k+1})^2 - \Ham(z_k)^2}{4}.\label{eq:limit z_k ham} \\
\Ham(z_{k+1})^2 - \widehat{H}(\hz_{k+1})
\leq \lim_{\epsilon\rightarrow 0^+}
\widehat{H}(\hz_{k+1}))  - \widehat{H}(\hz_{k+1})  \leq  \frac{\Ham(z_{k+1})^2 - \Ham(z_k)^2}{4}.\label{eq:limit z_k+1 ham} \end{align}
}

If $\Ham(z_k)^2-\Ham(z_{k+1})^2 < 0$, then 
$$\lim_{\epsilon\rightarrow 0^+} \left(\widehat{H}(\hz_{k})-\widehat{H}(\hz_{k+1})\right) \leq \Ham(z_k)^2-\Ham(z_{k+1})^2  <0.$$
Thus, we can choose a sufficiently small $\epsilon$ so that $\widehat{H}(\hz_{k})-\widehat{H}(\hz_{k+1}) <0$.
Together with Inequalities~\eqref{eq:before limit start}-\eqref{eq:limit eq start}, we can satisfy  all properties excluding the form property using $\hz_k$, $\hz_{k+\half}$, $\hz_{k+1}$, $\ha_k$, $\ha_{k+\half}$, $\ha_{k+1}$, $\hF(\hz_k)$, $\hF(\hz_{k+\half})$, $\hF(\hz_{k+1})$.

Now we show how to make the vectors also satisfy the form property.
We perform a change of basis so that vectors $\ha_k$, $\ha_{k+\half}$ and $\ha_{k+1}$ only depend on the first three coordinates.
We use the Gram–Schmidt process to generate a basis,
where vectors $\ha_k$, $\ha_{k+\half}$, $\ha_{k+1}$ all lie in the span of the first three vector of the new basis.
More formally,
let $N=n+5$ and $\{b_i\}_{i\in[N]}$ be a sequence of orthonormal vectors produced by the Gram-Schmidt process on ordered input $\ha_{k+1}$, $\ha_{k+\half}$, $\ha_k$ and $\{e_i\}_{i\in[N]}$.
Let $Q$ be the $N\times N$ matrix,
where the $i$-th row of $Q$ is vector $b_i$.
Observe that any vector $z\in \R^{N}$ written in the basis $\{e_i\}_{i\in[N]}$ can be represented by the basis $\{b_i\}_{i\in [N]}$ with coefficients $Q\cdot z$.

Let $\ba_k = Q \cdot \ha_k$, $\ba_{k+\half} = Q \cdot \ha_{k+\half}$, 
$\ba_{k+1} = Q \cdot \ha_{k+1}$, 
$\bz_k = Q \cdot \hz_k$, $\bz_{k+\half} = Q \cdot \hz_{k+\half}$, 
$\bz_{k+1} = Q \cdot \hz_{k+1}$,
$\bF(\bz_k) = Q \cdot \hF(\hz_k)$,
$\bF(\bz_{k+\half}) = Q \cdot \hF(\hz_{k+\half})$ and $\bF(\bz_{k+1}) = Q \cdot \hF(\hz_{k+1})$.
Note that $\ba_k$ is the coefficients of $\ha_k$ written in the basis $\{b_i\}_{i\in[N]}$ and 
similar reasoning holds for the rest of the defined vectors.
Clearly, for any $\hz,\hz'\in \R^{N}$, if we define $\bz=Q\cdot \hz$ and $\bz' = Q \cdot \hz'$, then we have
\begin{align}
    \langle \bz,\bz' \rangle = \langle \hz,\hz' \rangle. \label{eq:equal after rotation}
\end{align} 
Combining Equation~\eqref{eq:equal after rotation} and the fact that all properties but the form property hold for vectors $\ha_k$, $\ha_{k+\half}$, $\ha_{k+1}$, $\hz_k$, $\hz_{k+\half}$, $\hz_{k+1}$, $\hF(\hz_k)$, $\hF(\hz_{k+\half})$ and $\hF(\hz_{k+1})$,
we conclude that the same properties also hold for vectors $\ba_k$, $\ba_{k+\half}$, $\ba_{k+1}$, $\bz_k$, $\bz_{k+\half}$, $\bz_{k+1}$, $\bF(\bz_k)$, $\bF(\bz_{k+\half})$ and $\bF(\bz_{k+1})$.

Finally, by properties of the Gram-Schmidt process, the order of the vector in its input, and the fact that
vectors $\ha_k,\ha_{k+\half}$ and $\ha_{k+1}$ are linearly independent, we have
$\ha_{k+1}\in \Span(b_1)$, $\ha_{k+\half}\in \Span(b_1,b_2)$, $\ha_k\in \Span(b_1,b_2,b_3)$,
$\InAngles{\ha_{k+1},b_1}>0$,  $\InAngles{\ha_{k+\half},b_2}>0$
and $\InAngles{\ha_k,b_3}>0$.
Thus $\ba_k=(\beta_1,\beta_2,b,0,\ldots,0)$,
$\ba_{k+\half}=(\alpha,c,0,\ldots,0)$ 
and $\ba_{k+1}=(d,0,\ldots,0)$,
where $\beta_1, \beta_2, \alpha\in \R$ and $b,c,d > 0$.
By properly scaling $\ba_k$, $\ba_{k+\half}$ and $\ba_{k+1}$, we can make them satisfy the form property. This completes the proof.
\end{prevproof}

\notshow{
\begin{proof}
Recall that the $k$-th update of EG is as follows:
\begin{align}
    z_{k+\frac{1}{2}} &= \Pi_\Z\left[z_k-\eta F(z_k)\right]=\arg \min_{z\in \Z} \| z-\left(z_k-\eta F(z_k)\right)\|\\
    z_{k+1} & = \Pi_\Z\left[z_k-\eta F(z_{k+\frac{1}{2}})\right]= \arg \min_{z\in \Z} \left \| z-\left(z_k-\eta F(z_{k+\frac{1}{2}})\right)\right \|
\end{align}
We define the following vectors
\begin{align}
    -a_k &\in \argmin_{\substack{a\in \unitnormal_\Z(z_k),\\ \InAngles{F(z_k),a}\le 0}}
\|F(z_k) - \InAngles{F(z_k),a}a\|^2 \label{dfn:a_k} \\
    a_{k+\half} &= z_{k+\half} - z_k + \eta F(z_k) \label{dfn:a_k+1/2}\\
    a_{k+1} &= z_{k+1} - z_k + \eta F(z_{k+\half}).\label{dfn:a_k+1}
\end{align}
For now, let us assume that $a_k$, $a_{k+\frac{1}{2}}$, and $a_{k+1}$ are linear independent and have the form in the statement of the lemma and that $\InAngles{a_k,z_k}$, $\InAngles{a_{k+\half},z_{k+\half}}$, $\InAngles{a_{k+1},z_{k+1}}$ are all zero.
Then the instance $\mathcal{I}'$ coincides with instance $\mathcal{I}$, 
that is we set $\ha_k=a_k,\ha_{k+\half}=\ha_{k+\half},\ha_{k+1}=a_{k+1}$,
$\hz_k=z_k,\hz_{k+\half}=\hz_{k+\half},\hz_{k+1}=z_{k+1}$, $\hF(\cdot) = F(\cdot)$ and $\hZ = \Z$.
It is obvious that $\widehat{F}$ is monotone and $L$-Lipschitz.
We verify the other properties of $a_k$, $a_{k+\frac{1}{2}}$, $a_{k+1}$ as follows. 

By Equation (\ref{dfn:a_k}), we know that $\InAngles{F(z_k),a_k} \ge 0$. We also know that $-a_k \in \unitnormal_\Z(z_k)$ and thus
\begin{align}
    \InAngles{a_k ,z }\geq& \InAngles{a_k ,z_k } \quad \forall z \in \Z.
\end{align}
According to the update of EG (\ref{alg:EG}), Equation (\ref{dfn:a_k+1/2}), and Equation (\ref{dfn:a_k+1}), we know that for all $z \in \Z$, 
\begin{align}
    &\InAngles{a_{k+\half} ,z - z_{k+\half} } = \InAngles{ z_{k+\half} - z_k + \eta F(z_k), z-z_{k+\half}} \ge 0, \\
    &\InAngles{a_{k+1} ,z - z_{k+1} } = \InAngles{ z_{k+1} - z_k + \eta F(z_{k+\half}), z-z_{k+1}} \ge 0.
\end{align}
Thus $-a_{k+\half} \in \unitnormal(z_{k+\half})$ and $-a_{k+1} \in \unitnormal(z_{k+1})$. Additionally, we have
\begin{align}
    \InAngles{a_{k+\frac{1}{2}},z_k-\eta F(z_k) - z_{k+\half}} &\le 0,\\
    \InAngles{a_{k+1},z_k-\eta F(z_{k+\half}) - z_{k+1}} &\le 0.
\end{align}
By the definition of $a_{k+\frac{1}{2}}$ and $a_{k+1}$, we know that for any $a_{k+\half}^\perp$ and $a_{k+1}^{\perp}$ such that $\InAngles{a_{k+\half}^\perp, a_{k+\half} }=\InAngles{a_{k+1}^\perp, a_{k+1} }=0$:
\begin{align*}
&\InAngles{a_{k+\half}^\perp ,z_k - F(z_k) - z_{k+\half}} = 0,\\
&\InAngles{a_{k+1}^\perp ,z_k - F(z_{k+\half}) -  z_{k+1}} = 0.
\end{align*}

According to Definition \ref{def:projected hamiltonian} and Equation (\ref{dfn:a_k}), we know 
\begin{align}
    H(z_k) = \|F(z_k) - \InAngles{F(z_k),-a_k} (-a_k)\|^2 = \|F(z_k) - \InAngles{F(z_k),a_k} a_k\|^2.
\end{align}
Note that $-a_{k+1} \in \unitnormal(z_{k+1})$. If $\InAngles{F(z_{k+1}),a_{k+1}}\ge 0$, then
\begin{align}
    &\|F(z_{k+1})-\InAngles{F(z_{k+1}),a_{k+1}}a_{k+1}\mathbbm{1}[\InAngles{F(z_{k+1}),a_{k+1}}\ge 0]\|^2 \\
    &= \|F(z_{k+1})-\InAngles{F(z_{k+1}),a_{k+1}}a_{k+1}\|^2 \\
    &\ge \min_{\substack{a\in \unitnormal_\Z(z_{k+1}),\\ \InAngles{F(z_{k+1}),a}\le 0}}
\|F(z_{k+1}) - \InAngles{F(z_{k+1}),a}a\|^2\\
    & = H(z_{k+1}).
\end{align}
Otherwise $\InAngles{F(z_{k+1}),a_{k+1}}< 0$, then 
\begin{align}
    &\|F(z_{k+1})-\InAngles{F(z_{k+1}),a_{k+1}}a_{k+1}\mathbbm{1}[\InAngles{F(z_{k+1}),a_{k+1}}\ge 0]\|^2 \\
    &= \|F(z_{k+1})\|^2 \\
    & \ge H(z_{k+1}).
\end{align}
Thus we know 
\begin{align*}
&H(z_k) - H(z_{k+1}) \\
&\geq
\|F(z_k)-\InAngles{F(z_k),a_k}a_k\|^2 - \|F(z_{k+1})-\InAngles{F(z_{k+1}),a_{k+1}}a_{k+1}\mathbbm{1}[\InAngles{F(z_{k+1}),a_{k+1}}\ge 0]\|^2.
\end{align*}
This completes the proof for the case where $a_k$, $a_{k+\frac{1}{2}}$, $a_{k+1}$ are linear independent and $\InAngles{\ha_k,\hz_k}$,$\InAngles{\ha_{k+\half}, \hz_{k+\half}}$, $\InAngles{\ha_{k+1},\hz_{k+1}}$ are all equal to $0$. 
In a high-level idea, we introduce four dummy dimensions,
the purpose of the dummy dimension is to ensure that the updated equations $\InAngles{\ha_k,\hz_k}$,$\InAngles{\ha_{k+\half}, \hz_{k+\half}}$, $\InAngles{\ha_{k+1},\hz_{k+1}}$ are all equal to zero and the remaining three dummy dimensions ensure that our vectors are linearly independent.

Now we return to the case where $a_k$, $a_{k+\frac{1}{2}}$, $a_{k+1}$ are linear dependent or when $\InAngles{\ha_k,\hz_k},\InAngles{\ha_{k+\half},\hz_{k+\half}},\InAngles{\ha_{k+1},\hz_{k+1}}$ are not all equal to $0$. The idea here is to introduce four more dimensions. Specifically, we define $\hZ \subseteq \R^{n+4}$ and $\widehat{F}: \hZ \rightarrow \R^{n+4}$ as follows
\begin{align}
    &\hZ := \{(-1,0,0,0,z): z\in \Z \}.\\
    &\widehat{F}((-1,0,0,0,z)) :=
    \begin{cases}
    (\InAngles{a_k,z_k},1,0,0, F(z_k)), &\text{ if $z=z_k$ } \\
    (\InAngles{a_{k+\half},z_{k+\half}},0,1,0, F(z_{k+\half})), &\text{ if $z=z_{k+\half}$ } \\
    (\InAngles{a_{k+1},z_{k+1}},0,0,1, F(z_{k+1})), &\text{ if $z=z_{k+1}$ } \\
    (0,0,0,0, F(z)), &\forall z\in \Z\backslash\{z_k,z_{k+\half},z_{k+1}\}.
 \\
    \end{cases}
\end{align}
Since the first four coordinates of each point in $\hZ$ are fixed,
it is easy to see that $\widehat{F}$ over $\hZ$ is monotone and $L$-Lipschitz.
\yangnote{We define $\hz_k$ to be $(-1,0,0,0,z_k)$. It is not hard to see by executing EG on $\hZ$ for $\hF$, $\hz_{k+\half}=(-1,0,0,0,z_{k+\half})$ and $\hz_{k+1}=(-1,0,0,0,z_{k+1})$.}
We define $\ha_k$, $\ha_{k+\frac{1}{2}}$, $\ha_{k+1}$ as follows:
\begin{align}
    \ha_k &= (\InAngles{a_k,z_k},1,0,0,a_k),\\
    \ha_{k+\half} &= (\InAngles{a_{k+\half},z_{k+\half}},0,1,0,a_{k+\half}),\\
    \ha_{k+1} &= (\InAngles{a_{k+1},z_{k+1}},0,0,1,a_{k+1}).
\end{align}
Clearly, $\ha_k$, $\ha_{k+\frac{1}{2}}$, $\ha_{k+1}$ are linear independent.
Note that $\forall \hz=(-1,0,0,0,z)\in \hZ$
\begin{align}
\langle \ha_i,\hz \rangle =& \langle a_i, z \rangle - \langle a_i, z_i \rangle \geq 0
\qquad &\forall i\in\{k,k+\half,k+1\},\\
\langle \ha_i,\hz_i \rangle =& 0
\qquad &\forall i\in\{k,k+\half,k+1\},\\
\InAngles{\ha_k, \hF(\hz_k)} =& \InAngles{a_k, F(z_k)}\geq 0
\\
\ha_i =& \hz_i - \hz_k -\hF(\hz_{i-\half})
&\forall i\in\{k+\half,k+1\}
\end{align} Hence, all the properties are satisfied by $\hz_k,\hz_{k+\half},\hz_1$ and $\ha_k,\ha_{k+\half},\ha_{k+1}$.

Now we perform a change of base so that vectors $\ha_k,\ha_{k+\half}$ and $\ha_{k+1}$ only depend on the first three coordinates.
We use Gram–Schmidt process to generate a basis,
where vectors $\ha_k,\ha_{k+\half},\ha_{k+1}$ all lie in the span of the first three vector of the new basis.
More formally,
let $\{b_i\}_{i\in[N]}$ be a sequence of orthonormal vectors produced by the Gram-Schmidt process on ordered input $\ha_{k+1},\ha_{k+\half},\ha_k$ and $\{e_i\}_{i\in[N]}$.
By carefully choosing the directions of vectors $\{b_1,b_2,b_3\}$,
we also have that $\InAngles{b_1,\ha_k}\geq 0$,
$\InAngles{b_2,\ha_{k+\half}}\geq 0$
and $\InAngles{b_3,\ha_{k+1}}\geq 0$.
Let $Q$ be the $N\times N$ matrix,
where the $i$-th row of $Q$ is vector $b_i$.
Since $Q$ is orthonormal,
the inverse matrix of $Q$ is $Q^T$.
Observe that a vector $z\in \R^N$ written in base $\{e_i\}_{i\in[N]}$ can be written into base $\{b_i\}_{i\in [N]}$ with coefficients $Q\cdot z$.

We consider the set $\bZ = \{Q\cdot z: z\in \hZ\}$ and operator $\bF(\bz)=Q \cdot \hF(Q^T\cdot \bz):\bZ\rightarrow \R$.
Set $\bZ$ contains the vectors in $\Z$ written in base $\{b_i\}_{i\in[N]}$ and operator $\bF(\bz)$ is the outcome of $\hF(\hz)$ written in base $\{b_i\}_{i\in[N]}$,
where $\hz$ satisfies $\bz=Q\cdot \hz$,
that is $\bz$ is $\hz$ written in base $\{b_i\}_{i\in[N]}$.
Since $Q$ is orthonormal, function $f(\hz)=Q\cdot \hz:\hZ\rightarrow \bZ$ is a bijection mapping between set $\hZ$ and set $\bZ$ with inverse mapping $f^{-1}(\bz)=Q^T\cdot \bz:\bZ\rightarrow \hZ$.
Note that the only difference between using operator $\hF(\hz)$ on $\hz\in \hZ$ is the same as using $\bF(\bz)$ on $\bz\in \bZ$ just written in different base, which implies that we can couple the execution of EG on operator $\bF(\cdot)$ on set $\bZ$ with the execution of EG on operator $\hF(\cdot)$ on set $\hZ$.
For completeness reasons, we include a proof that operator $\bF(\cdot)$ is monotone and that the updates when using operator $\bF(\cdot)$ are the same as using operator $\hF(\cdot)$ just written in different bases.

Let $\ba_k = Q \cdot \ha_k$, $\ba_{k+\half} = Q \cdot \ha_{k+\half}$,
$\ba_{k+1} = Q \cdot \ha_{k+1}$,
$\bz_k = Q \cdot \hz_k$, $\bz_{k+\half} = Q \cdot \hz_{k+\half}$ and
$\bz_{k+1} = Q \cdot \hz_{k+1}$.
Vector $\bz_k$ is vector $\hz_k$ written in base $\{b_i\}_{i\in[N]}$,
similar reasoning holds for the rest of the defined vectors.
Clearly, for any $\hz,\hz'\in \hZ$ and their corresponding points $\bz=Q\cdot \hz,\bz'=Q\cdot \hz\in \bZ$:
\begin{align*}
    \langle \bz,\bz' \rangle = \langle \hz,\hz' \rangle \label{eq:equal after rotation}
\end{align*} 

Combining Equation~\eqref{eq:equal after rotation} and that properties of the Lemma hold for points $\ha_k,\ha_{k+\half},\ha_{k+1},\hz_k,\hz_{k+\half},\hz_{k+1}$, set $\hZ$ and operator $\hF(\cdot)$,
we conclude that the desired properties also hold for $\ba_k,\ba_{k+\half},\ba_{k+1},\bz_k,\bz_{k+\half},\bz_{k+1}$, set $\bZ$ and operator $\bF(\cdot)$.

Moreover, by properties of the Gram-Schmidt process and the order of the vector in its input,
since vectors $\ha_k,\ha_{k+\half}$ and $\ha_{k+1}$ are linearly independent,
$\ha_{k+1}\in Span(b_1),\ha_{k+\half}\in Span(b_1,b_2)$ and $\ha_k\in Span(b_1,b_2,b_3)$ and
$\InAngles{\ha_{k+1},b_1}>0$, $\InAngles{\ha_{k+\half},b_2}>0$,
and $\InAngles{\ha_k,b_3}>0$.
Thus $\ba_k=Q\ha_k=(\beta_1,\beta_2,b,0,\ldots,0)$,
$\ba_{k+\half}=(\alpha,c,0,\ldots,0)$
and $\ba_k=(d,0,\ldots,0)$,
where $\beta_1,\beta_2, \alpha\in \R$ and $b,c,d > 0$.
If $b\neq 1$ or $c\neq 1$ or $d \neq 1$,
then by properly scaling vectors $\ba_k$,$\ba_{k+\half}$ and $\ba_{k+1}$ we can finish the proof.

For completeness reasons, we verify that that operator $\bF(\cdot)$ is monotone and that the updates of EG with operator $\bF(\cdot)$ starting from $\bz_k$ coincide with the updates of EG on $\hF(\cdot)$ on set $\hZ$ written in different bases. 
First we show that $\bF(\cdot)$ is a monotone operator.
For any $\bz,\bz'\in \bZ$:
\begin{align*}
\InAngles{\bF(\bz) - \bF(\bz'), \bz-\bz'}
=&\InAngles{Q\cdot \left(\hF(Q^T\cdot\bz+z^*) - \hF(Q^T\cdot\bz'+z^*)\right), Q\cdot \left(Q^T\left(\bz+z^*-(\bz'+z^*)\right)\right)} \\
=&\InAngles{\hF(Q^T\bz+z^*) - \hF(Q\bz'+z^*),\left(Q^T\bz+z^*\right)-\left(Q^T\bz'+z^*\right)}\geq 0
\end{align*}
In the first equality we used that the since $Q$ is orthonormal,
its inverse matrix is $Q^T$ and in the second equation we used the fact that $F$ is a monotone operator.
Consider the points $\bz_k = Q(\hz_k-z^*)\in \bZ$ and let $\bz_{k+\half}$ and $\bz_{k+1}$ be the points according to the EG update when starting at $\bz_k$ with operator $\bF(\cdot)$ and convex set $\bZ$.
We are going to show that $\bz_{k+\half}=Q\left( z_{k+\half}- \bz^*\right)$ and $\bz_{k+1}=Q\left( z_{k+1}- \bz^*\right)$.
According to EG update rule Equation~\eqref{alg:EG} we have that

\begin{align*}
\bz_{k+\half} = 
\arg \min_{\bz\in \bZ} \| \bz-\left(\bz_k-\eta \bF(\bz_k)\right)\|
= &
\arg \min_{\bz\in \bZ} \| \bz-\left(Q(\hz_k-z^*)-\eta Q\cdot \hF(Q^T\cdot\bz_k+z^*\right)\| \\
=& Q\left( \arg \min_{\hz\in \hZ} \| Q\cdot\left(\hz-z^*\right)-\left(Q(\hz_k-z^*)-\eta Q\cdot \hF(\hz_k)\right)\| - z^*\right) \\
=& Q\left( \arg \min_{\hz\in \hZ} \| Q\cdot\left(\hz - \hz_k- \hF(\hz_k\right)\| - z^*\right) \\
=& Q\left( \arg \min_{\hz\in \hZ} \| \hz - \hz_k- \hF(\hz_k)\| - z^* \right)\\
=& Q\left( \hz_{k+\half} - z^*\right) 
\end{align*}

In the second equality we substitute $\bz\in \bZ$ with $Q(\hz-\hz^*)$ and minimized over $\hz\in\hZ$.
In the fourth equation we used that $Q$ is orthonormal matrix and in the final equality we used that $\hz_{k+\half} = 
\arg \min_{\hz\in \hZ} \| \hz-\left(\hz_k-\eta \hF(\hz_k)\right)\|$.
We can similarly show that $\bz_{k+1}= Q\left(\hz_{k+1} - z^* \right)$.
\end{proof}
}

\notshow{
\begin{figure}[h!]
\colorbox{MyGray}{
\begin{minipage}{0.97\textwidth} {
\noindent\textbf{Feasibility of set of polynomial Inequalities:}
\begin{itemize}
\item $v=\{v_i\}_{i\in[N]}$ are the independent variables.
\end{itemize}
\begin{equation*}
\begin{array}{ll@{}ll}
\text{find} & \displaystyle v \\
\text{s.t.} & \displaystyle f_0(v) < 0\\
&f_i(v) \geq 0 \qquad\qquad \forall i \in [M]\\
&h_i(v) = 0 \qquad\qquad \forall i \in [M']\\
\end{array}
\end{equation*}
\textbf{Constructing a Witness through a SOS Program:}
\begin{itemize}
\item $v=\{v_i\}_{i\in[N]}$ are the independent variables.
\end{itemize}
\begin{equation*}
\begin{array}{ll@{}ll}
\text{find} & \displaystyle \{d_i\}_{i\in[M]},\{d'_i\}_{i\in[N]}\\
\text{s.t.} &  \displaystyle f_0(v) + \sum_{i\in[M]}d_i\cdot f_i(v) + \sum_{i\in[M']}d'_i \cdot h_i(v) \text{ is sum of squares.} \\
&\displaystyle d_i\geq 0 \qquad \forall i \in[M]\\
&\displaystyle d_i\in\R \qquad \forall i \in[M']
\end{array}
\end{equation*}}
\end{minipage}} \caption{Constructing a witness for infeasibility of a set of polynomial Inequalities.}\label{fig:construct witness sos}
\end{figure}
}

In Theorem~\ref{thm:monotone hamiltonian high}, we establish the monotonicity of the \projhamnospace. Our proof is based on the solution to the degree-8 SOS program concerning polynomials in $27$ variables.

\vspace{-.1in}
\begin{theorem}\label{thm:monotone hamiltonian high}
Let $\Z \subseteq \R^n$ be a closed convex set and $F:\Z \rightarrow \R^n$ be a monotone and $L$-Lipschitz operator. For any step size $\eta \in (0, \frac{1}{L}$) and any $z_k \in \Z$, the EG method update satisfies $r^{tan}_{(F,\Z)}(z_k) \ge r^{tan}_{(F,\Z)}(z_{k+1})$.
\end{theorem}

\begin{prevproof}{Theorem}{thm:monotone hamiltonian high}
Assume towards contradiction that $\Ham(z_{k+1}) > \Ham(z_k)$,
using Lemma~\ref{lem:reduce to cone} there exists numbers $\alpha, \beta_1, \beta_2\in \R$ and vectors $\ba_k$, $ \ba_{k+\half}$, $\ba_{k+1}$, $\bz_k$, $\bz_{k+\half}$, $\bz_{k+1}$, $\bF(\bz_k)$, $\bF(\bz_{k+\half})$, $\bF(\bz_{k+1})\in \R^N$ where $N=n+5$ that satisfy the properties in the statement of Lemma~\ref{lem:reduce to cone} and
\begin{align*}
     0&>\InNorms{\bF(\bz_k)-\frac{\InAngles{F(\bz_k),\ba_k}}{\InNorms{\ba_k}^2}\ba_k}^2 
    -
    \InNorms{
    \bF(\bz_{k+1})-\frac{\InAngles{\bF(\bz_{k+1}),\ba_{k+1}}\ba_{k+1}}{\InNorms{\ba_{k+1}}^2}\mathbbm{1}[\InAngles{\bF(\bz_{k+1}),\ba_{k+1}}\ge 0]}^2\notag\\
    & = 
    \|\bF(\bz_k)\|^2-\|\bF(\bz_{k+1})\|^2- \frac{(\beta_1 \bF(\bz_k)[1] + \beta_2 \bF(\bz_k)[2] + \bF(\bz_k)[3])^2}{\beta_1^2 + \beta_2^2 + 1} + \bF(\bz_{k+1})[1]^2  \mathbbm{1}[\bF(\bz_{k+1})[1]\ge 0],
\end{align*}
where we use the fact that $\ba_k = (\beta_1,\beta_2,1,0,\ldots,0)$ 
and $\ba_{k+1}=(1,0,\ldots,0)$.

We use $\tar$ to denote $$\|\bF(\bz_k)\|^2-\|\bF(\bz_{k+1})\|^2 - \frac{(\beta_1 \bF(\bz_k)[1] + \beta_2 \bF(\bz_k)[2] + \bF(\bz_k)[3])^2}{\beta_1^2 + \beta_2^2 + 1} + \bF(\bz_{k+1})[1]^2  \mathbbm{1}[\bF(\bz_{k+1})[1]\ge 0],$$ and our goal is to show that $\tar$ is non-negative, and thus reach a contradiction.

Our plan is to show that we can obtain a sum of quotients of SOS polynomials by adding non-positive terms to $\tar$,
which implies the non-negativity of $\tar$. We add the non-positive terms in a few steps.

Combining Lemma \ref{lem:reduce to cone} and the fact that $\eta>0$ and $\InParentheses{\eta L}^2 \leq 1$, we derive the following two inequalities:
\begin{align}
    \InAngles{\eta \bF(\bz_{k+1}) - \eta\bF(\bz_k),\bz_k-\bz_{k+1}} &\leq0,\label{eq:mononotone} \\
    \InNorms{\eta \bF(\bz_{k+1}) - \eta \bF(\bz_{k+\half})}^2- \InNorms{\bz_{k+1} - \bz_{k+\half}}^2 &\le 0\label{eq:lipsitz}.
\end{align}

Equipped with these two inequalities, it is clear that \begin{equation}\label{eq:SOS first step}
\eta^2\cdot \tar\geq \eta^2\cdot \tar+ 2\cdot \text{LHS of Inequality}~\eqref{eq:mononotone}+\text{LHS of Inequality}~\eqref{eq:lipsitz}.
\end{equation}
Therefore, it is sufficient to show that the RHS of Inequality~\eqref{eq:SOS first step} is non-negative. 

We take advantage of the sparsity of vectors $\ba_k,\ba_{k+\half},\ba_{k+1}$ by
considering the following partition of $[N]$. We define $P_1 = \{1,2,3\}$ and $P_2 = \{4,5,\cdots,N\}$. For any vector $z \in \R^N$ and $j \in \{1,2\}$, we define $p_j(z) \in \R^n$ to be the vector such that $p_j(z)[i] = z[i]$ for $i \in P_j$ and $p_j(z)[i] = 0$ otherwise.
We divide the RHS of Inequality~\eqref{eq:SOS first step} using the partition to Expression~\eqref{eq:unconstrained term} and Expression~\eqref{eq:constrained term}:
\begin{align}\label{eq:unconstrained term}
    &  \eta^2 \InNorms{p_2(\bF(\bz_k))}^2- \eta^2 \InNorms{p_2(\bF(\bz_{k+1}))}^2 + 2\eta \InAngles{p_2(\bF(\bz_{k+1})) - p_2(\bF(\bz_k)) ,p_2(\bz_k)-p_2(\bz_{k+1})} \notag\\
    & \quad + \eta^2\InNorms{p_2(\bF(\bz_{k+1})) - p_2(\bF(\bz_{k+\half}))}^2- \InNorms{p_2(\bz_{k+1}) - p_2(\bz_{k+\half})}^2,
\end{align}
and
\begin{align}\label{eq:constrained term}
    & \eta^2 \InNorms{p_1(\bF(\bz_k))}^2- \eta^2 \InNorms{p_1(\bF(\bz_{k+1}))}^2 + 2\eta \InAngles{p_1(\bF(\bz_{k+1})) - p_1(\bF(\bz_k)) ,p_1(\bz_k)-p_1(\bz_{k+1})} \notag\\
    & \quad + \eta^2\InNorms{p_1(\bF(\bz_{k+1})) - p_1(\bF(\bz_{k+\half}))}^2- \InNorms{p_1(\bz_{k+1}) - p_1(\bz_{k+\half})}^2\notag\\ 
    &\quad - \frac{\eta^2}{\beta_1^2 + \beta_2^2 + 1}\InParentheses{\beta_1 \bF(\bz_k)[1] + \beta_2 \bF(\bz_k)[2] + \bF(\bz_k)[3]}^2  +\eta^2 \bF(\bz_{k+1})[1]^2  \mathbbm{1}[\bF(\bz_{k+1})[1]\ge 0] .
\end{align}

First we show that the Expression~\eqref{eq:unconstrained term} is non-negative.
According to Lemma \ref{lem:reduce to cone}, $\ba_{k+\half}$ and $\bz_{k+\half} -\bz_k +\eta \bF(\bz_k)$ are co-directed,
and $\ba_{k+1}$ and $\bz_{k+1}-\bz_k + \eta \bF(\bz_{k+\half})$ are co-directed. 
For any $i \in P_2$, we know $\ba_{k+\frac{1}{2}}[i] = \ba_{k+1}[i] = 0$ and thus we have
\begin{align*}
0=\InAngles{e_i, \ba_{k+\half}}
=\InAngles{e_i, \bz_{k+\half}-\bz_k + \eta\bF(\bz_{k+1})}\Leftrightarrow
\bz_{k+\half}[i]=\bz_k[i] - \eta\bF(\bz_{k+1})[i], \\
0=\InAngles{e_i, \ba_{k+1}}
=\InAngles{e_i, \bz_{k+1}-\bz_k + \eta\bF(\bz_{k+1})}\Leftrightarrow
\bz_{k+1}[i]=\bz_k[i] - \eta\bF(\bz_{k+1})[i],
\end{align*}
which implies that
\begin{align*}
    p_2(\bz_{k+\half}) &= p_2(\bz_k) -\eta \cdot p_2(\bF(\bz_k)), \\
    p_2(\bz_{k+1}) &= p_2(\bz_k) - \eta \cdot p_2(\bF(\bz_{k+\half})).
\end{align*}
Intuitively, one can think of the coordinates in $P_2$ as the ones where the EG update is unconstrained.
With the two new equalities, it is easy to verify that Expression~\eqref{eq:unconstrained term} is always $0$.
\begin{align*}
    &\eta^2 \InNorms{p_2(\bF(\bz_k))}^2- \eta^2 \InNorms{p_2(\bF(\bz_{k+1}))}^2 + 2\eta \InAngles{p_2(\bF(\bz_{k+1})) - p_2(\bF(\bz_k)) ,p_2(\bz_k)-p_2(\bz_{k+1})} \\
    &\quad + \eta^2\InNorms{p_2(\bF(\bz_{k+1})) - p_2(\bF(\bz_{k+\half}))}^2- \InNorms{p_2(\bz_{k+1}) - p_2(\bz_{k+\half})}^2 \\
     =& \eta^2 \InNorms{p_2(\bF(\bz_k))}^2- \eta^2 \InNorms{p_2(\bF(\bz_{k+1}))}^2 +2\eta^2 \InAngles{p_2(\bF(\bz_{k+1})) - p_2(\bF(\bz_k)) ,p_2(\bF(\bz_{k+\half}))} \\
    &\quad + \eta^2\InNorms{p_2(\bF(\bz_{k+1})) - p_2(\bF(\bz_{k+\half}))}^2- \eta^2\InNorms{p_2(\bF(\bz_k) - p_2(\bF(\bz_{k+\half}))}^2 \\
     =& 0.
\end{align*}

We now turn our attention to Expression~$\eqref{eq:constrained term}$ and show that it is non-negative. The analysis is more challenging here. We  introduce the following six non-positive expressions, multiply each of them with a carefully chosen coefficient, then add them together with 
Expression~\eqref{eq:constrained term}. We finally verify that the sum is a sum of quotients of SOS polynomials implying the non-negativity of Expression~\eqref{eq:constrained term}. We believe it will be extremely challenging if not impossible for human beings to discover these non-positive expressions and their associated coefficients manually to complete this proof. We instead harness the power of the SOS programming to overcome the difficulty and make the discovery.

We first present the six non-positive expressions.
\begin{align}
\bz_{k+\half}[1]\InParentheses{\InParentheses{\bz_{k+\half} -\bz_k +\eta \bF(\bz_k)}[1] - \alpha \InParentheses{\bz_{k+\half} -\bz_k +\eta \bF(\bz_k)}[2]} &= 0,\label{eq:cons1}\\
\bz_{k+1}[2]\InParentheses{\InParentheses{\bz_{k+\half} -\bz_k +\eta \bF(\bz_k)}[1] - \alpha \InParentheses{\bz_{k+\half} -\bz_k +\eta \bF(\bz_k)}[2]} &= 0,\label{eq:cons2} \\
\InParentheses{\alpha\InParentheses{\bz_{k}-\eta\bF(\bz_{k})}[1] + \InParentheses{\bz_{k}-\eta\bF(\bz_{k})}[2]} \InParentheses{\alpha\bz_{k+1}[1]+\bz_{k+1}[2]} &\le 0,\label{eq:cons3} \\
-\eta\InParentheses{\beta_1\bF(\bz_k)[1] + \beta_2\bF(\bz_k)[2] + \bF(\bz_k)[3]} \InParentheses{\beta_1\bz_{k+\half}[1] + \beta_2 \bz_{k+\half}[2] + \bz_{k+\half}[3]} &\le 0,\label{eq:cons4} \\
\bz_k[1]\InParentheses{\bz_k[1] - \eta \bF(\bz_{k+\half})[1]} &\le 0,\label{eq:cons5} \\
-\eta\bF(\bz_{k+1})[1] \mathbbm{1}\inteval{\bF(\bz_{k+1})[1]\le 0 }\InParentheses{\bz_k[1] - \eta \bF(\bz_{k+\half})[1]} &\le 0.\label{eq:cons6} 
\end{align}

Equation~\eqref{eq:cons1} and Equation~\eqref{eq:cons2} follow from the combination of the fact that $\InAngles{(1,-\alpha,0,\cdots,0), \ba_{k+\half}} = 0$ and that $\ba_{k+\half}$ and  $\bz_{k+\half} -\bz_k +\eta \bF(\bz_k)$ are co-directed:
\begin{align*}
    &\InAngles{(1,-\alpha,0,\cdots,0) ,\bz_{k+\half} -\bz_k +\eta \bF(\bz_k)}  = 0\\
\Leftrightarrow & \InParentheses{\bz_{k+\half} -\bz_k +\eta \bF(\bz_k)}[1] - \alpha \InParentheses{\bz_{k+\half} -\bz_k +\eta \bF(\bz_k)}[2] = 0.
\end{align*}

Note that the LHS of Inequality~\eqref{eq:cons3} is equal to $\InAngles{\ba_{k+\half},\bz_k - \eta\bF(\bz_k)}\cdot \InAngles{\ba_{k+\half},\bz_{k+1}}$.
 Lemma \ref{lem:reduce to cone} guarantees that
$\InAngles{\ba_{k+\half},\bz_{k+1}} \ge 0$. Since $\ba_{k+\half}$ and $\bz_k -\eta \bF(\bz_k) -\bz_{k+\half}$ are oppositely directed, and $\InAngles{\ba_{k+\half},\bz_{k+\half}} = 0$, we have that \begin{align*}
0 \geq \InAngles{\ba_{k+\half},\bz_{k} - \eta \bF(\bz_k)- \bz_{k+\half}} =\InAngles{\ba_{k+\half},\bz_{k} - \eta \bF(\bz_k)}.
\end{align*}
Hence, $\InAngles{\ba_{k+\half},\bz_k - \eta\bF(\bz_k)}\cdot \InAngles{\ba_{k+\half},\bz_{k+1}}\leq 0$.

\notshow{
\begin{align*}
    \InParentheses{\alpha\InParentheses{\bz_{k}}-\eta\bF(\bz_{k})}[1] + \InParentheses{\alpha\InParentheses{\bz_{k}}-\eta\bF(\bz_{k})}[2] = \InAngles{\ba_{k+\half},\bz_{k} - \eta \bF(\bz_k)} = \InAngles{\ba_{k+\half},\bz_{k} - \eta \bF(\bz_k)- \bz_{k+\half}} &\le 0,\\
    \alpha \bz_{k+1}[1] + \bz_{k+1}[2] = \InAngles{\ba_{k+\half},\bz_{k+1}} &\ge 0.
\end{align*}
The third non-positive Expression is 

\begin{align}\label{eq:cons3}
    \InParentheses{\InParentheses{\alpha\InParentheses{\bz_{k}}-\eta\bF(\bz_{k})}[1] + \InParentheses{\alpha\InParentheses{\bz_{k}}-\eta\bF(\bz_{k})}[2]} \InParentheses{\alpha\bz_{k+1}[1]+\bz_{k+1}[2]} \le 0.
\end{align}
}

Observe that the LHS of Inequality~\eqref{eq:cons4} is equal to $-\eta\InAngles{\ba_k,\bF(\bz_k)}\cdot \InAngles{\ba_k,\bz_{k+\half}}$.
Lemma \ref{lem:reduce to cone} guarantees that $\InAngles{\ba_k,\bF(\bz_k)} \ge 0$ and $\InAngles{\ba_k,\bz_{k+\half}} \ge 0$. 

\notshow{The fourth non-positive Expression is 
\begin{align}\label{eq:cons4}
    &-\InAngles{\ba_k,\bF(\ba_k)}\InAngles{\ba_k,\bF(\ba_{k+\half})}\notag\\
    &= -\InParentheses{\beta_1\bF(\bz_k)[1] + \beta_2\bF(\bz_k)[2] + \bF(\bz_k)[3]} \InParentheses{\beta_1\bz_{k+\half}[1] + \beta_2 \bz_{k+\half}[2] + \bz_{k+\half}[3]} \le 0.
\end{align}
}

Finally, we argue Inequality~\eqref{eq:cons5} and Inequality~\eqref{eq:cons6}.
Note that the LHS of Inequality~\eqref{eq:cons5} is equal to $\InAngles{\ba_{k+1},\bz_k}\cdot \InAngles{\ba_{k+1},\bz_k - \eta \bF(\bz_{k+\half})}$. 
Since $\ba_{k+1}$  and $\bz_k - \eta \bF(\bz_{k+\half})-\bz_{k+1}$ are oppositely directed, and $\InAngles{\ba_{k+1},\bz_{k+1}} = 0$, we have that
\begin{align*}
0 \geq \InAngles{\ba_{k+1},\bz_k - \eta \bF(\bz_{k+\half})-\bz_{k+1}}   = \InAngles{\ba_{k+1},\bz_k - \eta \bF(\bz_{k+\half})}.
\end{align*}
Clearly, $-\bF(\bz_{k+1})\ind[\bF(\bz_{k+1})[1]\leq 0]$ is non-negative 
and $\bz_k[1]=\InAngles{\ba_{k+1},\bz_k}$ is also non-negative due to Lemma~\ref{lem:reduce to cone}. Thus Inequality (\ref{eq:cons5}) and Inequality (\ref{eq:cons6}) hold. 

\notshow{
The fifth non-positive Expression is 
\begin{align}\label{eq:cons5-1}
    \InParentheses{\bz_k[1]+\bF(\bz_{k+1})[1] \mathbbm{1}\inteval{\bF(\bz_{k+1})[1]\ge 0} }\InParentheses{\bz_k[1] - \eta \bF(\bz_{k+\half})[1]} \le 0.
\end{align}
Let us first consider the case when $\bF(\bz_{k+1})[1] \le 0$. 
\paragraph{Case 1: $\bF(\bz_{k+1})[1] \ge 0$. } Then (\ref{eq:general-final-potential}) can be written as follows:
\begin{align}
    & \sum_{i=1}^3 \left( \eta^2 \bF(\bz_k)[i]^2- \eta^2 \bF(\bz_{k+1})[i]^2 + 2\eta \InParentheses{ \bF(\bz_{k+1})[i] - \bF(\bz_k)[i]}\InParentheses{ \bz_k[i]-\bz_{k+1}[i]}\right. \notag\\
    &\left. \quad + \eta^2 \InParentheses{ \bF(\bz_{k+1}))[i] - \bF(\bz_{k+\half})[i]}^2- \InParentheses{\bz_{k+1}[i] - \bz_{k+\half}[i]}^2\right)\notag\\ 
    &\quad - \frac{\eta^2}{\beta_1^2 + \beta_2^2 + 1}\InParentheses{\beta_1 \bF(\bz_k)[1] + \beta_2 \bF(\bz_k)[2] + \bF(\bz_k)[3]}^2  +\eta^2 \bF(\bz_{k+1})[1]^2 .\label{eq:general-final-potential-1}
\end{align}}
\notshow{
We have the following identity.
\begin{align*}
    &\text{Expression}~\eqref{eq:constrained term} + 2\times \InParentheses{\LHSE~\eqref{eq:cons1}+ \LHSI~\eqref{eq:cons5}}+ \frac{2\alpha}{1+\alpha^2} \times \LHSE~\eqref{eq:cons2} \\
    &\qquad\qquad\qquad\qquad\qquad\quad+ \frac{2}{1+\alpha^2} \times \LHSI~\eqref{eq:cons3} + \frac{2}{1+\beta_1^2+\beta_2^2} \times \LHSI~\eqref{eq:cons4} \\
     =&  \InParentheses{\bz_k[1]-\bF(\bz_{k+\half})[1] + \bF(\bz_{k+1})[1] \mathbbm{1}\inteval{\bF(\bz_{k+1})[1]\geq 0}}^2 + \frac{1}{1+\beta_2^2} \InParentheses{\bz_k[2] - \bF(\bz_{k})[2] + \alpha \bz_{k}[1]}^2 \\
     & \quad + \frac{1}{1+\beta_1^2+\beta_2^2} \left(\InParentheses{\bF(\bz_k)[3]^2 + \beta_1\bz_k[1]^2 + \beta_2 \bz_{k}[2] + \InParentheses{\alpha\beta_2-\beta_1}\bz_{k+\half}[1]}^2 \right.\\&\left.\quad  + \frac{1}{1+\beta_2^2} \InParentheses{\InParentheses{1+\beta_2^2}\InParentheses{\bF(\bz_k)[1]- \bz_k[1]} -\beta_1\beta_2 \InParentheses{\bF(\bz_k[2]) - \bz_{k}[2]} + \InParentheses{1+\beta_2^2 + \alpha\beta_1\beta_2}\bz_{k+\half}[2]}^2  \right)\\
     &\ge 0.
\end{align*}}


Our next step is to show that the following is non-negative.
\begin{align}\label{eq:LHS identity pre substitution}
    &\text{Expression}~\eqref{eq:constrained term} + 2\times \InParentheses{\LHSE~\eqref{eq:cons1}+ \LHSI~\eqref{eq:cons5}+\LHSI~\eqref{eq:cons6}}\notag \\
    &+ \frac{2\alpha}{1+\alpha^2} \times \LHSE~\eqref{eq:cons2}+ \frac{2}{1+\alpha^2} \times \LHSI~\eqref{eq:cons3} + \frac{2}{1+\beta_1^2+\beta_2^2} \times \LHSI~\eqref{eq:cons4}
\end{align}

We first simplify Expression~\eqref{eq:LHS identity pre substitution}, using the following relationship between the variables.
\begin{align}
    \bz_{k}[3] &= -\beta_1 \bz_k[1] - \beta_2 \bz_k[2], \label{substitute-1}\\
    \bz_{k+\half}[2] &= -\alpha \bz_{k+\half}[1],\label{substitute-2}\\
    \bz_{k+\half}[3] &= \bz_{k}[3] - \eta \bF(\bz_k)[3],\label{substitute-3}\\
    \bz_{k+1}[1] &= 0,\label{substitute-4}\\
    \bz_{k+1}[2] &= \bz_{k}[2] - \eta \bF(\bz_{k+\half})[2],\label{substitute-5}\\
    \bz_{k+1}[3] &= \bz_{k}[3] - \eta \bF(\bz_{k+\half})[3]\label{substitute-6}.
\end{align}
Equation (\ref{substitute-1}), Equation (\ref{substitute-2}), and Equation (\ref{substitute-4}) follows by $\InAngles{\ba_i,\bz_i} = 0$ for $i \in \{k,k+\half,k+1\}$ due to Lemma \ref{lem:reduce to cone}. We know that (i) $\InAngles{\ba_{k+\half},e_3} = \InAngles{\ba_{k+1},e_2} = \InAngles{\ba_{k+1},e_3} = 0$ by the definition of $\ba_{k+\half}$ and $\ba_{k+1}$, and (ii) $\ba_{k+\half}$ and $\bz_{k+\half}-\bz_k + \eta \bF(\bz_k)$ are co-directed, $\ba_{k+1}$ and $\bz_{k+1}-\bz_k + \eta \bF(\bz_{k+\half})$ are co-directed by Lemma \ref{lem:reduce to cone}, thus Equation (\ref{substitute-3}), Equation (\ref{substitute-5}), and Equation (\ref{substitute-6}) follow from the combination of (i) and (ii). 

We simplify Expression~\eqref{eq:LHS identity pre substitution} by substituting $\bz_k[3]$, $\bz_{k+\half}[2]$, $\bz_{k+\half}[3]$, $\bz_{k+1}[1]$, $\bz_{k+1}[2]$, and $\bz_{k+1}[3]$ using Equations~\eqref{substitute-1}-\eqref{substitute-6}. 

Expression (\ref{eq:constrained term}) is equal to the sum of the following three parts. 

The first part is 
\begin{align}
    & \eta^2 \InNorms{p_1(\bF(\bz_k))}^2- \eta^2 \InNorms{p_1(\bF(\bz_{k+1}))}^2 - \frac{\eta^2}{\beta_1^2 + \beta_2^2 + 1}\InParentheses{\beta_1 \bF(\bz_k)[1] + \beta_2 \bF(\bz_k)[2] + \bF(\bz_k)[3]}^2 \notag \\
    &\quad   +\eta^2 \bF(\bz_{k+1})[1]^2  \mathbbm{1}[\bF(\bz_{k+1})[1]\ge 0] \notag\\
    & = \sum_{i=1}^3 \InParentheses{\eta^2 \bF(\bz_k)[i]^2 - \eta^2 \bF(\bz_{k+1})[i]^2} - \frac{\eta^2}{\beta_1^2 + \beta_2^2 + 1}\InParentheses{\beta_1 \bF(\bz_k)[1] + \beta_2 \bF(\bz_k)[2] + \bF(\bz_k)[3]}^2 \notag \\
    &\quad + \eta^2 \bF(\bz_{k+1})[1]^2  \mathbbm{1}[\bF(\bz_{k+1})[1]\ge 0].\label{eq:final-cons1}
\end{align}

The second part is 
\begin{align}
    &  2\eta \InAngles{p_1(\bF(\bz_{k+1})) - p_1(\bF(\bz_k)) ,p_1(\bz_k)-p_1(\bz_{k+1})} \notag\\
    & = 2\eta\bz_k[1]\InParentheses{\bF(\bz_{k+1})[1] - \bF(\bz_{k})[1]} + 2\eta^2 \bF(\bz_{k+\half})[2]\InParentheses{\bF(\bz_{k+1})[2] - \bF(\bz_{k})[2]}  \notag \\
    & \quad +  2\eta^2 \bF(\bz_{k+\half})[3]\InParentheses{\bF(\bz_{k+1})[3] - \bF(\bz_{k})[3]}.\label{eq:final-cons2}
\end{align}

The third part is 
\begin{align}
    & \eta^2\InNorms{p_1(\bF(\bz_{k+1})) - p_1(\bF(\bz_{k+\half}))}^2- \InNorms{p_1(\bz_{k+1}) - p_1(\bz_{k+\half})}^2\notag\\ 
    &=  \sum_{i=1}^3 \eta^2 \InParentheses{\bF(\bz_{k+1})[i] -  \bF(\bz_{k+\half})[i]}^2 - \bz_{k+\half}[1]^2 - \InParentheses{\bz_{k}[2] - \eta \bF(\bz_{k+\half})[2] +\alpha \bz_{k+\half}[1]}^2 \notag\\
    &\quad -  \InParentheses{\eta \bF(\bz_k)[3] - \eta \bF(\bz_{k+\half})[3]}^2. \label{eq:final-cons3}
\end{align}
\notshow{
\begin{align}\label{eq:sub-constrained term}
    & \eta^2 \InNorms{p_1(\bF(\bz_k))}^2- \eta^2 \InNorms{p_1(\bF(\bz_{k+1}))}^2 + 2\eta \InAngles{p_1(\bF(\bz_{k+1})) - p_1(\bF(\bz_k)) ,p_1(\bz_k)-p_1(\bz_{k+1})} \notag\\
    & \quad + \eta^2\InNorms{p_1(\bF(\bz_{k+1})) - p_1(\bF(\bz_{k+\half}))}^2- \InNorms{p_1(\bz_{k+1}) - p_1(\bz_{k+\half})}^2\notag\\ 
    &\quad - \frac{\eta^2}{\beta_1^2 + \beta_2^2 + 1}\InParentheses{\beta_1 \bF(\bz_k)[1] + \beta_2 \bF(\bz_k)[2] + \bF(\bz_k)[3]}^2  +\eta^2 \bF(\bz_{k+1})[1]^2  \mathbbm{1}[\bF(\bz_{k+1})[1]\ge 0]\\
    & = \sum_{i=1}^3 \InParentheses{\eta^2 \bF(\bz_k)[i]^2 - \eta^2 \bF(\bz_{k+1})[i]^2} + \bz_k[1]\InParentheses{\bF(\bz_{k+1})[1] - \bF(\bz_{k})[1]} \\
    & \quad + \eta \bF(\bz_{k+\half})[2]\InParentheses{\bF(\bz_{k+1})[2] - \bF(\bz_{k})[2]} +  \eta \bF(\bz_{k+\half})[3]\InParentheses{\bF(\bz_{k+1})[3] - \bF(\bz_{k})[3]} \\
    & \quad + \sum_{i=1}^3 \eta^2 \InParentheses{\bF(\bz_{k+1})[i] -  \bF(\bz_{k+\half})[i]}^2 - \bz_{k+\half}[1]^2 - \InParentheses{\bz_{k}[2] - \eta \bF(\bz_{k+\half})[2] +\alpha \bz_{k+\half}[1]}^2\\
    &\quad +  \InParentheses{\eta \bF(\bz_k)[3] - \eta \bF(\bz_{k+\half})[3]}^2 - \frac{\eta^2}{\beta_1^2 + \beta_2^2 + 1}\InParentheses{\beta_1 \bF(\bz_k)[1] + \beta_2 \bF(\bz_k)[2] + \bF(\bz_k)[3]}^2 \\
    &\quad + \eta^2 \bF(\bz_{k+1})[1]^2  \mathbbm{1}[\bF(\bz_{k+1})[1]\ge 0].
\end{align}}

$2\times$ LHS of Equation (\ref{eq:cons1}) is equal to 
\begin{align}
    &2\bz_{k+\half}[1]\InParentheses{(\bz_{k+\half} -\bz_k +\eta \bF(\bz_k))[1] - \alpha (\bz_{k+\half} -\bz_k +\eta \bF(\bz_k))[2]} \notag\\
    & = 2\bz_{k+\half}[1]\InParentheses{(\bz_{k+\half} -\bz_k +\eta \bF(\bz_k))[1] - \alpha (-\alpha \bz_{k+\half}[1] -\bz_k[2] +\eta \bF(\bz_k)[2])} \label{eq:final-cons4}.
\end{align}

$\frac{2\alpha}{1+\alpha^2} \times$ LHS of Equation (\ref{eq:cons2}) is equal to 
\begin{align}
    &\frac{2\alpha}{1+\alpha^2}\bz_{k+1}[2]\InParentheses{(\bz_{k+\half} -\bz_k +\eta \bF(\bz_k))[1] - \alpha (\bz_{k+\half} -\bz_k +\eta \bF(\bz_k))[2]} \notag\\
    & = \frac{2\alpha}{1+\alpha^2}\InParentheses{\bz_{k}[2] - \eta \bF(\bz_{k+\half})[2]}\InParentheses{(\bz_{k+\half} -\bz_k +\eta \bF(\bz_k))[1] - \alpha \InParentheses{-\alpha \bz_{k+\half}[1] -\bz_k[2] +\eta \bF(\bz_k)[2]}} \label{eq:final-cons5}.
\end{align}

$\frac{2}{1+\alpha^2} \times$ LHS of Inequality (\ref{eq:cons3}) is equal to 
\begin{align}
    &\frac{2}{1+\alpha^2}\InParentheses{\alpha\InParentheses{\bz_{k}-\eta\bF(\bz_{k})}[1] + \InParentheses{\bz_{k}-\eta\bF(\bz_{k})}[2]} \InParentheses{\alpha\bz_{k+1}[1]+\bz_{k+1}[2]}\notag\\
    & = \frac{2}{1+\alpha^2} \InParentheses{\alpha\InParentheses{\bz_{k}-\eta\bF(\bz_{k})}[1] + \InParentheses{\bz_{k}-\eta\bF(\bz_{k})}[2]}\InParentheses{\bz_{k}[2] - \eta \bF(\bz_{k+\half})[2]} \label{eq:final-cons6}.
\end{align}

$\frac{2}{1+\beta_1^2+\beta_2^2} \times$ LHS of Inequality (\ref{eq:cons4}) is equal to 
\begin{align}
    &-\frac{2}{1+\beta_1^2+\beta_2^2}\InParentheses{\beta_1\eta\bF(\bz_k)[1] + \beta_2\eta\bF(\bz_k)[2] + \eta\bF(\bz_k)[3]} \InParentheses{\beta_1\bz_{k+\half}[1] + \beta_2 \bz_{k+\half}[2] + \bz_{k+\half}[3]}\notag \\ 
    & = -\frac{2}{1+\beta_1^2+\beta_2^2}\InParentheses{\beta_1\eta\bF(\bz_k)[1] + \beta_2\eta\bF(\bz_k)[2] + \eta\bF(\bz_k)[3]} \InParentheses{\InParentheses{\beta_1-\alpha\beta_2}\bz_{k+\half}[1] + \bz_{k}[3] - \eta \bF(\bz_k)[3] } \notag \\
    & = -\frac{2}{1+\beta_1^2+\beta_2^2}\InParentheses{\eta\beta_1\bF(\bz_k)[1] + \eta\beta_2\bF(\bz_k)[2] + \eta\bF(\bz_k)[3]} \notag\\ &\qquad\qquad\cdot\InParentheses{\InParentheses{\beta_1-\alpha\beta_2}\bz_{k+\half}[1] -\beta_1 \bz_{k}[1]-\beta_2\bz_{k}[2] - \eta \bF(\bz_k)[3] }\label{eq:final-cons7}.
\end{align}

$2\times $ LHS of Inequality (\ref{eq:cons5}) is equal to 
\begin{align}
    2\bz_k[1]\InParentheses{\bz_k[1] - \eta \bF(\bz_{k+\half})[1]}\label{eq:final-cons8}.
\end{align}

$2\times $ LHS of Inequality (\ref{eq:cons6}) is equal to 
\begin{align}
   -2\eta\bF(\bz_{k+1})[1] \mathbbm{1}\inteval{\bF(\bz_{k+1})[1]\le 0 }\InParentheses{\bz_k[1] - \eta \bF(\bz_{k+\half})[1]}\label{eq:final-cons9}.
\end{align}

After the substitution, we need to argue that the sum of Expression~\eqref{eq:final-cons1} to~\eqref{eq:final-cons9} is 
a sum of quotients of SOS polynomials, which we prove by establishing the following identity.

\begin{align}
    &\text{Expression}~\eqref{eq:final-cons1} + \text{Expression}~\eqref{eq:final-cons2} + \text{Expression}~\eqref{eq:final-cons3} + \text{Expression}~\eqref{eq:final-cons4} + \text{Expression}~\eqref{eq:final-cons5}\notag\\
    &\qquad\qquad\qquad + \text{Expression}~\eqref{eq:final-cons6} + \text{Expression}~\eqref{eq:final-cons7} + \text{Expression}~\eqref{eq:final-cons8} + \text{Expression}~\eqref{eq:final-cons9} \notag\\
     &=\InParentheses{\bz_k[1]-\eta\bF(\bz_{k+\half})[1] + \eta\bF(\bz_{k+1})[1] \cdot\ind[\bF(\bz_{k+1})[1]\geq 0]}^2 \label{eq:sos-1}\\
     & + \frac{\left(\bz_k[1] - \eta\bF(\bz_k)[1] -\bz_{k+\half}[1]\right)^2}{1+\beta_1^2+\beta_2^2}  \label{eq:sos-2}\\
     &  + \frac{\left( \eta\bF(\bz_k)[3] + \beta_1\bz_k[1] + \beta_2 \bz_{k}[2] + \InParentheses{\alpha\beta_2-\beta_1}\bz_{k+\half}[1]\right)^2}{1+\beta_1^2+\beta_2^2}  \label{eq:sos-3} \\
     & + \frac{\left(\bz_k[2] -\eta\bF(\bz_{k})[2] + \alpha \bz_{k+\half}[1]\right)^2}{1+\beta_1^2+\beta_2^2}  \label{eq:sos-4}\\
     & + \frac{\left(\beta_1\InParentheses{\bz_k[2] -\eta\bF(\bz_{k})[2] + \alpha \bz_{k+\half}[1]} - \beta_2 \InParentheses{\bz_k[1] - \eta\bF(\bz_k)[1] -\bz_{k+\half}[1]}\right)^2}{1+\beta_1^2+\beta_2^2}  \label{eq:sos-5}\\
     \ge & 0. \notag
\end{align}


The identity for the case where $\bF(\bz_{k+1})[1]\geq 0$ is verified at Appendix~F in the second version of this paper on arXiv, which can be found at this \href{https://arxiv.org/abs/2204.09228v2}{link}.
Observe that only Expression~\eqref{eq:final-cons1}, Expression~\eqref{eq:final-cons9} and Term~\eqref{eq:sos-1} depend on the sign of $\bF(\bz_{k+1})[1]$.
It is sufficient for us to verify the case where $\bF(\bz_{k+1})[1]\geq 0$, as when $\bF(\bz_{k+1})[1]<0$, we only need to subtract $\eta^2 \bF(\bz_{k+1})[1]^2 + 2\eta \bF(\bz_{k+1})[1] \InParentheses{\bz_{k}[1]-\eta \bF(\bz_{k+\half})[1]}$ from both the LHS and the RHS,\footnote{Notice that $\InParentheses{\bz_k[1]-\eta\bF(\bz_{k+\half})[1] +\eta \bF(\bz_{k+1})[1]}^2 = \InParentheses{\bz_k[1]-\eta\bF(\bz_{k+\half})[1]}^2+\eta^2 \bF(\bz_{k+1})[1]^2 + 2 \eta\bF(\bz_{k+1})[1] \InParentheses{\bz_{k}[1]-\eta \bF(\bz_{k+\half})[1]}$.} and the identity still holds.

Hence, Expression~\eqref{eq:constrained term} is non-negative. Combining with the non-negativity of Expression~\eqref{eq:unconstrained term}, we conclude that \tar~is non-negative. This completes the proof.
\end{prevproof}

\notshow{
\noindent{\bf Proof Sketch }   
Assume towards contradiction that $\Ham(z_k) < \Ham(z_{k+1})$,
using Lemma~\ref{lem:reduce to cone} there exist numbers $\alpha, \beta_1, \beta_2\in \R$ and vectors $\ba_k$, $ \ba_{k+\half}$, $\ba_{k+1}$, $\bz_k$, $\bz_{k+\half}$, $\bz_{k+1}$, $\bF(\bz_k)$, $\bF(\bz_{k+\half})$, $\bF(\bz_{k+1})\in \R^N$ where $N=n+5$ that satisfy the properties in the statement of Lemma~\ref{lem:reduce to cone} and
\begin{align*}
     0&>\InNorms{\bF(\bz_k)-\frac{\InAngles{F(\bz_k),\ba_k}}{\InNorms{\ba_k}^2}\ba_k}^2 
    - \InNorms{ \bF(\bz_{k+1})-\frac{\InAngles{\bF(\bz_{k+1}),\ba_{k+1}}\ba_{k+1}}{\InNorms{\ba_{k+1}}^2}\mathbbm{1}[\InAngles{\bF(\bz_{k+1}),\ba_{k+1}}\ge 0]}^2.
\end{align*}

We use $\tar$ to denote the RHS above.
We first present 8 inequalities/equations, which directly follow from Lemma \ref{lem:reduce to cone}. We verify their correctness in the complete proof in Appendix~\ref{appx:last iterate EG}.

\vspace{-0.2in}
\begin{align}
\InAngles{\eta \bF(\bz_{k+1}) - \eta\bF(\bz_k),\bz_k-\bz_{k+1}} &\leq0,\label{eq:sketch_mononotone} \\
\InNorms{\eta \bF(\bz_{k+1}) - \eta \bF(\bz_{k+\half})}^2- \InNorms{\bz_{k+1} - \bz_{k+\half}}^2 &\le 0,\label{eq:sketch_lipsitz}\\
\bz_{k+\half}[1]\InParentheses{\InParentheses{\bz_{k+\half} -\bz_k +\eta \bF(\bz_k)}[1] - \alpha \InParentheses{\bz_{k+\half} -\bz_k +\eta \bF(\bz_k)}[2]} &= 0,\label{eq:sketch_cons1}\\
\bz_{k+1}[2]\InParentheses{\InParentheses{\bz_{k+\half} -\bz_k +\eta \bF(\bz_k)}[1] - \alpha \InParentheses{\bz_{k+\half} -\bz_k +\eta \bF(\bz_k)}[2]} &= 0,\label{eq:sketch_cons2}\\
\InParentheses{\alpha\InParentheses{\bz_{k}-\eta\bF(\bz_{k})}[1] + \InParentheses{\bz_{k}-\eta\bF(\bz_{k})}[2]} \InParentheses{\alpha\bz_{k+1}[1]+\bz_{k+1}[2]} &\le 0,\label{eq:sketch_cons3} \\
-\eta\InParentheses{\beta_1\bF(\bz_k)[1] + \beta_2\bF(\bz_k)[2] + \bF(\bz_k)[3]} \InParentheses{\beta_1\bz_{k+\half}[1] + \beta_2 \bz_{k+\half}[2] + \bz_{k+\half}[3]} &\le 0,\label{eq:sketch_cons4} \\
\bz_k[1]\InParentheses{\bz_k[1] - \eta \bF(\bz_{k+\half})[1]} &\le 0,\label{eq:sketch_cons5} \\
-\eta\bF(\bz_{k+1})[1] \mathbbm{1}\inteval{\bF(\bz_{k+1})[1]\le 0 }\InParentheses{\bz_k[1] - \eta \bF(\bz_{k+\half})[1]} &\le 0.\label{eq:sketch_cons6} 
\end{align}

We first add several non-positive terms to $\tar$ to derive Expression~\ref{eq:proof sketch Hamiltonian monotonicity LHS}.
\begin{align}\label{eq:proof sketch Hamiltonian monotonicity LHS}
    &\eta^2\cdot \tar+ 2\cdot \LHSI~\eqref{eq:sketch_mononotone}+\LHSI~\eqref{eq:sketch_lipsitz}+2\cdot \LHSE~\eqref{eq:sketch_cons1}\notag\\
    &+ \frac{2\alpha}{1+\alpha^2} \cdot \LHSE~\eqref{eq:sketch_cons2}  \quad+ \frac{2}{1+\alpha^2} \cdot\LHSI~\eqref{eq:sketch_cons3}+ \frac{2}{1+\beta_1^2+\beta_2^2} \cdot \LHSI~\eqref{eq:sketch_cons4}  \notag \\
    &+ 2\cdot \LHSI~\eqref{eq:sketch_cons5}+ 2\cdot \LHSI~\eqref{eq:sketch_cons6}
    \end{align}
After substituting the following six variables  $\bz_{k}[3], \bz_{k+\half}[2], \bz_{k+\half}[3],\bz_{k+1}[1],\bz_{k+1}[2], \bz_{k+1}[3]$  using Equation~\eqref{substitute-1} to~\eqref{substitute-6}, Expression~\eqref{eq:proof sketch Hamiltonian monotonicity LHS} equals to the following polynomial, which is clearly non-negative. We verify the validity of the substitutions (Equation~\eqref{substitute-1} to~\eqref{substitute-6}) in the full proof.
 \notshow{ 
    \begin{align*}
    & \InParentheses{\bz_k[1]-\eta\bF(\bz_{k+\half})[1] + \eta\bF(\bz_{k+1})[1] \cdot\ind[\bF(\bz_{k+1})[1]\geq 0]}^2\\
    & + \frac{\left(\bz_k[2] -\eta\bF(\bz_{k})[2] + \alpha \bz_{k+\half}[1]\right)^2}{1+\beta_2^2}\\
    &  + \frac{\left( \eta\bF(\bz_k)[3] + \beta_1\bz_k[1] + \beta_2 \bz_{k}[2] + \InParentheses{\alpha\beta_2-\beta_1}\bz_{k+\half}[1]\right)^2}{1+\beta_1^2+\beta_2^2}  \\
    & + \frac{\left(\InParentheses{1+\beta_2^2}\InParentheses{\eta\bF(\bz_k)[1]- \bz_k[1]} -\beta_1\beta_2 \InParentheses{\eta\bF(\bz_k)[2] - \bz_{k}[2]} + \InParentheses{1+\beta_2^2 + \alpha\beta_1\beta_2}\bz_{k+\half}[1]\right)^2 }{\InParentheses{1+\beta_2^2}\InParentheses{1+\beta_1^2+\beta_2^2}}
\end{align*}
}
\begin{align*}
    &\InParentheses{\bz_k[1]-\eta\bF(\bz_{k+\half})[1] + \eta\bF(\bz_{k+1})[1] \cdot\ind[\bF(\bz_{k+1})[1]\geq 0]}^2 \\
     & + \frac{\left(\bz_k[1] - \eta\bF(\bz_k)[1] -\bz_{k+\half}[1]\right)^2}{1+\beta_1^2+\beta_2^2} \\
     &  + \frac{\left( \eta\bF(\bz_k)[3] + \beta_1\bz_k[1] + \beta_2 \bz_{k}[2] + \InParentheses{\alpha\beta_2-\beta_1}\bz_{k+\half}[1]\right)^2}{1+\beta_1^2+\beta_2^2}  \\
     & + \frac{\left(\bz_k[2] -\eta\bF(\bz_{k})[2] + \alpha \bz_{k+\half}[1]\right)^2}{1+\beta_1^2+\beta_2^2} \\
     & + \frac{\left(\beta_1\InParentheses{\bz_k[2] -\eta\bF(\bz_{k})[2] + \alpha \bz_{k+\half}[1]} - \beta_2 \InParentheses{\bz_k[1] - \eta\bF(\bz_k)[1] -\bz_{k+\half}[1]}\right)^2}{1+\beta_1^2+\beta_2^2}
\end{align*}
$\hfill \blacksquare$
}

\end{document}